\documentclass[reqno,11pt]{amsart}
\usepackage[left=1in,right=1in,top=1in,bottom=1in]{geometry}
\usepackage{enumitem}
\usepackage{graphicx}
\usepackage{tikz}
\usetikzlibrary{arrows, decorations.markings,matrix,trees}
\usepackage{hyperref}
\hypersetup{
   colorlinks=true,
   citecolor=blue,
   filecolor=blue,
   linkcolor=blue,
   urlcolor=blue
}
\def\Xint#1{\mathchoice
  {\XXint\displaystyle\textstyle{#1}}%
  {\XXint\textstyle\scriptstyle{#1}}%
  {\XXint\scriptstyle\scriptscriptstyle{#1}}%
  {\XXint\scriptscriptstyle\scriptscriptstyle{#1}}%
  \int}
\def\XXint#1#2#3{{\setbox0=\hbox{$#1{#2#3}{\int}$}
  \vcenter{\hbox{$#2#3$}}\kern-.5\wd0}}

\def\dashint{\Xint-}

\newcommand{\al}{\alpha}

\newcommand{\lda}{\lambda}
\newcommand{\om}{\Omega}
\newcommand{\pa}{\partial}
\newcommand{\va}{\varepsilon}
\newcommand{\ud}{\mathrm{d}}
\newcommand{\be}{\begin{equation}}
\newcommand{\ee}{\end{equation}}
\newcommand{\w}{\omega}
\newcommand{\Lda}{\Lambda}

\newcommand{\C}{\mathbb{C}}

\newcommand{\cD}{\mathcal{D}}

\newcommand{\cE}{\mathcal{E}}

\newcommand{\cF}{\mathcal{F}}

\newcommand{\bG}{\mathbb{G}}

\newcommand{\cL}{\mathcal{L}}

\newcommand{\Z}{\mathbb{Z}}

\newcommand{\M}{\mathcal{M}}
\newcommand{\MM}{\mathbb{M}}

\newcommand{\cN}{\mathcal{N}}

\newcommand{\cP}{\mathcal{P}}

\newcommand{\R}{\mathbb{R}}

\newcommand{\cR}{\mathcal{R}}

\newcommand{\Ss}{\mathbb{S}}

\newcommand{\cX}{\mathcal{X}}

\newcommand{\wc}{\rightharpoonup}

\newcommand{\HH}{\mathcal{H}}

\newcommand{\vp}{\varphi}

\newcommand{\Ga}{\Gamma}

\newcommand{\sg}{\sigma}
\newcommand{\ift}{\infty}
\newcommand{\wt}{\widetilde}
\newcommand{\wh}{\widehat}
\newcommand{\f}{\frac}

\newcommand{\ol}{\overline}

\newcommand{\op}{\operatorname}
\newcommand{\Sg}{\Sigma}

\newcommand{\na}{\nabla}

\newcommand{\abs}[1]{\lvert#1\rvert}

\DeclareMathOperator{\dist}{dist}

\DeclareMathOperator{\supp}{supp}

\DeclareMathOperator{\tr}{tr}
\DeclareMathOperator*{\osc}{osc}

\DeclareMathOperator{\sing}{sing}

\DeclareMathOperator{\loc}{loc}

\def\<{\langle}\def\>{\rangle}
\def\({\left(}\def\){\right)}
\numberwithin{equation}{section}
\theoremstyle{plain}
\newtheorem{thm}{Theorem}[section]

\newtheorem{lem}[thm]{Lemma}
\newtheorem{prop}[thm]{Proposition}

\theoremstyle{definition}
\newtheorem{defn}[thm]{Definition}

\theoremstyle{remark}
\newtheorem{rem}[thm]{Remark}

\title[Ginzburg-Landau type functionals in high dimensions]{Asymptotics of Minimizers for Ginzburg--Landau-type Functionals in High Dimensions}

\date{}

\author{Giacomo Canevari}
\address{Universit\`{a} di Verona, Strada Le Grazie 15, 37134 Verona, Italy}
\email{giacomo.canevari@univr.it}

\author{Haotong Fu}
\address{School of Mathematical Sciences, Peking University, Beijing 100871, China}
\email{2301110012@pku.edu.cn}

\author{Wei Wang}
\address{School of Mathematical Sciences, Peking University, Beijing 100871, China}
\email{gjmtamag@gmail.com,\,\,2201110024@stu.pku.edu.cn}

\begin{document}

\begin{abstract}
We investigate local minimizers of Ginzburg--Landau-type functionals in dimension $n\ge 3$ that satisfy logarithmic energy bounds, assuming the potential has a vacuum manifold with a finite fundamental group. We show that the normalized energy measures converge to an $(n-2)$-rectifiable measure associated with a stationary varifold, with quantized density determined by the homotopy classes of the vacuum manifold. Away from the support of the $(n-2)$-rectifiable measure, the minimizers converge strongly in $H^1_{\loc}$ to a minimizing harmonic map, which is smooth outside an $(n-3)$-rectifiable singular set.
\end{abstract}



\maketitle

\section{Introduction}

Let $ \om\subset\R^n $ with $n\in\Z$, $n\geq 2$ be a bounded domain, and let $ \cN $ be a smooth, compact, connected Riemannian manifold that is isometrically embedded in $ \R^{m} $ for some integer $m=m(\cN)$. For $ \va\in(0,1) $ and $ u\in H^1(\om,\R^m) $, we consider the Ginzburg--Landau-type functional
\be
E_{\va}(u,\om):=\int_{\om}e_{\va}(u)\ud x,\quad e_{\va}(u):=\f{1}{2}|\na u|^2+\frac{1}{\va^2}f(u),\tag{GL${}_{\va}$}\label{GLfunctional}
\ee
where the potential $f\in C^1(\R^m,[0,+\ift))$ attains its minimum on the ``vacuum manifold'' $ \cN=f^{-1}(0) $. The Euler--Lagrange equation associated with \eqref{GLfunctional} is
\be
\Delta u=\f{1}{\va^2}D_u f(u).\label{EL}
\ee
The Ginzburg--Landau-type functional \eqref{GLfunctional} arises in several physical models, including superconductivity and liquid crystals. A typical case is where
\begin{equation} \label{GLS1}
\cN=S^1\subset\C\quad\text{and}\quad f(u)=\frac{1}{2}(|u|^2-1)^2.
\end{equation}
This case has been extensively studied. In two dimensions, the asymptotic behavior of minimizers (with $f$ as in~\eqref{GLS1}) is well understood, following Bethuel, Br\'{e}zis, H\'{e}lein \cite{BBH93, BBH94} and Struwe \cite{Str94}. For a fixed boundary datum $g_{\va}=g:\pa\Omega\to S^1$ of degree~$d\neq 0$, the minimizers $ u_\varepsilon $ converge to a limit map, which is smooth away from $\abs{d}$ singular points. Moreover, the energy satisfies
\[
E_{\va}(u_{\va},\Omega)=\pi|d||\log\va|+O(1)\quad\text{as }\va\to 0^+.
\]
Analogous but generally weaker results for non-minimizing solutions of \eqref{EL}, \eqref{GLS1} hold under suitable assumptions; see \cite[Chapter X]{BBH94} and \cite{Bet06, Riv99} for further developments. In higher dimensions, singularities are no longer isolated and are expected to concentrate on codimension-two sets. One typically works under the energy bound
\be
E_{\va}(u_{\va},\om)\le C|\log\va|,\label{GLS1log}
\ee
where $C>0$ is independent of $\va$.
For minimizers, Lin and Rivi\`{e}re \cite{LR99} proved that, up to a subsequence, $ u_{\va} $ converges in $ C_{\loc}^k $ away from an $ (n-2) $-rectifiable integral area-minimizing varifold. Related variational descriptions through $ \Gamma $-convergence were developed in \cite{ABO05, JS02}. For solutions of \eqref{EL}, the case $ n=3 $ was treated in~\cite{LR01} and later extended to all dimensions $n\geq 3$ by Bethuel, Br\'{e}zis, Orlandi~\cite{BBO01}. In this setting, the authors established improved convergence away from an $ (n-2) $-rectifiable stationary varifold, along with global uniform estimates. Unlike the minimizing case, the limiting varifold need not be integral or area-minimizing; see \cite{PS23} for a recent example. Further results for the $\Ss^1$-valued Ginzburg--Landau model have been obtained in several directions, including Riemannian manifolds, weighted energies, and models with magnetic fields, as studied in \cite{AS98a, AS98b, BOS05, BBM04, SS07, Ste21}.

We now consider the case where the vacuum manifold $ \mathcal{N} $ is a general compact manifold, and the potential $ f $ satisfies suitable non-degeneracy assumptions near $ \cN $. One direct difference from the case $ \cN=S^1 $ is that singularities may appear even when the energy is bounded, that is,
\be
E_{\va}(u_{\va},\om) \leq C\label{boundedenergy}
\ee
for some $ \va $-independent constant $ C>0 $. For minimizers, when \eqref{boundedenergy} holds, Contreras, Lamy, and Rodiac proved in \cite{CL22, CLR18} that, up to a subsequence, $ u_{\va} $ converges strongly in $ H^1 $ to a minimizing harmonic map $ u_* $ with target manifold $ \cN $, and the convergence is uniform away from the singular set of $ u_* $. Standard results on minimizing harmonic maps (see \cite{NV17, SU82}) imply that the singular set of $ u_* $ is $ (n-3) $-rectifiable. For solutions of \eqref{EL} with uniform bounded energy, there exists a singular set that is instead $ (n-2) $-rectifiable, and the convergence is uniform away from this set. We refer to Lin and Wang~\cite{LW99} for details. One can also consider the logarithmic energy regime~\eqref{GLS1log}. In two dimensions, compactness results for minimizers were proved in~\cite{Can15}, followed by a theory of vortices by Monteil, Rodiac, and Van Schaftingen~\cite{MRS21, MRS22}, which extends earlier results for minimizers in the case~$\cN = S^1$~\cite{BBH93, BBH94}. In three dimensions, for the Landau--de Gennes model, both the uniaxial case ($ \cN =\R\mathbb{P}^2 $) and the biaxial case ($ \cN=\Ss^3/Q_8 $, where~$ Q_8 $ is the Quaternion group) have been studied in~\cite{Can17} and~\cite{WZ24}, respectively. These results are analogous to those in~\cite{LR99}. In higher dimensions, when the fundamental group of $ \mathcal{N} $ is Abelian, a variational characterization of the energy concentration sets has been obtained via the $ \Gamma $-convergence~\cite{CO21}.

For solutions of \eqref{EL} satisfying \eqref{GLS1log}, with general~$\cN$, the situation is quite different from the case where $ \mathcal{N} = S^1 $. In general, one cannot rely on tools based on the Hodge decomposition for maps taking values in $ S^1 $. Furthermore, even in the case of minimizers, the arguments in~\cite{Can17, WZ24} are essentially based on the facts that the fundamental group of $ \mathcal{N} $ is finite and that $ n=3 $. These issues motivate the setting of the present paper: we consider the case where the fundamental group of $ \mathcal{N} $ is finite, and the domain has dimension $ n\geq 3 $. More precisely, we assume the potential $ f $ and the vacuum manifold $ \mathcal{N} $ satisfy the following conditions:

\begin{enumerate}[label=$(\op{H}\theenumi)$]
\item\label{A1} $ f\in C^{\ift}(\R^m) $ and $ f(y)\geq 0 $ for any $ y\in\R^m $.
\item\label{A2} The fundamental group $ \pi_1(\cN) $ of $ \cN $ is finite.
\item\label{A3} $ f $ does not vanish at infinity, i.e.,
\[
\liminf_{|y|\to+\ift}f(y)>0.
\]
\item\label{A4}$ f $ vanishes non-degenerately on $ \cN $, that is,
\[
D^2f(x)[v,v]>0\quad\text{ for any }x\in\cN,\,\,v\in(T_x\cN)^{\perp} \backslash\{0\}.
\]
\end{enumerate}

Before stating the main theorem, we define the notion of local minimizers for a given functional.

\begin{defn}[Local minimizers]
Let $ \om\subset\R^n $ be a bounded domain. Define
\[
\cF(u,U):=\int_UF(x,u,\nabla u)\ud x,
\]
where $ U\subset\Omega $ is open and $ F:\om\times\R^m\times \R^{m\times n}\to \R $. We say that $ u\in H^1_{\loc}(\Omega,\R^m) $ is a local minimizer of $ \cF $ if, for any ball $ B_r(x)\subset\subset\om $ and any
$ v\in H^1(B_r(x),\R^m) $ satisfying $ v=u $ on $ \pa B_r(x) $ in the trace sense, we have
\[
\cF(u,B_r(x))\leq\cF(v,B_r(x)).
\]
\end{defn}

Our main result concerns the asymptotic behavior, as~$\va\to 0$,
of a family of local minimizers $ \{u_{\va}\}_{\va\in(0,1)}\subset H^1(\om,\R^m) $ of \eqref{GLfunctional} satisfying
\be
E_{\va}(u_{\va},\om)\leq M(|\log\va|+1)\quad\text{and}\quad\|u_{\va}\|_{L^{\ift}(\om,\R^m)}\leq M\label{assumptionbound}
\ee
for some $\va$-independent constant~$ M>0 $.
Given such a family, we define the normalized energy measures $ \{\mu_{\va}\}_{\va\in(0,1)}\subset(C^0(\ol{\om}))' $ by
\be
\mu_{\va}(A):=\f{1}{|\log\va|}\int_Ae_{\va}(u_{\va})\ud x\label{definemeasurewithuva}
\ee
for any Borel set $ A \subset\om $. Because of the assumption~\eqref{assumptionbound}, there exist a $ \mu_*\in(C^0(\ol{\om}))' $ and a sequence $ \va_i\to 0^+ $ such that $ \mu_{\va_i}\wc^*\mu_* $ in $ (C^0(\ol{\om}))' $.
Our main result provides information on the support~$ S_*:=\supp(\mu_*) $  of~$ \mu_* $. 

\begin{thm}\label{maintheorem}
Let $ M>0 $ and $ \om\subset\R^n $ be a bounded domain with $ n\geq 3 $. Assume that \ref{A1}--\ref{A4} hold, and let $ \{u_{\va}\}_{\va\in(0,1)}\subset H^1(\om,\R^m) $ be a family of local minimizers of \eqref{GLfunctional} satisfying~\eqref{assumptionbound}. Then, the following properties hold.
\begin{enumerate}[label=$(\theenumi)$]
\item $ S_* $ is an $ (n-2) $-rectifiable set.
\item $ \mu_* $ corresponds to an $ (n-2) $-stationary varifold $($see Section \ref{stationaryvarifoldsec}$)$. Moreover,
\[
\mu_*\llcorner\om=\Theta^{n-2}(\mu_*,x)\HH^{n-2}\llcorner(S_*\cap\om),
\]
and for $ \HH^{n-2} $-a.e. $ x\in S_* $,
\be
\Theta^{n-2}(\mu_*,x)\in\{|\sg|_*:\sg\in[\Ss^1,\cN]\},\label{Thetadiscrete}
\ee
where $ [\Ss^1,\cN] $ denotes the collection of free homotopy classes of $ \cN $ $($see Definition \ref{definitionfreehomotopyclass}$)$, and $ |\cdot|_* $ is a suitable norm on~$ [\Ss^1,\cN] $ $($see Definition \ref{DefinitionEsg}$)$.
\item For any open set $ U\subset\subset\om\backslash S_* $, there exists $ C>0 $, depending only on $ f,M,\cN,n $, and $ U $ such that \[
\limsup_{i\to+\ift}E_{\va_i}(u_{\va_i},U)\leq C.
\]
\item There exists $ u_*\in H_{\loc}^1(\om\backslash S_*,\cN) $, a local minimizer of the Dirichlet energy
\be
E_0(u,\om):=\f{1}{2}\int_{\om}|\na u|^2\ud x,\label{DirichletEnergy}
\ee
such that, up to a subsequence $ u_{\va_i}\to u_* $ strongly in $ H_{\loc}^1(\om\backslash S_*,\R^m) $. Up to a subsequence, $ u_{\va_i}\to u_* $ in $ C_{\loc}^0(\om\backslash(S_*\cup\sing(u_*))) $, where $ \sing(u_*) $ represents the singular set of $ u_* $ and $ \sing(u_*) $ is $ (n-3) $-rectifiable.
\end{enumerate}
\end{thm}

\begin{rem}
Several remarks are as follows.
\begin{enumerate}
\item If $n=3$, then by \eqref{Thetadiscrete} and the structure theorem for $1$-dimensional stationary varifolds \cite[Theorem, p. 89]{AA76}, for any open set $U\subset\subset\om$, there exist finitely many closed line segments $\{\ell_i\}_{i=1}^k$ and classes $\{\sigma_i\}_{i=1}^k \subset [\Ss^1,\cN]$ such that
\[
S_*\cap\ol{U}=\bigcup_{i=1}^k\ell_i,\quad\mu_*=\sum_{i=1}^k|\sg_i|_*\HH^1\llcorner\ell_i.
\]
In particular, $ \sing(u_*) $ is locally finite.
\item When $n=3$, we introduce an additional variable $x_4 \in (-1,1)$ and define $\wt{u}_{\va}(x,x_4) = u_{\va}(x)$ for $x=(x_1,x_2,x_3)\in\om$. One checks that $\{\wt{u}_{\va}\}_{\va\in(0,1)}$ is a family of local minimizers of $E_{\va}(\cdot,\om\times(-1,1))$. Hence, all the properties of Theorem \ref{maintheorem} for $n=3$ follow from the corresponding results in dimension $ n=4 $. For this reason, we assume that $n \ge 4$ in the rest of this paper, unless otherwise specified.
\item In special cases, such as the three-dimensional Landau--de Gennes model studied in \cite{Can17,WZ24}, there are sufficient conditions (see \cite[Section 6]{Can17} and \cite[Section 2.2]{WZ24}) ensuring that global minimizers satisfy \eqref{assumptionbound}. For the Ginzburg--Landau-type functional considered here, we also provide general conditions (see Proposition \ref{sufficientconditions}).
\end{enumerate}
\end{rem}

\subsection{Difficulties and strategies}

The key ingredient in the proof of Theorem \ref{maintheorem} is the clearing-out property. More precisely, let $ \{u_{\va}\}_{\va\in(0,1)} $ be a family of local minimizers of \eqref{GLfunctional} in $\om\subset\R^n$. We aim to prove that there exist $\ol{\va}_0,\eta\in(0,1)$ and $C>0$, all independent of $\va$, such that whenever $\va\in(0,\ol{\va}_0 r)$,
\be
E_{\va}(u_{\va},B_{2r}^n(x))\leq \eta r^{n-2}\log\f{r}{\va}\quad\Longrightarrow\quad E_{\va}(u_{\va},B_r^n(x))\leq Cr^{n-2}.\label{clearingout1}
\ee
To explain the difficulty in arbitrary dimensions, we first recall the main step in the $3$-dimensional case, following \cite{Can17}. In dimension three, the proof of \eqref{clearingout1} is based on a Luckhaus-type lemma on $\pa B_1^3$, inspired by~\cite[Lemma~1]{Luc88}. Generally speaking, given a map $w_{\va}\in H^1(\pa B_1^3,\R^m)$ such that
\be
E_{\va}(w_{\va},\pa B_1^3)\ll|\log\va|,\label{wvaB13small}
\ee
one can construct a map $ v_{\va}\in H^1(\pa B_1^3,\cN) $ together with an interpolation $\vp_{\va}\in H^1(B_1^3\backslash B_{1-h(\va)}^3,\R^m)$, where $h(\va)=\va^{\f{1}{2}}|\log\va|$, such that
\begin{gather*}
\vp_{\va}(x)=w_{\va}(x)\text{ and }
\vp_{\va}((1-h(\va))x)=v_{\va}(x)\quad \text{for }\HH^2\text{-a.e. }x\in\pa B_1^3,\\
\|\na v_{\va}\|_{L^2(\pa B_1^3)}\leq C E_{\va}(w_{\va},\pa B_1^3),\\
E_{\va}(\vp_{\va},B_1^3\backslash B_{1-h(\va)}^3)\leq C h(\va)E_{\va}(w_{\va},\pa B_1^3).
\end{gather*}
Using this construction and an extension argument as in~\cite{HKL86}, one obtains suitable competitors for the local minimizers. It then follows that, under the assumption in \eqref{clearingout1}, there exists a set $D\subset [r,2r]$ with $\HH^1(D)\geq\f{r}{4}$ such that, for every $\rho\in D$,
\be
E_{\va}(u_{\va},B_{\rho}^3(x))\leq Cr\((E_{\va}(u_{\va},\pa B_{\rho}^3(x)))^{\f{1}{2}}+1\).\label{ODE11}
\ee
This estimate is strong enough to yield the desired clearing-out property. The above argument is specific to dimension three. The reason is that the construction is carried out on the two-dimensional boundary $\pa B_1^3$, which is the critical dimension for the appearance of singularities. In this setting, \eqref{wvaB13small} is enough to rule out topological obstructions in each two-dimensional cell in a suitable partition of $\pa B_1^3$ at the scale $h(\va)$. This is the essential point of the Luckhaus lemma. However, in higher dimensions, this argument no longer applies. Indeed, the singular set on $\pa B_1^n$ is expected to have dimension $n-3$. In particular, one cannot choose a predetermined scale $ h(\va)\to 0^+ $ depending only on $\va$ and still guarantee that all two-dimensional cells in the construction are free of topological defects.

To overcome this difficulty, we choose the interpolation scale adaptively based on the boundary datum. More precisely, for $w_{\va}\in H^1(\pa B_1^n,\R^m)$ satisfying
\[
E_{\va}(w_{\va},\pa B_1^n)\ll|\log\va|,
\]
we choose the scale $h(\va)$ to depend on $w_{\va}$, namely,
\[
h(\va):=\(\f{E_{\va}(w_{\va},\pa B_1^n)}{\eta_0|\log\va|}\)^{\f{1}{n-3}},
\]
where $\eta_0\in(0,1)$ is a sufficiently small constant to be fixed later. This choice allows us to control the topological obstructions on every $2$-dimensional cell and thereby enables the proof of the Luckhaus-type lemma. The exponent $\f{1}{n-3}$ reflects the expected dimension of the singular set in $\pa B_1^n$.

Since $h(\va)$ now depends on the boundary energy, the estimate \eqref{ODE11} no longer holds in its original form. Instead, we obtain a weaker differential inequality whose right-hand side contains a nonlinear term originating from the dependence of $h(\va)$ on $E_{\va}(u_{\va},\pa B_{\rho}^n(x))$. However, this new inequality is weaker than \eqref{ODE11}; after technical refinements to the ODE argument, it is still sufficient to derive the clearing-out property \eqref{clearingout1}.

\subsection{Organization of this paper}

In Section \ref{SectionPreliminaries}, we collect the basic ingredients for the proof of the main theorem. These include the analysis of topological singularities in the two-dimensional case, several extension results, lower and upper energy bounds on disks and cylinders, the monotonicity formula, and the compactness of minimizers in the bounded-energy regime. In Section \ref{SectionLuckhaus}, we establish Luckhaus-type results on both the unit ball and the cylinders. These results are a key tool in the proof of the clearing-out property in Proposition \ref{clearingout}, which is proved in Section \ref{SectionSingularset1}. Using the clearing-out property, together with several standard tools in geometric measure theory, we show in Section \ref{SectionSingularset1} that the limiting measure is associated with an $(n-2)$-stationary varifold. Finally, in Section \ref{SectionSingularset1prime}, we prove quantization of the density of the stationary varifold obtained in the previous section and complete the proof of Theorem \ref{maintheorem}.

\subsection{Notations and conventions}

\begin{itemize}
\item Throughout the paper, $C$ denotes a positive constant. When necessary, we write $C(a,b,...)$ to indicate its dependence on the parameters $a,b,...$ Additionally, the value of $C$ may change from line to line.

\item We adopt the Einstein summation convention: repeated indices are implicitly summed over their range.

\item For $k\in\Z$, $k\geq 2$, we define
\[
B_r^k(x):=\{y\in\R^k:\,\,|y-x|<r\}.
\]
We denote the origin of $\mathbb R^k$ by $0^k$. When $k=n$, we suppress the superscript and simply write $0$. Thus, when $k=n$, we write $B_r(x)$ instead of $B^n_r(x)$, and we write $B_r$ instead of $B^n_r(0)$.

\item For $k\in\Z\cap[0,n]$, $\HH^k$ denotes the $k$-dimensional Hausdorff measure in $\R^n$, and $\dim_{\HH}$ denotes the Hausdorff dimension.

\item  We denote the Euclidean inner product in $\R^n$ by $u\cdot v$ for $u,v\in\R^n$, and the inner product in $\R^m$ by $u:v$ for $u,v\in\R^m$.

\item Let $k\in\Z\cap[1,n]$, and let $\M\subset\R^n$ be a $k$-dimensional Lipschitz submanifold. We denote by $\na_{\M}$ the tangential gradient on $\M$, which exists $\HH^k$-a.e. by Rademacher's theorem. If $\M$ is of class $C^j$ with $j\in\Z$ such that $ j\geq 2$, we denote by $D_{\M}^j$ the $j$-th order tangential derivatives. For $u\in H^1(\M,\R^m)$, we define
\[
e_{\va}(u,\M):=\frac{1}{2}|\na_{\M}u|^2+\frac{1}{\va^2}f(u),
\quad
E_{\va}(u,\M):=\int_{\M} e_{\va}(u,\M)\ud\HH^k,
\]
for $\va\in(0,1)$. If $\M\subset\R^n$ is an open set, we write $\na_{\M}=\na$, $D_{\M}^j=D^j$, and $e_{\va}(u,\M)=e_{\va}(u)$.

\item Let $U\subset\R^k$ be a Lipschitz domain and $u,v\in H^1(U,\R^m)$. We write $u|_{\pa U}=v|_{\pa U}$ to mean that $u=v$ on $\pa U$ in the sense of traces.

\item We use $\HH^k$ to denote the $k$-dimensional Hausdorff measure. In particular, when $k=n$, $\HH^n$ coincides with the Lebesgue measure $\cL^n$ on $\R^n$. For simplicity, we sometimes omit $\ud\cL^n$ in the integrals.
\end{itemize}

\section{Preliminaries}\label{SectionPreliminaries}

\subsection{Basic properties of the manifold \texorpdfstring{$ \cN $}{} and the potential \texorpdfstring{$ f $}{}}

In this subsection, we collect some standard properties of the target manifold $\cN$ and the potential $f$, which will be used throughout the paper. For $ \delta>0 $, we denote the $ \delta $-neighborhood of $ \cN $ by
\[
\cN_{\delta}:=\{y\in\R^{m}:\dist(y,\cN)<\delta\}.
\]

\begin{lem}
Assume that $ f $ and $ \cN $ satisfy \ref{A1}-\ref{A4}. Then there exists $ \delta_{\cN}>0 $, depending only on $ \cN $, such that the nearest point projection onto $\cN$ is well-defined. More precisely, there exists a unique $C^{\ift}$ map $\Pi_{\cN}:\cN_{\delta_{\cN}}\to\cN$ such that
\begin{equation}
|y-\Pi_{\cN}(y)|=\dist(y,\cN)\quad\text{for any }y\in\cN_{\delta_{\cN}}.
\label{Nearestpointprojection}
\end{equation}
Moreover, there exist $\delta_f\in(0,\delta_{\cN})$ and constants $0<m_f<M_f<+\ift$ such that, for any $y\in\cN_{\delta_f}$,
\begin{equation}
\begin{gathered}
m_f(\dist(y,\cN))^2\leq Df(y)\cdot(y-\Pi_{\cN}(y))\leq M_f(\dist(y,\cN))^2,\\
m_f(\dist(y,\cN))^2\leq f(y)\leq M_f(\dist(y,\cN))^2,
\end{gathered}
\label{A31}
\end{equation}
and
\begin{equation}
f(ty+(1-t)\Pi_{\cN}(y))\leq M_ft^2f(y)\quad\text{for any }t\in[0,1].
\label{fBconvex}
\end{equation}
\end{lem}

\begin{proof}
The existence and smoothness of $\Pi_{\cN}$ follow from \cite[Lemma 2.1]{MRS21}. The estimates in \eqref{A31} follow from \cite[Lemmas 2.2 and 2.3]{MRS21}, together with \ref{A4}. It remains to prove \eqref{fBconvex}. Fix $ y\in\cN_{\delta_f} $ and set $ \wt{y}:=y-\Pi_{\cN}(y) $. If $ |\wt{y}|=0 $, then $y\in\cN$ and \eqref{fBconvex} is trivial. Hence, we assume $ |\wt{y}|>0 $. Define
\[
g(s):=f\left(\Pi_{\cN}(y)+s\frac{\wt{y}}{|\wt{y}|}\right),\quad s\in\R.
\]
By Taylor expansion at $s=0$ (possibly after reducing $\delta_f$), we have
\[
g(s)=g(0)+g'(0)s+\frac{g''(0)}{2}s^2+R(s),\quad |R(s)|\leq C|s|^3,
\]
for any $ s\in(0,\delta_f) $. Since $\Pi_{\cN}(y)\in\cN$, assumption \ref{A1} implies $g(0)=g'(0)=0$. Moreover, by the smoothness of $f$ and the compactness of $\cN$, we have $|g''(0)|\leq C$. It follows that $ g(s)\leq Cs^2 $ for any $ s\in(0,\delta_f) $. Taking $ s=t|\wt{y}| $ with $ t\in[0,1] $, we obtain
\[
f(\Pi_{\cN}(y)+t(y-\Pi_{\cN}(y)))=g(t|\wt{y}|)\leq Ct^2|\wt{y}|^2.
\]
Using \eqref{A31}, we have $|\wt{y}|^2\leq C f(y)$, and therefore \eqref{fBconvex} holds.
\end{proof}

\subsection{Topology of the manifold \texorpdfstring{$ \cN $}{}}

In this subsection, we introduce the free homotopy classes in $\cN$ and the associated concepts. These notions go back to \cite{Mer79}; see also \cite{Can15,MRS21,MRS22}. We recall them here for convenience.

\begin{defn}\label{definitionfreehomotopyclass}
Let $\ell_1,\ell_2\in C^0(\Ss^1,\cN)$ be two loops. We say that $\ell_1$ and $\ell_2$ are freely homotopic if there exists $G\in C^0([0,1]\times\Ss^1,\cN)$ such that
\[
G(0,\cdot)=\ell_1\quad\text{and}\quad G(1,\cdot)=\ell_2.
\]
We denote this by $\ell_1\sim_{\cN}\ell_2$. This defines an equivalence relation on $C^0(\Ss^1,\cN)$. The set of free homotopy classes is defined by
\[
[\Ss^1,\cN]:=C^0(\Ss^1,\cN)/\sim_{\cN}.
\]
For $\ell\in C^0(\Ss^1,\cN)$, we denote its class by $[\ell]_{\cN}$.
\end{defn}

Unlike homotopy classes defining the fundamental group $\pi_1(\cN,x)$, free homotopy does not fix a base point. It follows from \cite[Exercise 6, p. 38]{Hat02} that there is a natural bijection
\be
[\Ss^1,\cN]\longleftrightarrow\{\text{conjugacy classes of }\pi_1(\cN)\}.\label{bijection}
\ee

\begin{defn}\label{Definitionplus}
Let $\cP([\Ss^1,\cN])$ be the power set of $[\Ss^1,\cN]$. We define an operation
\[
+:[\Ss^1,\cN]\times[\Ss^1,\cN]\to\cP([\Ss^1,\cN])
\]
as follows. Let $\{B_{r_i}^2(x_i)\}_{i=0}^2$ be three balls in $\R^2$ such that
\[
B_{r_1}^2(x_1)\cup B_{r_2}^2(x_2)\subset B_{r_0}^2(x_0)\quad\text{and}\quad\overline{B}_{r_1}^2(x_1)\cap \overline{B}_{r_2}^2(x_2)=\emptyset.
\]
Given $\al,\beta\in[\Ss^1,\cN]$, we define $\al+\beta$ as the set of all $\sg\in[\Ss^1,\cN]$ for which there exists $G\in C^0(\ol{B}_{r_0}^2(x_0),\cN)$ such that
\[
[G|_{\pa B_{r_1}^2(x_1)}]_{\cN}=\al,\quad
[G|_{\pa B_{r_2}^2(x_2)}]_{\cN}=\beta,\quad\text{and }
[G|_{\pa B_{r_0}^2(x_0)}]_{\cN}=\sg.
\]
\end{defn}

By \eqref{bijection}, each $\al\in[\Ss^1,\cN]$ corresponds to a conjugacy class of $\pi_1(\cN)$. For $\al,\beta\in[\Ss^1,\cN]$, the operation $\al+\beta$ admits the algebraic interpretation
\begin{equation} \label{conjugacy}
\al+\beta=\{\text{conjugacy classes of }ab:a\in\al,\,\,b\in\beta\} ,
\end{equation}
where we have identified~$\alpha$, $\beta$ with subsets of~$\pi_1(\cN)$
(conjugacy classes) on the right-hand side.
(See e.g.~\cite[Section~III.C]{Mer79} or~\cite[Lemma~2.2]{Can15} for a proof of~\eqref{conjugacy}.) In particular, $+$ is commutative, although it may be multi-valued when $\pi_1(\cN)$ is non-Abelian. Nevertheless, this operation still retains some of the properties that are distinctive of a group. More precisely, we have:
\begin{enumerate}[label=(\roman*)]
 \item there exists~$\va\in [\Ss^1,\cN]$
 (the homotopy class of constant paths) such that for all~$\alpha\in [\Ss^1,\cN]$,
 $\va + \alpha = \{\alpha\}$;
 \item for all~$\alpha\in [\Ss^1,\cN]$, there exists~$-\alpha\in [\Ss^1,\cN]$
 such that~$\alpha - \alpha \ni \va$;
 \item the operation~$+$ is associative, that is, we have
 \[
  (\alpha + \beta) + \gamma := \bigcup_{\sigma\in\alpha + \beta} (\sigma + \gamma)
  = \bigcup_{\sigma\in\beta + \gamma} (\alpha + \sigma)
  =: \alpha + (\beta + \gamma)
 \]
 for all~$\alpha$, $\beta$, $\gamma$ in~$[\Ss^1,\cN]$;
 \item the operation~$+$ is commutative, that is, we have
 $\alpha + \beta = \beta + \alpha$
 for all~$\alpha$, $\beta$ in~$[\Ss^1,\cN]$.
\end{enumerate}
All of these properties can be proved in a rather straightforward way using the characterization~\eqref{conjugacy} of the operation. For instance, Property~(iv) follows from the identity~$ab = (aba^{-1})a$, valid for all~$a$, $b$ in~$\pi_1(\cN)$

For $\sg\in[\Ss^1,\cN]$, define
\[
E_{\op{min}}(\sg):=\inf\left\{\frac{1}{2}\int_{\Ss^1}|u'|^2\ud\HH^1:u\in H^1(\Ss^1,\cN),\,\, [u]_{\cN}=\sg\right\}.
\]
Using the direct method of the calculus of variations, we see that the infimum is attained by a geodesic representative of the class $\sg$.

\begin{defn}\label{DefinitionEsg}
For $ \sg\in[\Ss^1,\cN] $, define the norm $ |\cdot|_* $ by
\be
|\sg|_*:=\inf\left\{\sum_{i=1}^kE_{\op{min}}(\sg_i):\sg\in\sum_{i=1}^k\sg_i,\,\,\sg_i\in[\Ss^1,\cN],\,\,k\in\Z_+\right\}.\label{normsg}
\ee
\end{defn}

Since by \ref{A2}, $\pi_1(\cN)$ is finite, the infimum in \eqref{normsg} is achieved. We refer to $|\cdot|_*$ as a norm, since it satisfies the following properties:
\begin{enumerate}[label=(\alph*)]
\item $|\sg|_*=0$ if and only if $\sg=0$, where $0$ denotes the trivial class.
\item If $\sg,\sg'\in[\Ss^1,\cN]$ and $0\in\sg+\sg'$, then $|\sg|_*=|\sg'|_*$.
\item If $\sg\in\sg_1+\sg_2$, then $|\sg|_*\leq |\sg_1|_*+|\sg_2|_*$.
\end{enumerate}
Moreover, there exists a constant $c_0=c_0(\cN)>0$ such that
\be
\inf_{\sg\in[\Ss^1,\cN],\,\,\sg\neq 0}|\sg|_*=\min_{\sg\in[\Ss^1,\cN],\,\,\sg\neq 0}|\sg|_*>c_0.\label{c0geq}
\ee
Property~\eqref{c0geq} follows from the compact embedding~$H^1(\Ss^1)\hookrightarrow C(\Ss^1)$: if~\eqref{c0geq} were false, there would be a sequence of non-null-homotopic maps~$\Ss^1\to\cN$ that converge strongly in~$H^1(\Ss^1)$, and hence uniformly, to a constant; this is a contradiction because homotopy classes are stable with respect to uniform convergence. Property~(a) follows from~\eqref{c0geq}, while Properties~(b) and~(c) are direct consequences of Definition~\ref{DefinitionEsg}.

\begin{rem}\label{remequvalence}
In \cite{MRS21,MRS22}, the authors use the notion of singular energy $\cE^{\op{sg}}(\cdot)$. This is equivalent to the norm $|\cdot|_*$.
\end{rem}

We conclude with the following extension lemma.

\begin{lem}\label{propextension}
Let $N\in\Z_+$ and $\{\sg_i\}_{i=0}^N\subset[\Ss^1,\cN]$. Let $\{B_{r_i}^2(x_i)\}_{i=0}^N$ be balls in $\R^2$ such that
\[
\bigcup_{i=1}^N B_{r_i}^2(x_i)\subset B_{r_0}^2(x_0),\quad\text{and}\quad
B_{r_i}^2(x_i)\cap B_{r_j}^2(x_j)=\emptyset\text{ for any }i\neq j.
\]
Define
\[
U:=B_{r_0}^2(x_0)\backslash\(\bigcup_{i=1}^N \ol{B}_{r_i}^2(x_i)\).
\]
Assume that $\sigma_0\in\sum_{i=1}^N\sigma_i$. Then there exists $G\in C^1(\ol{U},\cN)$ such that the following properties hold.
\begin{itemize}
\item $[G|_{\pa B_{r_i}^2(x_i)}]_{\cN}=\sg_i$ for any $i\in\Z\cap[0,N]$.
\item $\|\nabla G\|_{L^\infty(U)}\leq C$, where $C>0$ depends only on $N,\cN$, and $U$.
\end{itemize}
\end{lem}

\begin{proof}
Let $D_i:=B_{r_i}^2(x_i)$ for $i\in\Z\cap[0,N]$. For $s>0$, write $sD_i:=B_{sr_i}^2(x_i)$. By Definition \ref{Definitionplus}, there exists $G_0\in C^0(\ol{U},\cN)$ such that $[G_0|_{\pa D_i}]_{\cN}=\sg_i$ for any $i\in\Z\cap[0,N]$. There exists $\rho_0>0$, depending only on $U$, such that
\[
\bigcup_{i=1}^N (1+\rho_0)D_i\subset (1-\rho_0)D_0,\quad\text{and}\quad
(1+\rho_0)D_i\cap(1+\rho_0)D_j=\emptyset\text{ for any }i\neq j.
\]
For $\rho\in(0,\rho_0)$, define
\[
U_\rho:=(1-\rho)D_0\backslash\(\bigcup_{i=1}^N (1+\rho)\ol{D}_i\).
\]
By convolution, we construct a sequence $\{G_k\}_{k\in\Z_+}\subset C^1(U,\R^m)$ such that $G_k\to G_0$ uniformly in $U_{\f{\rho_0}{4}}$. Hence
\[
\lim_{k\to+\ift}\(\sup_{\ol{U}_{\f{\rho_0}{4}}}\dist(G_k(x),\cN)\)=0.
\]
For sufficiently large $k$, we have the following:
\begin{itemize}
\item $\Pi_{\cN}\circ G_k\in C^1(U_{\f{\rho_0}{4}},\cN)$;
\item $[\Pi_{\cN}\circ G_k|_{\pa((1+\f{\rho_0}{2})D_i)}]_{\cN}=\sg_i$ for $ i\in\Z\cap[1,N] $ and $[\Pi_{\cN}\circ G_k|_{\pa((1-\f{\rho_0}{2})D_0)}]_{\cN}=\sg_0$.
\end{itemize}
Set $H:=\Pi_{\cN}\circ G_k\in C^1(\ol{U}_{\f{\rho_0}{2}},\cN)$. Let $\Phi:\ol{U}\to\ol{U}_{\f{\rho_0}{2}}$ be a bi-Lipschitz map. Then $G:=H\circ\Phi$ satisfies the required properties.
\end{proof}

\subsection{Extension theory}

In this section, we present some extension results that will be used later to construct suitable competitors. Let $ \wt{\cN} $ be the universal covering space of $ \cN $, and let $ \pi:\wt{\cN}\to\cN $ be the covering map. Recall that $ \cN $ is a smooth compact Riemannian manifold, isometrically embedded in $ \R^{\ell} $, and that its fundamental group $ \pi_1(\cN) $ is finite. Hence, $ \wt{\cN} $ is also smooth and compact. We equip $\wt{\cN}$ with the pullback metric $ \wt{g}_{\wt{\cN}}:=\pi^*g_{\cN} $, defined by
\be
\wt{g}_{\wt{\cN}}|_x(v,w)
= g_{\cN}|_{\pi(x)}(\ud\pi_x(v),\ud\pi_x(w)), \quad \text{for any } v,w\in T_x\wt{\cN}.\label{ComparisonofMetics}
\ee
Then the covering map $\pi:(\wt{\cN},\wt{g}_{\wt{\cN}})\to(\cN,g_{\cN})$ is a local isometry.

\begin{lem}\label{ExtensionLemma1}
Let $ r>0 $, and let $ g\in H^1(\pa B_r^2,\cN) $ be homotopically trivial. Then there exists $ u\in H^1(B_r^2,\cN) $ such that $ u|_{\pa B_r^2}=g $ and
\be
\|\na u\|_{L^2(B_r^2)}^2 \leq Cr\|\na_{\pa B_r^2}g\|_{L^2(\pa B_r^2)}^2,
\label{ExtensionLemma1eq}
\ee
where $ C>0 $ depends only on $ \cN $.
\end{lem}

\begin{proof}
By Sobolev embedding, $ H^1(\pa B_r^2,\cN)\subset C^{0,\f{1}{2}}(\pa B_r^2,\cN) $. Since $ g $ is homotopically trivial, there exists a lift $ \wt{g}\in C^0(\pa B_r^2,\wt{\cN}) $ such that $ g=\pi\circ\wt{g} $. Using \eqref{ComparisonofMetics}, we obtain $ \wt{g}\in H^1(\pa B_r^2,\wt{\cN}) $. The conclusion now follows from \cite[Lemma 2.13]{BSV25}.
\end{proof}

\begin{lem}\label{ExtensionLemma2}
Let $ r>0 $ and $ g\in H^1(\pa B_r,\cN) $, with $B_r = B_r^n$ and $n\geq 3$. Then there exists $ u\in H^1(B_r,\cN) $ such that $ u|_{\pa B_r}=g $, and
\[
\|\na u\|_{L^2(B_r)}^2 \leq Cr^{\f{n-1}{2}}\|\na_{\pa B_r}g\|_{L^2(\pa B_r)},
\]
where $ C>0 $ depends only on $ \cN $ and $ n $.
\end{lem}

\begin{proof}
By scaling, it suffices to consider the case $ r=1 $. Define $ U:=B_{\f{11}{10}}\backslash \ol{B}_{\f{9}{10}} $ and define $ \wt{g}\in H^1(U,\cN) $ by $ \wt{g}(x):=g(\f{x}{|x|}) $. Since $ g\in H^1(\pa B_1,\cN) $, it follows that $ \wt{g}\in H^1(U,\cN) $.

Since $ U\subset\R^n $ is a smooth bounded domain and is simply connected for $ n\geq 3 $, it follows from \cite[Theorem 1]{BC07} that there exists a lift $ \wt{\vp}:U\to\wt{\cN} $ such that $ \wt{g}=\pi\circ\wt{\vp} $. Note that $ \wt{g} $ is homogeneous by construction, and hence $ \wt{\vp} $ is also homogeneous. More precisely, for any $ x\in U $ and $ \lda>0 $ such that $ \lda x\in U $, we have $ \wt{\vp}(\lda x)=\wt{\vp}(x) $. Therefore, the trace $ \vp:=\wt{\vp}|_{\pa B_1} $ is well-defined and belongs to $ H^1(\pa B_1,\wt{\cN}) $, with $ g=\pi\circ\vp $. The result now follows from \cite[Lemma 2.15]{BSV25}.
\end{proof}

\subsection{A technical lemma on homotopy classes for \texorpdfstring{$ H^1 $}{} maps}

We now study the homotopy behavior of the $ H^1 $-maps defined on a cylinder~$ \Ss^1\times B_1^k $. It follows from Fubini's theorem and Sobolev embedding that for $ \HH^k $-a.e. $ x\in B_1^k $, the slice $ u(\cdot,x) $ defines a continuous map from $ \Ss^1 $ to $ \cN $, and hence determines a homotopy class in $ [\Ss^1,\cN] $. A natural question is whether this homotopy class depends on $ x $. For homotopy classes that can be represented by topological degree, this result follows from~\cite[Theorem~1]{BLMN99}.

The next lemma gives a negative answer to this question in greater generality and provides a well-defined homotopy class associated with $ u $, which will be used in the sequel.

\begin{lem}\label{H1homotopyclass}
Assume that $ u\in H^1(\Ss^1\times B_1^k,\cN) $, $ k\in\Z_+ $. Then there exists $ \sg\in[\Ss^1,\cN] $ such that for $ \HH^{k} $-a.e. $ x\in B_1^k $, $ [u(\cdot,x)]_{\cN}=\sg $.
\end{lem}


\begin{rem}\label{Remarkkeq1}
For $ k=1 $, the conclusion follows from arguments similar to \cite[Section 2.2]{Can17}, where the key tool is the density of $ C^{\infty}(\cX,\cN) $ in $ H^1(\cX,\cN) $ for any smooth two-dimensional manifold $ \cX $ with $ C^1 $ boundary (see \cite[Proposition p. 267]{SU83}). When $ k\geq 2 $, such an approximation argument is no longer available in a direct way. Therefore, we adopt a different approach.
\end{rem}

\begin{proof}[Proof of Lemma \ref{H1homotopyclass}]
By Remark \ref{Remarkkeq1}, it suffices to consider the case $ k\geq 2 $. We write points in $ \Ss^1\times B_1^k $ as $ (\theta,x) $, and denote by $ \na $ the derivatives with respect to $ (\theta,x) $. By a standard covering and translation argument, it is enough to prove that there exists $ \sg\in[\Ss^1,\cN] $ such that for $ \HH^{k} $-a.e. $ x\in B_{\f{1}{2}}^k $, one has $ [u(\cdot,x)]_{\cN}=\sg $. Using the Lebesgue differentiation theorem, for $ \HH^k $-a.e. $ x\in B_1^k $,
\be
\lim_{\delta\to 0^+}\f{1}{\delta^k}\int_{B_{\delta}^k(x)}\int_{\Ss^1}(|\na u(\theta,y)-\na u(\theta,x)|^2+|u(\theta,y)-u(\theta,x)|^2)\ud\HH^1(\theta)\ud\HH^k(y)=0.\label{defLebesgue}
\ee
Define
\[
I:=\{x\in B_1^k:x\text{ satisfies }\eqref{defLebesgue}\}.
\]
Then $ \HH^k(B_1^k\backslash I)=0 $. For $ x\in B_{\f{1}{2}}^k $ and $ \eta\in(0,\f{1}{10}) $, define
\begin{align*}
I_x&:=\left\{v\in\Ss^{k-1}:\liminf_{\delta\to 0^+}\(\f{1}{\delta}\int_{(B_{\delta}^k(x)\backslash B_{\f{\delta}{2}}^k(x))\cap\ell_{v}(x)}\int_{\Ss^1}(|\na u-\na u(\cdot,x)|^2+|u-u(\cdot,x)|^2)\ud\HH^1\ud\HH^1\)=0\right\},\\
I_{x,\eta}&:=\left\{v\in\Ss^{k-1}:\int_{(B_1^k\backslash B_{\eta}^k(x))\cap\ell_{v}(x)}\int_{\Ss^1}(|\na u|^2+|u|^2)\ud\HH^1\ud\HH^1<+\ift\right\} \! ,
\end{align*}
where $ \ell_v(x)=\{y\in\R^k:y=x+tv,\,\,t>0\} $.

\smallskip
\noindent\textbf{Claim 1.} Let $ x\in B_{\f{1}{2}}^k\cap I $. Then $ v\in I_x $ for $ \HH^{k-1} $-a.e. $ v\in\Ss^{k-1} $.

Direct computations imply that
\begin{align*}
&\f{1}{\delta}\int_{\Ss^{k-1}}\(\int_{\Ss^1}\int_{(B_{\delta}^k(x)\backslash B_{\f{\delta}{2}}^k(x))\cap\ell_v(x)}(|\na u-\na u(\cdot,x)|^2+|u-u(\cdot,x)|^2)\ud\HH^1\ud\HH^1\)\ud\HH^{k-1}(v)\\
&\quad\quad=\f{1}{\delta}\int_{\Ss^{k-1}}\left[\int_{\Ss^1}\(\int_{\f{\delta}{2}}^{\delta}(|\na u(\cdot,x+tv)-\na u(\cdot,x)|^2+|u(\cdot,x+tv)-u(\cdot,x)|^2)\ud t\)\ud\HH^1\right]\ud\HH^{k-1}(v)\\
&\quad\quad=\f{1}{\delta}\int_{\Ss^1}\(\int_{B_{\delta}^k(x)\backslash B_{\f{\delta}{2}}^k(x)}\f{1}{|y-x|^{k-1}}(|\na u(\cdot,y)-\na u(\cdot,x)|^2+|u(\cdot,y)-u(\cdot,x)|^2)\ud\HH^k(y)\)\ud\HH^1\\
&\quad\quad\leq\f{C}{\delta^{k}}\int_{\Ss^1}\(\int_{B_{\delta}^k(x)}(|\na u-\na u(\cdot,x)|^2+|u-u(\cdot,x)|^2)\ud\HH^k\)\ud\HH^1\to 0^+
\end{align*}
as $ \delta\to 0^+ $, since $ x\in I $. It follows from Fatou's lemma that
\[
\int_{\Ss^{k-1}}\liminf_{\delta\to 0^+}\left[\f{1}{\delta}\(\int_{\Ss^1}\int_{(B_{\delta}^k(x)\backslash B_{\f{\delta}{2}}^k(x))\cap\ell_v(x)}(|\na u-\na u(\cdot,x)|^2+|u-u(\cdot,x)|^2)\ud\HH^1\ud\HH^1\)\right]\ud\HH^{k-1}(v)=0,
\]
which proves the claim.

\smallskip
\noindent\textbf{Claim 2.} Let $ x\in B_{\f{1}{2}}^k\cap I $ and $ \eta\in(0,\f{1}{10}) $. Then $ v\in I_{x,\eta} $ for $ \HH^{k-1} $-a.e. $ v\in\Ss^{k-1} $.

Since $ x\in I $, we have
\begin{align*}
&\int_{\Ss^{k-1}}\left[\int_{\Ss^1}\(\int_{(B_1^k\backslash B_{\eta}^k(x))\cap\ell_v(x)}(|\na u|^2+|u|^2)\ud\HH^1\)\ud\HH^1\right]\ud\HH^{k-1}(v)\\
&\quad\quad\leq\f{C}{\eta^{k-1}}\int_{\Ss^1}\(\int_{B_1^k\backslash B_{\eta}^k(x)}(|\na u|^2+|u|^2)\ud\HH^k\)\ud\HH^1<+\ift.
\end{align*}
Then for $ \HH^{k-1} $-a.e. $ v\in\Ss^{k-1} $,
\[
\int_{\Ss^1}\(\int_{(B_1^k\backslash B_{\eta}^k(x))\cap\ell_v(x)}(|\na u|^2+|u|^2)\ud\HH^1\)\ud\HH^1<+\ift,
\]
implying this claim.

\smallskip
\noindent\textbf{Completion of the proof.} Let $ x_1,x_2\in B_{\f{1}{2}}^k\cap I $. We show that
\be
[u(\cdot,x_1)]_{\cN}=[u(\cdot,x_2)]_{\cN}.\label{ux1ux2samehomotopy}
\ee
Using the above two claims, we choose $ x_3\in B_{\f{1}{2}}^k\backslash\{x_1,x_2\} $ such that for any $ j\in\{1,2\} $,
\be
\f{x_3-x_j}{|x_3-x_j|}\in I_{x_j}\cap\bigcap_{i=1}^{+\ift}I_{x_j,i^{-1}}.\label{choosex3}
\ee
Given points~$x\in B^k$, $y\in B^k$, we will denote by~$\overline{xy}$ the straight line segment of endpoints~$x$, $y$. Fix $ \eta'>0 $. Since $ \frac{x_3-x_j}{|x_3-x_j|}\in I_{x_j} $, we can choose
\[
\delta=\delta(\eta')\in\(0,\f{1}{10}\min\{\HH^1(\ol{x_1x_3}),\HH^1(\ol{x_2x_3}),1\}\)
\]
sufficiently small such that
\be
\int_{\Ss^1}\(\f{1}{\delta}\int_{(B_{\delta}^k(x_j)\backslash B_{\f{\delta}{2}}^k(x_j))\cap\ol{x_jx_3}}(|\na u-\na u(\cdot,x_1)|^2+|u-u(\cdot,x_1)|^2)\ud\HH^1\)\ud\HH^1<\f{\eta'}{100}\label{etaprimeuse}
\ee
for any $ j\in\{1,2\} $. Choose $ x_1',x_1''\in\ol{x_1x_3} $ and $ x_2',x_2''\in\ol{x_2x_3} $ such that
\begin{align*}
\HH^1(\ol{x_1x_1'})=\HH^1(\ol{x_2x_2'})=\f{\delta}{2},\quad\HH^1(\ol{x_1x_1''})=\HH^1(\ol{x_2x_2''})=\delta.
\end{align*}
From \eqref{etaprimeuse}, there exists $ S_j\subset\ol{x_j'x_j''} $ with $ \HH^1(S_j)>0 $ such that for any $ y\in S_j $,
\[
\|u(\cdot,y)-u(\cdot,x_j)\|_{H^1(\Ss^1,\R^m)}<\eta'.
\]
Choosing $ \eta'>0 $ small enough (depending only on $ \cN $), we ensure that for $ \HH^1 $-a.e. $ y\in S_j $, the map $ u(\cdot,y)\in C^0(\Ss^1,\cN) $ and $ [u(\cdot,y)]_{\cN}=[u(\cdot,x_1)]_{\cN} $.

Let $ i_0\in\Z_+\cap[\f{4}{\delta},+\ift) $. Then
\[
\ol{x_1'x_3}\subset B_1^k\backslash B_{i_0^{-1}}(x_1),\quad\ol{x_2'x_3}\subset B_1^k\backslash B_{i_0^{-1}}(x_2).
\]
It follows from \eqref{choosex3} and the definition of $ I_{x,\eta} $ that
\[
\int_{\Ss_1}\(\int_{\ol{x_1'x_3}\cup\ol{x_2'x_3}}(|\na u|^2+|u|^2)\ud\HH^1\)\ud\HH^1<+\ift.
\]
Applying the result for $ k=1 $ to $ u|_{\ol{x_1'x_3}\cup\ol{x_2'x_3}} $, we find $ \sg\in[\Ss^1,\cN] $ such that for $ \HH^1 $-a.e. $ y\in\ol{x_1'x_3}\cup\ol{x_2'x_3} $, $ [u(\cdot,y)]_{\cN}=\sg $. Since $ \HH^1(S_1),\HH^1(S_2)>0 $, this implies \eqref{ux1ux2samehomotopy}.
\end{proof}

\subsection{Lower bound on the energy}

Next, we recall a lower bound for the Ginzburg–Landau-type energy in two-dimensional balls. This estimate goes back to the case $ \cN=\Ss^1 $, where it was first obtained in \cite{San98, Jer99}. It shows that non-trivial topology forces a logarithmic energy cost and provides a lower bound on the energy in terms of the homotopy class.

\begin{lem}\label{LowerBound}
Let $ \va\in(0,\f{r}{2}) $ and $ u\in H^1(B_r^2,\R^m) $. Assume that $ g=u|_{\pa B_r^2}\in H^1(\pa B_r^2,\R^m) $ and $ \dist(g,\cN)<\delta_{\cN} $ on $ \pa B_r^2 $. Then $ \sg:=[\Pi_{\cN}\circ g]_{\cN} $ is well-defined and
\[
E_{\va}(u,B_r^2)+CrE_{\va}(g,\pa B_r^2)\geq|\sg|_*\log\f{r}{\va}-C,
\]
where $ C>0 $ depends only on $ f $ and $ \cN $.
\end{lem}
\begin{proof}
Using a scaling argument, we may assume that $ r=1 $. The conclusion then follows from Remark \ref{remequvalence} and \cite[Corollary 6.8]{MRS21}.
\end{proof}

\subsection{Upper bound on the energy}
In this subsection, we construct suitable maps that yield an upper bound for the energy in terms of the topological class of the boundary data. We begin with the two-dimensional case. The following lemma extends \cite[Lemma 28]{Can17}, which treats the Landau--de Gennes model, and can also be viewed as another form of \cite[Proposition 5.1]{MRS21}.

\begin{lem}\label{upperboundleast}
Let $ r>0 $. There exists a constant $ C>0 $, depending only on $ f $ and $ \cN $, such that the following property holds. For any $ \va\in(0,r) $ and any $ g\in H^1(\pa B_r^2,\cN) $, there exists $ u_{\va}\in H^1(B_r^2,\R^m) $ such that $ u_{\va}|_{\pa B_r^2}=g $ and
\be
E_{\va}(u_{\va},B_r^2)\leq|\sg|_*\log\f{r}{\va}+C\big(r\|\na_{\pa B_r^2}g\|_{L^2(\pa B_r^2)}^2+1\big),\label{upperboundleast1}
\ee
where $ \sg:=[g]_{\cN} $.
\end{lem}

To prove the above lemma, we first establish the following auxiliary result.

\begin{lem}\label{Eminlemma}
Let $ \va\in(0,1) $, and $ \sg\in[\Ss^1,\cN] $. Assume that $ \ell\in H^1(\Ss^1,\cN) $ is a geodesic loop such that
\[
E_{\op{min}}(\sg)=\int_{\Ss^1}|\ell'|^2,\quad [\ell]_{\cN}=\sg.
\]
Then there exists $ w_{\va}\in H^1(B_1^2,\R^m) $ such that $ w_{\va}|_{\pa B_1^2}=\ell $ and
\[
E_{\va}(w_{\va},B_1^2)\leq E_{\op{min}}(\sg)|\log\va|+C,
\]
where $ C>0 $ depends only on $ f $ and $ \cN $.
\end{lem}
\begin{proof}
We use polar coordinates $ (\rho,\theta)\in[0,1]\times\Ss^1 $ in $ B_1^2 $. Define
\[
\eta_{\va}:=\left\{\begin{aligned}
&1&\text{ if }&\rho\in[\va,1],\\
&\va^{-1}\rho&\text{ if }&\rho\in[0,\va).
\end{aligned}\right.
\]
Set $ w_{\va}(\rho,\theta)=\eta_{\va}(\rho)\ell(\theta) $. Then
\begin{align*}
&|\na w_{\va}|\leq C\va^{-1}\text{ in }B_{\va}^2,\quad|\na w_{\va}|=\rho^{-2}|\pa_{\theta}\ell|^2\text{ in }B_1\backslash B_{\va}^2,\quad\text{and}\quad 0\leq f(w_{\va})\leq C\chi_{B_{\va}^2},
\end{align*}
where $ \chi_{B_{\va}^2} $ is the characteristic function in $ B_{\va}^2 $. A direct computation gives
\begin{align*}
E_{\va}(w_{\va},B_1^2)&\leq\int_{B_{\va}^2}\f{C}{\va^2}\ud x+\int_{\va}^1\(\int_{\Ss^1}\rho^{-1}|\pa_{\theta}\ell|^2\)\ud\HH^1(\theta)\ud\rho\\
&\leq E_{\op{min}}(\sg)|\log\va|+C.
\end{align*}
This completes the proof.
\end{proof}

As above, we denote by~$\overline{xy}$ the straight line segments between two points~$x$, $y$. For a loop $ \ell\in C^0(\Ss^1,\cN) $ passing through $ x_0 \in\cN$, we denote by $ [\ell]_{\cN,x_0} $ its (based) homotopy class in $ \pi_1(\cN,x_0) $.
We denote by~$*$ the composition of paths: Given two paths~$\ell_1\colon [0, 1]\to\cN$, $\ell_2\colon[0, 1]\to\cN$ with~$\ell_1(1) = \ell_2(0)$, the composition~$\ell_1*\ell_2$ is the path obtained by first following $\ell_1$ and then $\ell_2$, that is,
\[
 (\ell_1*\ell_2)(t) :=
 \begin{cases}
  \ell_1(2t) &\textrm{if } 0 \leq t < \frac{1}{2} \\
  \ell_2(2t - 1) &\textrm{if } \frac{1}{2} \leq t \leq 1.
 \end{cases}
\]
Given~$\ell\colon [0,1]\to\cN$, we denote by~$\widetilde{\ell}$ the path with the opposite orientation, namely~$\widetilde{\ell}(t) := \ell(1 - t)$ for~$t\in [0, 1]$.

\begin{proof}[Proof of Lemma \ref{upperboundleast}]
By scaling, we may assume that $ r=1 $. It follows from \eqref{c0geq} and the finiteness of $ \pi_1(\cN) $ that there exist elements $ \{\sg_i\}_{i=1}^N $ such that $ \sg\in\sum_{i=1}^N\sg_i $ and
\be
|\sg|_*=\sum_{i=1}^NE_{\op{min}}(\sg_i),\quad N\leq C(\cN).\label{sgCcN}
\ee
Let $ D:=B_{\f{1}{2}}^2 $. Using \eqref{sgCcN}, we choose $ \rho=\rho(\cN)\in(0,\f{1}{10}) $ and points $ \{x_i\}_{i=1}^N\subset D $ such that
\begin{align*}
\dist(B_{\rho}(x_i),B_{\rho}(x_j))&\geq\rho,\quad\text{ for any }i,j\in\Z\cap[1,N],\,\,i\neq j,\\
\dist(B_{\rho}(x_i),\pa D)&\geq\rho,\quad\text{ for any }i\in\Z\cap[1,N].
\end{align*}
Set $ U:=D\backslash\cup_{i=1}^N\ol{B}_{\rho}^2(x_i) $. It follows from Lemma \ref{propextension} that there exists $ G\in C^1(\ol{U},\cN) $ such that the following properties hold.
\begin{itemize}
\item $ G $ satisfies the inequality
\be
\|\na G\|_{L^{\ift}(U)}\leq C(f,\cN).\label{nablaGfcN}
\ee
\item $ [G|_{\pa D}]_{\cN}=\sg_0 $ and $ [G|_{\pa B_{\rho}^2(x_i)}]_{\cN}=\sg_i $ for each $ i\in\Z\cap[0,N] $.
\end{itemize}
We will define the map~$u_\va$ to be equal to~$G$ on~$U$, and extend it inside each hole~$B_\rho(x_i)$ by making use of Lemma~\ref{Eminlemma}. However, we need to define~$u_\va$ in the annulus~$B^2\backslash D$ in such a way that it matches the boundary condition~$g$ on~$\partial B^2$. To this end, we consider the unit square $ [0,1]^2\subset\R^2 $. For simplicity, let
\[
a_1:=(0,0), \quad a_2:=(1,0), \quad a_3:=(1,1), \quad a_4:=(0,1).
\]
We first construct a suitable boundary datum~$\phi\colon\partial[0,1]^2\to\cN$
that is homotopic to a constant, extend it to a map~$[0, 1]^2\to\cN$, and finally obtain the desired extension~$B^2\backslash D\to\cN$ by identifying the edges~$\overline{a_2a_3}$ and~$\overline{a_4a_1}$ of the square.

The boundary datum~$\phi\colon \partial[0,1]^2\to\cN$ is defined as follows. On the edge~$\overline{a_1a_2}$, we define~$\phi$ as a reparameterization of~$g$, that is, $\phi(x_1, 0) := g(\exp(2\pi i x_0))$ for~$x_1\in [0, 1]$. Note that~$g$ is a loop in~$\cN$ based at the point~$x_0 := \phi(a_1) = \phi(a_2)$, and that~$[g]_{\cN,x_0}\in\pi_1(\cN,x_0)$. On the edge~$\overline{a_3a_4}$, we define~$\phi$ as a reparameterization of~$g_0 := G|_{\pa D}$, that is $\phi(x_1, 1) = g_0(\exp(2\pi i x_1))$ for~$x_1\in [0, 1]$.
The map~$g_0$ is a loop based at the point~$y_0 := \phi(a_3) = \phi(a_4)$. Choose a continuous curve $ c_1:[0,1]\to\cN $ such that $ c_1(0)=x_0 $ and $ c_1(1)=y_0 $, and let 
\[
h_0:=c_1 * g_0 * \wt{c}_1.
\]
The map~$h_0$ is a loop based at~$x_0$. Since $ [g_0]_{\cN}=[g]_{\cN}=\sg $, the elements $ [h_0]_{\cN,x_0} $ and $ [g]_{\cN,x_0} $ lie in the same conjugacy class in $ \pi_1(\cN,x_0) $. Therefore, there exists a loop $ c_2\in C^1(\Ss^1,\cN) $ based at $ x_0 $ such that $ x_0\in c_2 $ and
\be
[c_2]_{\cN,x_0}[h_0]_{\cN,x_0}[c_2]_{\cN,x_0}^{-1}=[g]_{\cN,x_0}.
\label{h0property}
\ee
On the segments $ \ol{a_4a_1} $ and $ \ol{a_2a_3} $, we define~$\phi(0, x_2) := (\wt{c}_1*\wt{c}_2)(x_2)$ and~$\phi(1, x_2) := (c_2*c_1)(x_2)$, for~$x_2\in [0, 1]$, respectively. By \eqref{h0property}, the resulting boundary datum~$\phi\colon\pa[0,1]^2\to\cN$ is homotopically trivial. Hence, by Lemma \ref{ExtensionLemma1}, there exists $ v\in H^1([0,1]^2,\cN) $ such that $v|_{\partial[0,1]^2}=\phi$. Moreover, \eqref{ExtensionLemma1eq} yields
\be
\|\na v\|_{L^2([0,1]^2)}\leq C(\|\na_{\pa B_1^2}g\|_{L^2(\pa B_1^2)}+1).\label{nablavestimate}
\ee
For each $ i\in\Z\cap[1,N] $, we apply Lemma \ref{Eminlemma} in $ B_{\rho}^2(x_i) $ to construct $ w_{\va}^i\in H^1(B_{\rho}^2(x_i),\R^m) $ such that
\be
E_{\va}(w_{\va}^i,B_{\rho}^2(x_i))\leq E_{\op{min}}(\sg_i)|\log\va|+C.
\label{EwvaiestimateEmin}
\ee
Let $ \Phi:B_1^2\backslash D\to[0,1]^2 $ be a Lipschitz map identifying the edges $ \ol{a_2a_3} $ and $ \ol{a_1a_4} $. Define $ u_{\va}:B_1^2\to\R^m $ by
\[
u_{\va}(x)=\left\{\begin{aligned}
&w_{\va}^i(x)&\text{ if }&x\in B_{\rho}^2(x_i),\\
&G(x)&\text{ if }&x\in U,\\
&(v\circ\Phi)(x)&\text{ if }&x\in B_1^2\backslash D.
\end{aligned}\right.
\]
Then the estimate \eqref{upperboundleast1} follows from \eqref{sgCcN}, \eqref{nablaGfcN}, \eqref{nablavestimate}, and \eqref{EwvaiestimateEmin}.
\end{proof}

For $ r,L>0 $, we define
\be
\Lda_{r,L}:=B_r^2\times B_L^{n-2},\quad\Ga_{r,L}:=\pa B_r^2\times B_L^{n-2}.\label{LdaGarLdef}
\ee
Let $ u\in H^1(\Ga_{r,L},\cN) $. By Lemma \ref{H1homotopyclass}, there exists $ \sg\in[\Ss^1,\cN] $ such that for $ \HH^{n-2} $-a.e. $ x\in B_L^{n-2} $, we have $ u(\cdot,x)\in C^0(\Ss^1,\cN) $ and $ [u(\cdot,x)]_{\cN}=\sg $. We denote this homotopy class by $ [u]_{\cN} $. The following result extends Lemma \ref{Eminlemma} to the cylindrical domain $ \Lda_{r,L} $.

\begin{lem}\label{cylinderextension1}
There exists a constant $ C>0 $ depending only on $ f,\cN $, and $ n $, such that the following properties hold. For any $ \va\in(0,r) $ and any $ g\in H^1(\Ga_{r,L},\cN) $ with $ \sg:=[g]_{\cN} $, there exists $ u_{\va}\in H^1(\Lda_{r,L},\R^m) $ such that $ u_{\va}=g $ on $ \Ga_{r,L} $,
\be
E_{\va}(u_{\va},\Lda_{r,L})\leq CL\(\f{L}{r}+\f{r}{L}\)\|\na_{\Ga_{r,L}}g\|_{L^2(\Ga_{r,L})}^2+|\sg|_*\HH^{n-2}(B_L^{n-2})\log\f{r}{\va}+CL^{n-2},
\label{LdarLestimate}
\ee
and
\be
E_{\va}(u_{\va},B_r^2\times\pa B_L^{n-2})\leq C\(\f{L}{r}+\f{r}{L}\)\|\na_{\Ga_{r,L}}g\|_{L^2(\Ga_{r,L})}^2+|\sg|_*\HH^{n-3}(\pa B_L^{n-2})\log\f{r}{\va}+CL^{n-3}.
\label{LdarLestimate1}
\ee
\end{lem}
\begin{proof}
By scaling, we may assume $ r=1 $. Let $ Q=(-L,L)^{n-2} $, and let $ \Phi:\ol{B}_L^{n-2}\to\ol{Q} $ be a  bi-Lipschitz map (that is, an invertible Lipschitz map with a Lipschitz inverse) such that $ \Phi(0^{n-2})=0^{n-2} $ and
\be
\|\na\Phi\|_{L^{\infty}(\ol{B}_L^{n-2})}+\|\na(\Phi^{-1})\|_{L^{\infty}(\ol{Q})}\leq C(n).
\label{nablaPhiestimate}
\ee
Let $ U\subset\R^{n-2} $ be a domain. We use polar coordinates $ (\rho,\theta,y')\in[0,1]\times\Ss^1\times U $ for the points in $ B_1^2\times U $.

By Lemma \ref{H1homotopyclass} and Fubini's theorem, we may choose
$y_0'\in B_{\f{L}{4}}^{n-2}$ such that $ g(1,\cdot,y_0')\in H^1(\Ss^1,\cN) $, $ [g(1,\cdot,y_0')]_{\cN}=\sg $, and
\be
\int_{\pa B_1^2}|\pa_{\theta}g(1,\cdot,y_0')|^2\ud\HH^1\leq\f{C}{L^{n-2}}\|\na_{\Ga_{1,L}}g\|_{L^2(\Ga_{1,L})}^2.
\label{gy0estimate}
\ee
After translating the $y'$-coordinates, we assume that $y_0'=0^{n-2}$.
We keep the same notation for the translated variables. In these coordinates, we define $\wt{g}:\pa B_1^2\times Q\to\cN$ by
\[
\wt{g}(1,\theta,y'):=g(1,\theta,\Phi^{-1}(y')),
\qquad (\theta,y')\in \pa B_1^2\times Q.
\]

We also choose the translated coordinates so that the coordinate slices through the origin are good. More precisely, for $j=0$, set $Q_0:=Q$, and for
$1\leq j\leq n-2$, set
\[
Q_j:=\{y'\in Q:\ y_1=...=y_j=0\}.
\]
By applying Fubini's theorem successively in the coordinate variables, we may also assume that, for every $j\in\Z\cap[1,n-2]$,
\be
L^j\int_{\pa B_1^2\times Q_j}(|\pa_\theta\wt{g}|^2+|\na_{y'}\wt{g}|^2)
\ud\HH^{n-1-j}\leq C(n)\int_{\pa B_1^2\times Q}
(|\pa_\theta\wt{g}|^2+|\na_{y'}\wt{g}|^2)\ud\HH^1
\ud y'.\label{Fubinimore}
\ee

Divide the radial interval $ [\f{1}{2},1] $ into $ n-2 $ equal sub-intervals and define
\[
r_j:=1-\f{j}{2n-4},\quad\text{for }j\in\Z\cap[0,n-2].
\]
For $ \rho\in[\f{1}{2},1] $ and $ y'=(y_1,\dots,y_{n-2})\in Q $, define
\[
\Psi(\rho,y'):=(\Psi_1(\rho,y_1),...,\Psi_{n-2}(\rho,y_{n-2}))
\]
with
\[
\Psi_j(\rho,y_j)=\left\{\begin{aligned}
&y_j&\text{ if }&\rho\geq r_{j-1},\\
&\op{sgn}(y_j)\max(0,|y_j|-(2n-4)L(r_{j-1}-\rho))&\text{ if }&\rho\in[r_j,r_{j-1}],\\
&0&\text{ if }&\rho\in[0,r_j].
\end{aligned}\right.
\]
Thus, on the interval $[r_j,r_{j-1}]$, only the $j$-th coordinate is moved, while the previous coordinates have already been collapsed to $0$. 

We define the extension $\wt u$ of $\wt g$ in $ (B_1^2\backslash B_{\f{1}{2}}^2)\times Q $ by
\[
\wt{u}(\rho,\theta,y')=\wt{g}(1,\theta,\Psi(\rho,y')).
\]
Since $\wt g$ is $\cN$-valued, the potential term vanishes for $\wt u$.
Using the chain rule, the definition of $\Psi$, together with \eqref{nablaPhiestimate}--\eqref{Fubinimore}, we obtain
\be
\begin{aligned}
E_{\va}(\wt{u},(B_1^2\backslash B_{\f{1}{2}}^2)\times Q)&=\f{1}{2}\int_{\f{1}{2}}^1\rho\ud\rho\int_{\Ss^1}\ud\HH^1(\theta)\int_{Q}(\rho^{-2}|\pa_{\theta}\wt{u}|^2+|\na_{y'}\wt{u}|^2+|\pa_{\rho}\wt{u}|^2)\ud y'\\
&\leq C\int_{\Ss^1}\(\int_{B_L^{n-2}}(|\pa_{\theta}g|^2+(L^2+1)|\na_{y'}g|^2)\ud y'\)\ud\HH^1(\theta)\\
&\leq C\int_{\Ga_{1,L}}(|\na_{\Ga_{1,L}}g|^2+|\na_{\Ga_{1,L}}g|^2+L^2|\na_{\Ga_{1,L}}g|^2)\ud\HH^{n-1}\\
&\leq CL(L+L^{-1})\|\na_{\Ga_{1,L}}g\|_{L^2(\Ga_{1,L})}^2.
\end{aligned}\label{EvawtPestimate1}
\ee
Indeed, on $[r_j,r_{j-1}]$, the image of $\Psi$ is contained in $Q_{j-1}$,
and only the $j$-th coordinate varies. The factors produced by the already
collapsed variables are controlled by \eqref{Fubinimore}, while
$|\pa_\rho\Psi_j|\leq CL$ gives the additional factor $L^2$ for the second inequality of \eqref{EvawtPestimate1}. From the construction of $ \wt{u} $, we have
\be
\wt{u}\(\f{1}{2},\theta,y'\)=\wt{g}(1,\theta,0^{n-2}),\quad(\theta,y')\in \Ss^1\times Q.\label{boundarywtu}
\ee
Given \eqref{gy0estimate}, we apply Lemma \ref{upperboundleast} to obtain a map $ v_{\va}\in H^1(B_{\f{1}{2}}^2,\R^m) $ such that
\[
v_{\va}\(\f{1}{2},\theta\)=g(1,\theta,0^{n-2}),\quad\theta\in\Ss^1,
\]
and
\be
\begin{aligned}
E_{\va}(v_{\va},B_{\f{1}{2}}^2)&\leq|\sg|_*|\log\va|+C(\|\pa_{\theta}g(1,\cdot,0^{n-2})\|_{L^2(\pa B_1^2)}^2+1)\\
&\stackrel{\eqref{gy0estimate}}{\leq}|\sg|_*|\log\va|+C(L^{2-n}\|\na_{\Ga_{1,L}}g\|_{L^2(\Ga_{1,L})}^2+1).
\end{aligned}\label{Evavvaestimate}
\ee
Define
\[
u_{\va}(\rho,\theta,y'):=\left\{\begin{aligned}
&v_{\va}(\rho,\theta)&\text{ if }&(\rho,\theta,y')\in B_{\f{1}{2}}^2\times B_L^{n-2},\\
&\wt{u}(\rho,\theta,\Phi(y'))&\text{ if }&(\rho,\theta,y')\in (B_1^2\backslash B_{\f{1}{2}}^2)\times B_L^{n-2}.
\end{aligned}\right.
\]
By \eqref{boundarywtu}, we have $ u_{\va}\in H^1(\Lda_{r,L},\R^m) $. Combining \eqref{nablaPhiestimate} and \eqref{EvawtPestimate1}, we obtain
\be
E_{\va}(u_{\va},(B_1^2\backslash B_{\f{1}{2}}^2)\times B_L^{n-2})\leq CL(L+L^{-1})\|\na_{\Ga_{1,L}}g\|_{L^2(\Ga_{1,L})}^2.\label{EvawtPestimate}
\ee
From \eqref{Evavvaestimate}, we deduce
\[
E_{\va}(u_{\va},B_{\f{1}{2}}^2\times B_L^{n-2})\leq |\sg|_*\HH^{n-2}(B_L^{n-2})|\log\va|+C\|\na_{\Ga_{1,L}}g\|_{L^2(\Ga_{1,L})}^2+CL^{n-2}
\]
Together with \eqref{EvawtPestimate}, this yields \eqref{LdarLestimate}.

Moreover, from the definition of $ u_{\va} $ on $ B_{\f{1}{2}}^2\times B_L^{n-2} $, we obtain
\be
E_{\va}(u_{\va},B_{\f{1}{2}}^2\times\pa B_L^{n-2})\leq |\sg|_*\HH^{n-3}(\pa B_L^{n-2})|\log\va|+C(L^{-1}\|\na_{\Ga_{1,L}}g\|_{L^2(\Ga_{1,L})}^2+L^{n-3}).\label{toplow0}
\ee
For any $ k\in\Z\cap[1,n-2] $, let
\[
F_k^+:=\{y'\in\pa Q:\ y_k=L\},
\quad
F_k^-:=\{y'\in\pa Q:\ y_k=-L\}.
\]
We claim that
\be
E_{\va}(\wt{u},(B_1^2\backslash B_{\f{1}{2}}^2)\times F_k^{\pm})\leq C(L+L^{-1})\|\na_{\Ga_{1,L}}g\|_{L^2(\Ga_{1,L})}^2.\label{Fkestimate}
\ee
Summing over $ k\in\Z\cap[1,n-2] $ and using \eqref{nablaPhiestimate}, we obtain
\[
E_{\va}(u_{\va},(B_1^2\backslash B_{\f{1}{2}}^2)\times\pa B_L^{n-2})\leq C(L+L^{-1})\|\na_{\Ga_{1,L}}g\|_{L^2(\Ga_{1,L})}^2.
\]
This, together with \eqref{toplow0}, implies \eqref{LdarLestimate1}.

It remains to prove \eqref{Fkestimate}. By symmetry, it suffices to prove the estimate for $F_1^+$. Set $y'':=(y_2,...,y_{n-2})$. On each interval
$[r_j,r_{j-1}]$, the same change-of-variables argument as above gives
\[
\int_{r_j}^{r_{j-1}}\(\int_{F_1^+}
H(\Psi(\rho,y'))\ud\HH^{n-3}(y')\)\ud\rho
\leq
CL^{-1}\int_Q H(y')\ud y'
\]
for every non-negative integrable function $H$ on $Q$. Applying this estimate with $H=|\pa_\theta\wt g|^2+|\na_{y'}\wt g|^2$, and using
$|\pa_\rho\Psi_j|\leq CL$, together with \eqref{Fubinimore}, we have
\begin{align*}
E_{\va}(\wt{u},(B_1^2\backslash B_{\f{1}{2}}^2)\times F_1)&\leq\sum_{j=1}^{n-2}\int_{F_1}\ud y''\int_{r_j}^{r_{j-1}}\ud\rho\int_{\Ss^1}\(\rho^{-1}|\pa_{\theta}\wt{u}|^2+\rho|\pa_{\rho}\wt{u}|^2+\rho|\na_{y''}\wt{u}|^2\)\ud\HH^1(\theta)\\
&\leq C\sum_{j=1}^{n-2}\int_{F_1}\ud y''\int_{0}^{L}\f{\ud\rho}{L}\int_{\Ss^1}\(|\pa_{\theta}\wt{g}|^2+L^2|\pa_{\rho}\wt{g}|^2+|\na_{y''}\wt{g}|^2\)\ud\HH^1(\theta)\\
&\leq C\int_{\pa Q}\ud y'\int_{\Ss^1}\(L^{-1}|\pa_{\theta}\wt{g}|^2+L|\na_{y'}\wt{g}|^2+L^{-1}|\na_{y'}\wt{g}|^2\)\ud\HH^1(\theta)\\
&\leq C(L+L^{-1})\int_{\Ga_{1,L}}|\na_{\Ga_{1,L}}g|^2\ud\HH^{n-1}\\
&\leq C(L+L^{-1})\|\na_{\Ga_{1,L}}g\|_{L^2(\Ga_{1,L})}^2,
\end{align*}
where we have used the change of variables for the second inequality. Then we complete the proof.
\end{proof}

\subsection{Monotonicity formula} In this subsection, we derive a monotonicity formula for solutions of~\eqref{EL}. This formula is a basic tool for analyzing the limiting measure $ \mu_* $ associated with $ \mu_{\va} $, defined in \eqref{definemeasurewithuva}. In view of~\eqref{assumptionbound}, standard elliptic regularity theory implies that local minimizers of~\eqref{GLfunctional} are smooth and solve \eqref{EL}. Therefore, throughout this subsection, we fix a bounded domain $ \om\subset\R^n $ with $ n\geq 3 $, and assume that $ u\in C^{\ift}(\om,\R^m) $ is a solution of~\eqref{EL}.

\begin{prop}[Monotonicity]\label{Mo}
For any $ x\in\om $ and $ r\in(0,\dist(x,\pa\om)) $, we have
\be
\f{\ud}{\ud r}(r^{2-n}E_{\va}(u,B_r(x)))
= r^{2-n}\int_{\pa B_r(x)}\left|\f{y-x}{r}\cdot\na u\right|^2\ud\HH^{n-1}(y)
+ r^{1-n}\int_{B_r(x)}\f{2}{\va^2}f(u)\ud y.\label{MonotonicityFormula}
\ee
In particular, $ r^{2-n}E_{\va}(u,B_r(x)) $ is non-decreasing in $ (0,\dist(x,\pa\om)) $.
\end{prop}

Before proving Proposition \ref{Mo}, we establish two identities.

\begin{lem}[Stress-energy identity]
For any $ i\in\Z\cap[1,n] $,
\be
\pa_j\big(e_{\va}(u)\delta_{ij}-\pa_iu:\pa_ju\big)=0.\label{stressidentity}
\ee
\end{lem}

\begin{proof}
Fix $ i\in\Z\cap[1,n] $. A direct computation gives
\begin{align*}
\pa_j\big(e_{\va}(u)\delta_{ij}-\pa_iu:\pa_ju\big)
&=\pa_ie_{\va}(u)-\pa_{ij}^2u:\pa_ju-\pa_iu:\Delta u\\
&=\pa_{ik}^2u:\pa_ku+\f{1}{\va^2}D_uf(u):\pa_iu-\pa_{ij}^2u:\pa_ju-\pa_iu:\Delta u\\
&=\f{1}{\va^2}D_uf(u):\pa_iu-\pa_iu:\Delta u.
\end{align*}
Using \eqref{EL}, i.e., $ \Delta u=\va^{-2}D_uf(u) $, we conclude that the right-hand side vanishes.
\end{proof}

\begin{lem}[Pohozaev identity]
Let $ U\subset\subset\om $ be a $ C^1 $ domain and let $ x\in\R^n $. Then
\be
\begin{aligned}
&\f{n-2}{2}\int_{U}|\na u|^2\ud y+\int_{U}\f{n}{\va^2}f(u)\ud y\\
&\quad\quad=\int_{\pa U}((y-x)\cdot\nu(y))e_{\va}(u)\ud\HH^{n-1}-\int_{\pa U}((y-x)\cdot\nu(y))|\pa_{\nu}u|^2\ud\HH^{n-1}
\\
&\quad\quad\quad\quad-\int_{\pa U}\pa_{\nu}u:(\na u\cdot\mathbb{P}_{\pa U}(y-x))\ud\HH^{n-1},
\end{aligned}\label{Pohozaev}
\ee
where $ \nu(y) $ is the outer unit normal, $ \pa_{\nu}u=\nu\cdot\na u $, and $ \mathbb{P}_{\pa U}(y-x) $ denotes the orthonormal projection of $ y-x $ onto the tangential space of $ \pa U $ at $ y $.
\end{lem}

\begin{proof}
Let $ \nu_j(y) $ be the $ j $-th component of $ \nu(y) $. Testing \eqref{stressidentity} with the vector field $ \vp(y):=y-x $ and integrating by parts, we obtain
\be
\begin{aligned}
0
&=\int_U (|\na u|^2-ne_{\va}(u))\ud y
+\int_{\pa U}(e_{\va}(u)\delta_{ij}-\pa_iu:\pa_ju)(y_i-x_i)\nu_j(y)\ud\HH^{n-1}\\
&=\f{2-n}{2}\int_U|\na u|^2\ud y-\int_U\f{n}{\va^2}f(u)\ud y\\
&\quad\quad+\int_{\pa U}((y-x)\cdot\nu(y))e_{\va}(u)\ud\HH^{n-1}
-\int_{\pa U}((y-x)\cdot\na u):\pa_{\nu}u\ud\HH^{n-1}.
\end{aligned}\label{Pohozaev1}
\ee
We decompose
\[
(y-x)=((y-x)\cdot\nu(y))\nu(y)+\mathbb{P}_{\pa U}(y-x),
\]
and substitute this into the last boundary term in \eqref{Pohozaev1}. This yields \eqref{Pohozaev}.
\end{proof}

\begin{proof}[Proof of Proposition \ref{Mo}]
Apply \eqref{Pohozaev} with $ U=B_r(x) $. Since $ \nu(y)=\f{y-x}{r} $ on $ \pa B_r(x) $, we obtain
\begin{align*}
&r \int_{\pa B_r(x)} e_{\va}(u) \ud\HH^{n-1}
- (n-2) \int_{B_r(x)} e_{\va}(u)\ud y\\
&\quad\quad= r \int_{\pa B_r(x)} \left|\f{y-x}{r}\cdot\na u\right|^2 \ud\HH^{n-1}
+ \int_{B_r(x)} \f{2}{\va^2}f(u)\ud y.
\end{align*}
Dividing both sides by $ r^{n-1} $ gives \eqref{MonotonicityFormula}.
\end{proof}

\subsection{Compactness of minimizers with bounded energy}

In this subsection, we study the convergence of local minimizers of \eqref{GLfunctional} assuming a uniform energy bound.

\begin{prop}\label{boundedenergyuse}
Let $ \om\subset\R^n $ be a bounded domain with $ n\geq 3 $. Assume that $ \{u_{\va}\}_{\va\in(0,1)} $ is a family of local minimizers of \eqref{GLfunctional} such that
\be
E_{\va}(u_{\va},\om)+\|u_{\va}\|_{L^{\ift}(\om,\R^m)}\leq M.\label{assumptionbounduniform}
\ee
Then there exist a sequence $ \va_i\to 0^+ $ and a map $ u_*\in H^1(\om,\cN) $, which is a local minimizer of the Dirichlet energy \eqref{DirichletEnergy}, such that the following properties hold:
\begin{enumerate}[label=$(\theenumi)$]
\item $ u_{\va_i}\to u_* $ strongly in $ H_{\loc}^1(\om,\R^m) $.
\item Define the singular set of $ u_* $ by
\[
\sing(u_*):=\{x\in\om:\text{$u_*$ is not continuous in any }B_r(x),\,\, r>0\}.
\]
Then $ \sing(u_*) $ is $(n-3)$-rectifiable, and $ u_{\va_i}\to u_* $ in $ C_{\loc}^0(\om\backslash\sing(u_*),\R^m) $.
\end{enumerate}
\end{prop}

\begin{proof}
By \eqref{assumptionbounduniform}, the sequence $ \{u_{\va}\}_{\va\in(0,1)} $ is bounded in $ H^1(\om,\R^m) $. Hence, there exist a subsequence $ \va_i\to 0^+ $ and a map $ u_*\in H^1(\om,\R^m) $ such that
\[
u_{\va_i}\wc u_* \quad\text{weakly in }H^1(\om,\R^m).
\]
Since $ f\geq 0 $, Fatou's lemma yields
\[
0\leq\int_{\om}f(u_*)\leq\liminf_{i\to+\ift}\int_{\om}f(u_{\va_i})\leq \liminf_{i\to+\ift} \va_i^2 E_{\va}(u_{\va_i},\om)\leq \liminf_{i\to+\ift} M\va_i^2=0.
\]
Thus, $ f(u_*)=0 $ a.e. in $ \om $, and hence $ u_*\in H^1(\om,\cN) $. It follows from \cite[Proposition B.1]{CL22} that $ u_* $ is a local minimizer of \eqref{DirichletEnergy} and that $ u_{\va_i}\to u_* $ strongly in $ H_{\loc}^1(\om,\R^m) $. The $(n-3)$-rectifiability of $ \sing(u_*) $ follows from \cite[Theorem II]{SU82} and \cite[Theorem 1.5]{NV17}. Finally, by \cite[Theorem 1.3]{CLR18}, we have $ u_{\va_i}\to u_* $ in $ C_{\loc}^0(\om\backslash\sing(u_*),\R^m) $. We note that \ref{A4} corresponds to \cite[(1.6)]{CLR18} and the assumption $ \|u_{\va}\|_{L^{\ift}(\om)}\leq M $ plays the role of \cite[(1.4)]{CLR18}.
\end{proof}

\subsection{Sufficient conditions for \eqref{assumptionbound}}

In this subsection, we provide a simple criterion that ensures the assumptions in \eqref{assumptionbound} for global minimizers.

\begin{prop}\label{sufficientconditions}
Let $n\geq 3$, and let $\om\subset\R^n$ be a $C^1$-domain that is bi-Lipschitz equivalent to a ball. More precisely, assume that there exists a bi-Lipschitz homeomorphism $ \Phi:\overline{\om}\to\overline{B}_1 $ 
such that
\be
\|\na\Phi\|_{L^{\ift}(\om)}+\|\na(\Phi^{-1})\|_{L^{\ift}(B_1)}\leq M_0.\label{PhiestimateM0uva}
\ee
Assume that $\{g_{\va}\}_{\va\in(0,1)}\subset H^1(\pa\om,\R^m)$ satisfies
\be
E_{\va}(g_{\va},\pa\om)\leq M_0(|\log\va|+1),
\qquad
\|g_{\va}\|_{L^{\ift}(\pa\om,\R^m)}\leq M_0, \label{gvaassumption}
\ee
for some $M_0>0$. Suppose that there exists $R_0>0$ such that the potential $f\in C^{\ift}(\R^m,[0,+\ift))$ satisfies
\be
y\cdot Df(y)\geq 0
\qquad\text{for all } |y|\geq R_0. \label{ygeqR0}
\ee
For each $\va\in(0,1)$, let $u_{\va}\in H^1(\om,\R^m)$ be a global minimizer of \eqref{GLfunctional} with boundary condition $g_{\va}$; that is,
\[
E_{\va}(u_{\va},\om)\leq E_{\va}(v_{\va},\om)
\]
for every $v_{\va}\in H^1(\om,\R^m)$ such that $v_{\va}=g_{\va}$ on $\pa\om$ in the trace sense. Then
\begin{gather}
E_{\va}(u_{\va},\om)\leq C(|\log\va|+1),
\qquad
\|u_{\va}\|_{L^{\ift}(\om,\R^m)}\leq C, \label{logassumption}
\end{gather}
where $C=C(n,f,M_0,\cN,R_0)$.
\end{prop}

\begin{proof}
We first prove the energy bound by constructing an admissible comparison map. Define the boundary map $ h_{\va}:=g_{\va}\circ \Phi^{-1} $ on $ \pa B_1 $. Since $\Phi|_{\pa\om}$ is bi-Lipschitz with constants controlled by $M_0$, the trace composition theorem and the area formula give
\be
E_{\va}(h_{\va},\pa B_1)
\leq C E_{\va}(g_{\va},\pa\om)
\leq C(|\log\va|+1), \label{boundaryenergyonball}
\ee
where $C$ depends only on $n$ and $M_0$.

Let $H_{\va}:B_1\to\R^m$ be the homogeneous extension of $h_{\va}$:
\[
H_{\va}(y):=h_{\va}\left(\frac{y}{|y|}\right)
\quad\text{for } y\in B_1\backslash\{0\}.
\]
Note that $ H_{\va}|_{\pa B_1}=h_{\va} $. Because $n\geq 3$, we have $H_{\va}\in H^1(B_1,\R^m)$. Moreover,
\[
E_{\va}(H_{\va},B_1)\leq C(n)E_{\va}(h_{\va},\pa B_1).
\]
Now set $ v_{\va}:=H_{\va}\circ\Phi $ in $ \om $. Then $v_{\va}|_{\pa\om}=g_{\va}$. Using \eqref{PhiestimateM0uva}, we obtain
\[
E_{\va}(v_{\va},\om)
\leq C E_{\va}(H_{\va},B_1)
\leq C E_{\va}(g_{\va},\pa\om)
\leq C(|\log\va|+1).
\]
By the minimality of $u_{\va}$,
\[
E_{\va}(u_{\va},\om)
\leq E_{\va}(v_{\va},\om)
\leq C(|\log\va|+1).
\]

It remains to prove the uniform $L^{\ift}$-bound. By \cite[Lemma 8.3]{Lam14} and the assumption \eqref{ygeqR0}, the maximum principle for minimizers gives
\[
\|u_{\va}\|_{L^{\ift}(\om,\R^m)}
\leq R_0+\|g_{\va}\|_{L^{\ift}(\pa\om,\R^m)}
\leq R_0+M_0.
\]
Combining this estimate with the energy bound proves \eqref{logassumption}.
\end{proof}

\begin{rem}
The conditions in Proposition \ref{sufficientconditions} are not intended to be optimal. Rather, they provide a convenient criterion that can be verified in applications. A class of boundary data $ \{g_{\va}\}_{\va\in(0,1)} $ that meets \eqref{gvaassumption} can be constructed as \cite{BBO01}. More precisely, one may assume that $ \om $ is smooth and that there exists a finite collection $ \Sg $ of smooth $ (n-3) $-dimensional submanifolds of $ \pa\om $ such that
\begin{align*}
|g_{\va}(x)|&\leq C_0,\quad\text{for any }x\in\pa\om,\\
g_{\va}(x)&\in\cN,\quad\text{if }x\in\pa\om\text{ and }\dist(x,\Sg)\geq\va,\\
|\na_{\pa\om}g_{\va}(x)|&\leq\frac{C_0}{\max(\dist(x,\Sg),\va)} \quad\text{for any }x\in\pa\om.
\end{align*}
Such boundary data are known to satisfy energy bounds of the form \eqref{gvaassumption}.
\end{rem}

\section{Luckhaus-type lemmas}\label{SectionLuckhaus}

\subsection{Statements of the results}

In this section, we collect two Luckhaus-type lemmas. The first lemma is stated on the sphere, whereas the second one is formulated on cylindrical domains. These results will be used to prove the clearing-out lemma and derive energy quantization. We begin with the spherical version.

\begin{lem}\label{Luckhauslemma}
Assume that $ \{u_{\ol{\va}}\}_{\ol{\va}\in(0,1)}\subset H^1(\pa B_1,\R^m) $ satisfies
\be
\|u_{\ol{\va}}\|_{L^{\ift}(\pa B_1)}\leq M
\label{Boun}
\ee
for some $ M>0 $. Then there exist constants $ \eta_0,\ol{\va}_0\in(0,1) $ and $ C>0 $, depending only on $ f,M,\cN $, and $ n $, such that the following properties hold.

For any $ \eta\in(0,100^{-n}\eta_0) $ and $ \ol{\va}\in(0,\ol{\va}_0) $, if
\begin{align}
E_{\ol{\va}}(u_{\ol{\va}},\pa B_1)\leq\eta|\log\ol{\va}|,
\label{log1}
\end{align}
then there exist $ v_{\ol{\va}}\in H^1(\pa B_1,\cN) $ and $ \vp_{\ol{\va}}\in H^1(B_1\backslash B_{1-h(\ol{\va})},\R^m) $ such that
\begin{gather}
\vp_{\ol{\va}}(x)=u_{\ol{\va}}(x)\text{ and }
\vp_{\ol{\va}}((1-h(\ol{\va}))x)=v_{\ol{\va}}(x)
\quad \text{for }\HH^{n-1}\text{-a.e. }x\in\pa B_1,
\label{interpoL}\\
\frac{1}{2}\int_{\pa B_1}|\na_{\pa B_1}v_{\ol{\va}}|^2\ud\HH^{n-1}
\leq C E_{\ol{\va}}(u_{\ol{\va}},\pa B_1),
\label{VboundL}\\
E_{\ol{\va}}(\vp_{\ol{\va}},B_1\backslash B_{1-h(\ol{\va})})
\leq Ch(\ol{\va})E_{\ol{\va}}(u_{\ol{\va}},\pa B_1),
\label{WboundL}
\end{gather}
where
\be
h(\ol{\va}):=\left(\frac{E_{\ol{\va}}(u_{\ol{\va}},\pa B_1)+\eta_0}{\eta_0|\log\ol{\va}|}\right)^{\frac{1}{n-3}}
\in\left(0,\frac{1}{100}\right).
\label{hdef}
\ee
\end{lem}

We next present a version of cylindrical domains to be used in product-type regions. We recall that $\Gamma_{r,L} := \partial B_r^2 \times B_L^{n-2}$.

\begin{lem}\label{Luckhauslemma1}
Let $ \delta\in(0,\frac{1}{10}] $. Assume that $ \{g_{\delta,\ol{\va}}\}_{\ol{\va}\in(0,1)}\subset H^1(\Ga_{\delta,1},\R^m) $ satisfies
\be
\|g_{\delta,\ol{\va}}\|_{L^{\ift}(\Ga_{\delta,1},\R^m)}\leq M
\label{fGboundbound0}
\ee
for some $ M>0 $. Then there exist constants $ \eta_1,\wh{\va}_1\in(0,1) $, depending only on $ f,M,\cN $, and $ n $, such that the following properties hold.

If $ \ol{\va}\in(0,\wh{\va}_1\delta) $ and
\be
E_{\ol{\va}}(g_{\delta,\ol{\va}},\Ga_{\delta,1})
\leq \delta^{n-3}\eta_1\log\frac{\delta}{\ol{\va}},
\label{GvaboundaryH10}
\ee
then there exist
\[
r\in\left(1-\delta,1-\frac{\delta}{2}\right), \quad
v_{\delta,\ol{\va}}\in H^1(\pa B_\delta^2\times B_r^{n-2},\cN), \quad
\vp_{\delta,\ol{\va}}\in H^1((B_1^2\backslash B_{\frac{\delta}{2}}^2)\times B_r^{n-2},\R^m),
\]
and a constant $ C>0 $, depending only on $ f,M,\cN $, and $ n $, such that the following assertions hold.

\begin{enumerate}[label=$(\theenumi)$]
\item For $ \HH^{n-1} $-a.e. $ x=(y,z)\in \pa B_{\delta}^2\times B_r^{n-2} $,
\[
\vp_{\delta,\ol{\va}}(x)=g_{\delta,\ol{\va}}(x),
\quad
\vp_{\delta,\ol{\va}}\left(\frac{y}{2},z\right)=v_{\delta,\ol{\va}}(y,z).
\]

\item The maps $ v_{\delta,\ol{\va}} $ and $ \vp_{\delta,\ol{\va}} $ satisfy
\begin{align}
E_{\ol{\va}}(v_{\delta,\ol{\va}},\pa B_{\frac{\delta}{2}}^2\times B_r^{n-2})
&\leq C E_{\ol{\va}}(g_{\delta,\ol{\va}},\Ga_{\delta,1}),
\label{Vvaes1}\\
E_{\ol{\va}}(\vp_{\delta,\ol{\va}},(B_{\delta}^2\backslash B_{\frac{\delta}{2}}^2)\times B_r^{n-2})
&\leq C\delta E_{\ol{\va}}(g_{\delta,\ol{\va}},\Ga_{\delta,1}).
\label{Wvaest1}
\end{align}

\item The trace
\[
\vp_{\delta,\ol{\va}}^{*}
:=\vp_{\delta,\ol{\va}}\big|_{(B_{\delta}^2\backslash B_{\frac{\delta}{2}}^2)\times\pa B_r^{n-2}}
\in H^1((B_{\delta}^2\backslash B_{\frac{\delta}{2}}^2)\times\pa B_r^{n-2},\R^m)
\]
satisfies
\be
E_{\ol{\va}}(\vp_{\delta,\ol{\va}}^*,(B_{\delta}^2\backslash B_{\frac{\delta}{2}}^2)\times\pa B_r^{n-2})
\leq C E_{\ol{\va}}(g_{\delta,\ol{\va}},\Ga_{\delta,1}).
\label{WpluesminusC11}
\ee
\end{enumerate}
\end{lem}

\subsection{Proof of Lemma \ref{Luckhauslemma}}

For simplicity, set $ h:=h(\ol{\va}) $. Throughout the proof, we choose parameters $ \eta_0,\ol{\va}_0\in(0,1) $ as needed and keep the same notation. We divide the proof into several steps.

\smallskip
\noindent\textbf{Step 1: A ``good" partition of $ \pa B_1 $.} As in the proof of \cite[Lemma 1]{Luc88}, we establish a partition of $ \pa B_1 $ as
\be
\partial B_1=\bigsqcup_{j=0}^{n-1} \cR_j, \quad \cR_j=\bigsqcup_{i=1}^{k_j} c_{ij},\label{Rjkjdefuse}
\ee
where each~$c_{ij}$ is a $j$-cell ---  that is, an open set homeomorphic to the $j$-dimensional ball~$B^j$ --- and $ \bigsqcup $ denotes a disjoint union. Moreover, there exists $C_0=C_0(n) > 0$ such that the following properties hold.
\begin{itemize}
\item The closure of each $j$-cell $\overline{c}_{ij}$ is bi-Lipschitz equivalent to $\overline{B}_h^j$: there exists an invertible map $\phi_{ij}:\overline{c}_{ij}\to \overline{B}_h^j $ satisfying
\be
\|\na\phi_{ij}\|_{L^{\ift}(c_{ij})}+\|\na(\phi_{ij}^{-1})\|_{L^{\ift}(B_h^j)}\leq C_0.
\label{phiijestimate}
\ee
\item For any $ j\in\Z\cap[0,n-2] $ and $ i\in\Z\cap[1,k_j] $,
\be
\#\{\ell\in\Z\cap[1,k_{j+1}]:c_{ij}\subset\pa c_{\ell,j+1}\}\leq C_0.\label{ellnumber}
\ee
\item 
For any $j\in\Z\cap [0, n-1]$,
\be
\begin{aligned}
\int_{\cR_j}|\na_{\cR_j}u_{\ol{\va}}|^2\ud\HH^j&\leq C_0h^{j+1-n}\int_{\pa B_1}|\na_{\pa B_1}u_{\ol{\va}}|^2\ud\HH^{n-1},\\
\int_{\cR_j}f(u_{\ol{\va}})\ud\HH^j&\leq C_0h^{j+1-n}\int_{\pa B_1}f(u_{\ol{\va}})\ud\HH^{n-1},
\end{aligned}\label{uvafvaleq}
\ee
and in particular,
\be
E_{\ol{\va}}(u_{\ol{\va}},\cR_j)\leq C_0h^{j+1-n}E_{\ol{\va}}(u_{\ol{\va}},\pa B_1).\label{Evagrids}
\ee
\end{itemize}
A partition with these properties can be constructed by mapping $\partial B_1$ into the boundary of a cube $\partial [0,1]^n$, decomposing each boundary face of $\partial [0,1]^n$ into, say, $(\lfloor h^{-1}\rfloor)^n$ cubes of equal size, and then mapping this decomposition back to the sphere. By applying a suitable rotation and using Fubini's theorem, we can ensure that the resulting decomposition satisfies~\eqref{uvafvaleq}. See e.g. \cite[Lemma~3.7]{BSV25} for details. This decomposition of $ \pa B_1 $ induces a partition of $ B_1\backslash B_{1-h} $ by
\be
B_1\backslash B_{1-h}=\bigsqcup_{j=0}^{n-1}\widehat{\cR}_j,\quad \widehat{\cR}_j=\bigsqcup_{i=1}^{k_j}\wh{c}_{ij},\quad\wh{c}_{ij}=\left\{x \in B_1\backslash B_{1-h}:\frac{x}{|x|}\in c_{ij}\right\}.\label{induceddecomposition}
\ee

\smallskip
\noindent\textbf{Step 2: Construction of $ v_{\ol{\va}} $ on $ \cR_1 $.}
By \ref{A1} and \eqref{A31}, there exists $ \psi\in C^1([0,+\ift),[0,+\ift)) $ such that
\be
\left\{\begin{aligned}
&\psi(s)=\beta s^2&\text{ if }&s\in[0,\delta_0),\\
&0<\psi(s)\leq C&\text{ if }&s\in[\delta_0,+\ift),\\
&\psi(\dist(y,\cN))\leq f(y)&\text{ if }&y\in\R^m,
\end{aligned}\right.\label{propertypsi}
\ee
where $ \beta,C,\delta_0>0 $ depend only on $ f $ and $ \cN $. Define
\be
G(t):=\int_0^t\psi^{\f{1}{6}}(s)\ud s\quad\text{and}\quad d_{\ol{\va}}:=\dist(u_{\ol{\va}},\cN). \label{Gdefinition}
\ee
Note that $ |\na_{\cR_j}d_{\ol{\va}}|\leq|\na_{\cR_j}u_{\ol{\va}}| $ for any $ j\in\Z\cap[0,n-1] $. Since $ u_{\ol{\va}}\in H^1(\pa B_1,\R^m) $, we have $ d_{\ol{\va}}\in H^1(\pa B_1) $. Moreover, \eqref{propertypsi} gives $ \psi(d_{\ol{\va}})\leq f(u_{\ol{\va}}) $. Applying \eqref{log1} and Young's inequality $ a+b\geq Ca^{\f{3}{4}}b^{\f{1}{4}} $, we obtain
\begin{align*}
C_0|\log\ol{\va}|&\geq h^{n-2}\int_{\cR_1}\(\f{1}{2}|\na d_{\ol{\va}}|^2+\f{1}{\ol{\va}^2}\psi(d_{\ol{\va}})\)\ud\HH^1\\
&\geq C\ol{\va}^{-\f{1}{2}}h^{n-2}\int_{\cR_1}|\na d_{\ol{\va}}|^{\f{3}{2}}\psi^{\f{1}{4}}(d_{\ol{\va}})\ud\HH^1\\
&=C\ol{\va}^{-\f{1}{2}}h^{n-2}\int_{\cR_1}|\na G(d_{\ol{\va}})|^{\f{3}{2}}\ud\HH^1.
\end{align*}
The Sobolev embedding theorem then gives
\be
\(\osc_{c_{i1}}G(d_{\ol{\va}})\)^{\f{3}{2}}\leq Ch^{\f{1}{2}}\int_{c_{i1}}|\na G(d_{\ol{\va}})|^{\f{3}{2}}\ud\HH^1\leq C\ol{\va}^{\f{1}{2}}h^{\f{5}{2}-n}|\log\ol{\va}|.\label{Gosc}
\ee
Recalling the definition of $ h $ in \eqref{hdef}, we have
\[
\ol{\va}^{\f{1}{2}}h^{\f{5}{2}-n}|\log\ol{\va}|\leq \ol{\va}^{\f{1}{2}}|\log\ol{\va}|^{\f{2n-5}{2n-6}}\cdot|\log\ol{\va}|\to 0^+
\]
as $ \ol{\va}\to 0^+ $. Combined with \eqref{propertypsi}, \eqref{Gdefinition}, and \eqref{Gosc}, this implies
\be
\lim_{\ol{\va}\to 0^+}\(\osc_{\cR_1}d_{\ol{\va}}\)=0.\label{oscdva}
\ee
We claim that
\be
\lim_{\ol{\va}\to 0^+}\(\sup_{i\in\Z\cap[1,k_1]}\dashint_{c_{i1}}d_{\ol{\va}}\)\ud\HH^1=0.\label{claimaverage}
\ee
By \eqref{Boun}, there exists $ \lda_0=\lda_0(M,\cN)>0 $ such that
\[
\|d_{\ol{\va}}\|_{L^{\ift}(\pa B_1)}\leq\lda_0.
\]
For $ \lda\in(0,\lda_0) $, set
\[
\psi_*(\lda):=\inf_{\lda\leq s\leq\lda_0}\psi(s)>0.
\]
Then, for any $ i\in[1,k_1] $,
\be
\begin{aligned}
0\leq\dashint_{c_{i1}}d_{\ol{\va}}\ud\HH^1&=\f{1}{\HH^1(c_{i1})}\int_{\{d_{\ol{\va}}\leq\lda\}\cap c_{i1}}d_{\ol{\va}}\ud\HH^1+\f{1}{\HH^1(c_{i1})}\int_{\{d_{\ol{\va}}\geq\lda\}\cap c_{i1}}d_{\ol{\va}}\ud\HH^1\\
&\leq\f{\HH^1(\{d_{\ol{\va}}\leq\lda\}\cap c_{i1})}{\HH^1(c_{i1})}\lda+\f{\HH^1(\{d_{\ol{\va}}\geq\lda\}\cap c_{i1})}{\HH^1(c_{i1})}\lda_0\\
&\leq\lda+\f{\lda_0}{\psi_*(\lda)}\dashint_{\{d_{\ol{\va}}\geq\lda\}\cap c_{i1}}\psi(d_{\ol{\va}})\ud\HH^1.
\end{aligned}\label{dvaaverage}
\ee
From \eqref{log1}, \eqref{uvafvaleq}, and \eqref{propertypsi},
\begin{align*}
\dashint_{c_{i1}}\psi(d_{\ol{\va}})\ud\HH^1&\leq Ch^{-1}\int_{\cR_1}f(u_{\ol{\va}})\ud\HH^1\leq C\ol{\va}^2h^{1-n}E_{\ol{\va}}(u_{\ol{\va}},\pa B_1)\\
&\leq Ch^{1-n}\ol{\va}^2|\log\ol{\va}|\leq C\ol{\va}^2|\log\ol{\va}|^{\f{n-1}{n-3}}\cdot|\log\ol{\va}|\to 0^+
\end{align*}
as $ \ol{\va}\to 0^+ $. For any $ \delta>0 $, choose $ \lda=\f{\delta}{2} $ in \eqref{dvaaverage}. Then for any sufficiently small $ \ol{\va}\in(0,1) $,
\[
0\leq\dashint_{c_{i1}}d_{\ol{\va}}\ud\HH^1\leq\delta.
\]
establishing \eqref{claimaverage}. Together with \eqref{oscdva}, this yields
\be
\lim_{\ol{\va}\to 0^+}\(\sup_{\cR_1}d_{\ol{\va}}\)=0.\label{dvalimit0}
\ee
In particular, there exists $ \ol{\va}_0=\ol{\va}_0(f,M,\cN,n)\in(0,1) $ such that for any $ \ol{\va}\in(0,\ol{\va}_0) $, the nearest-point projection $ \Pi_{\cN}\circ u_{\ol{\va}} $ given by \eqref{Nearestpointprojection} is well-defined on $ u_{\ol{\va}}(\cR_1) $. Set $ v_{\ol{\va}}:=\Pi_{\cN}\circ u_{\ol{\va}} $ on $ \cR_1 $. In summary, for any $ x\in \cR_1 $,
\be
v_{\ol{\va}}(x)=\Pi_{\cN}\circ u_{\ol{\va}}(x),\quad|v_{\ol{\va}}(x)-u_{\ol{\va}}(x)|<\delta_f,\label{leqdeltaf}
\ee
where $ \delta_f>0 $ is given by \eqref{A31}.

\smallskip
\noindent\textbf{Step 3: Construction of $ v_{\ol{\va}} $ on $ \pa B_1 $.} We construct $ v_{\ol{\va}} $ on $ \cR_j $ for $ j\in\Z\cap[2,n-1] $ inductively, establishing
\be
\int_{\cR_j}|\na_{\cR_j}v_{\ol{\va}}|^2\ud\HH^j\leq Ch^{j+1-n}E_{\ol{\va}}(u_{\ol{\va}},\pa B_1),\label{inductionvolva}
\ee
where $ C=C(f,M,\cN,n)>0 $. Taking $j = n-1$ gives \eqref{VboundL} and completes the construction.

Let $ j=2 $. We claim that there exists $ \eta_0=\eta_0(f,M,\cN,n)\in(0,1) $ such that, under the assumptions \eqref{log1} and \eqref{hdef}, the restriction $ v_{\ol{\va}}|_{\pa c_{i2}} $ is null-homotopic for any $ i\in\Z\cap[1,k_2] $. By \eqref{uvafvaleq} and the Sobolev embedding theorem, $ v_{\ol{\va}}|_{\cR_1}\in C^0(\cR_1,\cN) $, so the homotopy class of $ v_{\ol{\va}}|_{\pa c_{i2}} $ is well-defined. Suppose that the claim fails; then there exists a $2$-cell $c_{i2} =: c$ such that $ u|_{\pa c} $ is not null-homotopic. From \eqref{hdef}, for $ \ol{\va}\in(0,\ol{\va}_0) $ with $ \ol{\va}_0 $ sufficiently small,
\be
h>|\log\ol{\va}|^{-\f{1}{n-3}}>\ol{\va}^{\f{1}{2}}>\ol{\va}.\label{hvageqva1}
\ee
Lemma \ref{LowerBound} then gives
\begin{align*}
E_{\ol{\va}}(u_{\ol{\va}},c)+hE_{\ol{\va}}(u_{\ol{\va}},\pa c)&\geq C^{-1}|[v_{\ol{\va}}|_{\pa c}]_{\cN}|_{*}\log\f{h}{\ol{\va}}-C\stackrel{\eqref{c0geq}}{\geq}C_1|\log\ol{\va}|-C_2,
\end{align*}
where $ C_1,C_2>0 $ depends only on $ f,M,\cN,n $. By \eqref{hdef} and \eqref{Evagrids},
\[
C_1|\log\ol{\va}|-C_2\leq Ch^{3-n}E_{\ol{\va}}(u_{\ol{\va}},\pa B_1)\leq C\eta_0|\log\ol{\va}|.
\]
Choosing $\eta_0\in(0,1)$ to be sufficiently small leads to a contradiction. Since $ v_{\ol{\va}}|_{\pa c_{i2}} $ is homotopically trivial for any $ i\in[1,k_2] $, Lemma \ref{ExtensionLemma1} extends $ v_{\ol{\va}}|_{\pa c_{i2}} $ to $c_{i2}$ with
\[
\int_{c_{i2}}|\na_{c_{i2}}v_{\ol{\va}}|^2\ud\HH^2\leq Ch\int_{\pa c_{i2}}|\na_{\pa c_{i2}}v_{\ol{\va}}|^2\ud\HH^1.
\]
Summing over $i$ and applying \eqref{ellnumber} and \eqref{uvafvaleq},
\begin{align*}
\int_{\cR_2}|\na_{\cR_2}v_{\ol{\va}}|^2\ud\HH^2&\leq\sum_{i=1}^{k_2}\int_{c_{i2}}|\na_{c_{i2}}v_{\ol{\va}}|^2\ud\HH^2\leq Ch\(\sum_{i=1}^{k_2}\int_{\pa c_{i2}}|\na_{\pa c_{i2}}v_{\ol{\va}}|^2\ud\HH^1\)\\
&\leq Ch\int_{\cR_1}|\na_{\pa c_{i2}}v_{\ol{\va}}|^2\ud\HH^1\leq Ch^{3-n}E_{\ol{\va}}(u_{\ol{\va}},\pa B_1),
\end{align*}
confirming \eqref{inductionvolva} for $j = 2$.

Now suppose $ v_{\ol{\va}} $ is defined on $ \cR_{j-1} $ for some $ j\in\Z\cap[3,n-1] $ and satisfies
\be
\int_{\cR_{j-1}}|\na_{\cR_{j-1}}v_{\ol{\va}}|^2\ud\HH^{j-1}\leq Ch^{j-n}E_{\ol{\va}}(u_{\ol{\va}},\pa B_1).\label{inductionHypo1}
\ee
For each $ c_{ij} $, extend $ v_{\ol{\va}} $ from $ \pa c_{ij}$ by homogeneous extension:
\[
v_{\ol{\va}}(x):=v_{\ol{\va}}\circ(\phi_{ij}^{-1})\circ\(\f{h\phi_{ij}(x)}{|\phi_{ij}(x)|}\),\quad x\in c_{ij}\backslash\{\phi_{ij}^{-1}(0^j)\}
\]
where $ \phi_{ij} $ is as in \eqref{phiijestimate}. Direct computation gives
\begin{align*}
\int_{\cR_j}|\na_{\cR_j}v_{\ol{\va}}|^2\ud\HH^j&\leq\sum_{i=1}^{k_j}\int_{c_{ij}}|\na_{c_{ij}}v_{\ol{\va}}|^2\ud\HH^j\leq Ch\(\sum_{i=1}^{k_j}\int_{\pa c_{ij}}|\na_{\pa c_{ij}}v_{\ol{\va}}|^2\ud\HH^{j-1}\)\\
&\stackrel{\eqref{ellnumber}}{\leq} Ch\int_{\cR_{j-1}}|\na_{\cR_{j-1}}v_{\ol{\va}}|^2\ud\HH^{j-1}\stackrel{\eqref{inductionHypo1}}{\leq} Ch^{j+1-n}E_{\ol{\va}}(u_{\ol{\va}},\pa B_1),
\end{align*}
completing the construction of $ v_{\ol{\va}} $ on $ \pa B_1 $.

\smallskip\noindent\textbf{Step 4: Construction of $ \vp_{\ol{\va}} $.} Using the decomposition \eqref{induceddecomposition} of $ B_1\backslash B_{1-h} $, we construct $ \vp_{\ol{\va}} $ on $ \wh{\cR}_j $ for $ j\in\Z\cap[1,n-1] $ inductively, showing that
\be
E_{\ol{\va}}(\vp_{\ol{\va}},\wh{\cR}_j)\leq Ch^{j+2-n}E_{\ol{\va}}(u_{\ol{\va}},\pa B_1).\label{inductionEstimate11}
\ee
Taking $j = n-1$ gives \eqref{WboundL} and completes the proof.

For $j = 1$ and any $i \in [1, k_1]$ where $ k_1 $ is as in \eqref{Rjkjdefuse}, define
\[
\vp_{\ol{\va}}(x):=\(1-\f{1-|x|}{h}\)u_{\ol{\va}}\(\f{x}{|x|}\)+\f{1-|x|}{h}v_{\ol{\va}}\(\f{x}{|x|}\),\quad x\in\wh{\cR}_1.
\]
This also defines $ \vp_{\ol{\va}} $ on $ \wh{\cR}_0 $. By \eqref{fBconvex} with $t=1-\f{1-|x|}{h}$ and \eqref{leqdeltaf},
\begin{align*}
f(\vp_{\ol{\va}})&\leq f\!\(tu_{\ol{\va}}\(\f{x}{|x|}\)+(1-t)v_{\ol{\va}}\(\f{x}{|x|}\)\)\\
&\leq C\(\f{1-h-|x|}{h}\)^2f\!\(u_{\ol{\va}}\(\f{x}{|x|}\)\).
\end{align*}
Using \eqref{uvafvaleq} and \eqref{hvageqva1}, for small $ \ol{\va}\in(0,\ol{\va}_0) $ satisfying \eqref{hvageqva1}, we have
\be
\ol{\va}^{-2}\int_{\wh{\cR}_1}f(\vp_{\ol{\va}})\ud\HH^2\leq C\ol{\va}^{-2}h\int_{\cR_1}f(u_{\ol{\va}})\ud\HH^1\leq Ch^{3-n}E_{\ol{\va}}(u_{\ol{\va}},\pa B_1).\label{fvpfresults}
\ee
In follows from \eqref{uvafvaleq}, \eqref{leqdeltaf}, and \eqref{hvageqva1} that
\begin{align*}
\int_{\wh{\cR}_1}|\na_{\wh{\cR}_1}\vp_{\ol{\va}}|^2\ud\HH^2&\leq Ch^{-1}\int_{\cR_1}|u_{\ol{\va}}-v_{\ol{\va}}|^2\ud\HH^1\leq Ch^{-1}\int_{\cR_1}f(u_{\ol{\va}})\ud\HH^1\leq Ch^{3-n}E_{\ol{\va}}(u_{\ol{\va}},\pa B_1).
\end{align*}
Together with \eqref{fvpfresults}, this implies \eqref{inductionEstimate11} for $j = 1$.

Now suppose $ \vp_{\ol{\va}} $ is constructed on $ \wh{\cR}_{j-1} $ for some $ j\in\Z\cap[2,n-1] $ with
\be
E_{\ol{\va}}(\vp_{\ol{\va}},\wh{\cR}_{j-1})\leq Ch^{j+1-n}E_{\ol{\va}}(u_{\ol{\va}},\pa B_1).\label{inductionHypo11}
\ee
By \eqref{phiijestimate}, for each $ i\in\Z\cap[1,k_j] $, there exists a bi-Lipschitz map $ \Phi_{ij}:\ol{\wh{c}}_{ij}\to\ol{B}_h^{j+1} $ such that
\[
\|\na\Phi_{ij}\|_{L^{\ift}(\wh{c}_{ij})}+\|\na(\Phi_{ij}^{-1})\|_{L^{\ift}(B_h^{j+1})}\leq C(n).
\]
Define $ \vp_{\ol{\va}} $ on $ \wh{c}_{ij} $ by homogeneous extension from $ \pa\wh{c}_{ij} $:
\[
\vp_{\ol{\va}}(x):=\vp_{\ol{\va}}\circ(\Phi_{ij}^{-1})\circ\(\f{h\Phi_{ij}(x)}{|\Phi_{ij}(x)|}\)
\]
for any $ x\in\wh{c}_{ij}\backslash\{\Phi_{ij}^{-1}(0^{j+1})\} $. Direct computation and \eqref{inductionHypo11} give
\begin{align*}
E_{\ol{\va}}(\vp_{\ol{\va}},\wh{\cR}_j)&=\sum_{i=1}^{k_j}E_{\ol{\va}}(\vp_{\ol{\va}},\wh{c}_{ij})\leq Ch\(\sum_{i=1}^{k_j}E_{\ol{\va}}(\vp_{\ol{\va}},\pa\wh{c}_{ij})\)\\
&\leq Ch(E_{\ol{\va}}(\vp_{\ol{\va}},\wh{\cR}_{j-1})+E_{\ol{\va}}(u_{\ol{\va}},\cR_j)+E_{\ol{\va}}(v_{\ol{\va}},\cR_j))\stackrel{\eqref{inductionvolva}}{\leq} Ch^{j+2-n}E_{\ol{\va}}(u_{\ol{\va}},\pa B_1),
\end{align*}
This implies \eqref{inductionHypo11} for $ j $ and completes the inductive argument.

\qed

\subsection{Proof of Lemma \ref{Luckhauslemma1}} Set $\wh{\va} := \frac{\ol{\va}}{\delta}$ and $g_{\wh{\va}}(x) := g_{\delta,\ol{\va}}(\delta x)$. By \eqref{fGboundbound0} and \eqref{GvaboundaryH10},
\be
\begin{gathered}
\|g_{\wh{\va}}\|_{L^\ift(\Ga_{1,\delta^{-1}},\R^m)} \leq M, \\
E_{\wh{\va}}(g_{\wh{\va}}, \Ga_{1,\delta^{-1}}) \leq \eta_1 |\log \wh{\va}|,
\end{gathered} \label{scalingproperties}
\ee
where $\eta_1 > 0$ is a small constant to be chosen. By Fubini's theorem, there exists $L \in (\delta^{-1}-1, \delta^{-1}-\frac{1}{2})$ such that
\be
E_{\wh{\va}}(g_{\wh{\va}}, \pa B_1^2 \times \pa B_L^{n-2})
\leq 4 E_{\wh{\va}}(g_{\wh{\va}}, \Ga_{1,\delta^{-1}}). \label{chooseL}
\ee
We show that there exist
\[
v_{\wh{\va}} \in H^1(\pa B_1^2 \times B_L^{n-2}, \cN),\quad
\vp_{\wh{\va}} \in H^1((B_1^2 \backslash B_{\f{1}{2}}^2) \times B_L^{n-2}, \R^m),
\]
and a constant $C > 0$ depending only on $f, M, \cN, n$, such that the following properties hold.

\begin{enumerate}[label=$(\theenumi)$]
\item For $\HH^{n-1}$-a.e. $x = (y,z) \in \pa B_1^2 \times B_L^{n-2}$,
\[
\vp_{\wh{\va}}(x) = g_{\wh{\va}}(x)
\quad \text{and} \quad
\vp_{\wh{\va}}\left(\frac{y}{2}, z\right) = v_{\wh{\va}}(x).
\]

\item $v_{\wh{\va}}$ and $\vp_{\wh{\va}}$ satisfy
\begin{align}
E_{\wh{\va}}(v_{\wh{\va}}, \pa B_{\f{1}{2}}^2 \times B_L^{n-2})
&\leq C E_{\wh{\va}}(g_{\wh{\va}}, \Ga_{1,\delta^{-1}}), \label{Vvaescylinder} \\
E_{\wh{\va}}(\vp_{\wh{\va}}, (B_1^2 \backslash B_{\f{1}{2}}^2) \times B_L^{n-2})
&\leq C E_{\wh{\va}}(g_{\wh{\va}}, \Ga_{1,\delta^{-1}}). \label{Wvaestcylinder}
\end{align}

\item The restriction $\vp_{\wh{\va}}^* := \vp_{\wh{\va}}|_{(B_1^2 \backslash B_{\f{1}{2}}^2) \times \pa B_L^{n-2}}$ belongs to
$H^1((B_1^2 \backslash B_{\f{1}{2}}^2) \times \pa B_L^{n-2}, \R^m)$ and satisfies
\be
E_{\wh{\va}}(\vp_{\wh{\va}}^*, (B_1^2 \backslash B_{\f{1}{2}}^2) \times \pa B_L^{n-2})
\leq C E_{\wh{\va}}(g_{\wh{\va}}, \Ga_{1,\delta^{-1}}). \label{WpluesminusCcylinder}
\ee
\end{enumerate}

Given $v_{\wh{\va}}$ and $\vp_{\wh{\va}}$ as above, set $v_{\delta,\ol{\va}}(x) := v_{\wh{\va}}(\f{x}{\delta})$ and
$\vp_{\delta,\ol{\va}}(x) := \vp_{\wh{\va}}(\f{x}{\delta})$ with $r := \delta L$. Then \eqref{Vvaes1}, \eqref{Wvaest1}, and \eqref{WpluesminusC11} follow from \eqref{Vvaescylinder}, \eqref{WpluesminusCcylinder}, and \eqref{Wvaestcylinder}, respectively.

The proof proceeds in four steps. Throughout the proof, the parameters $\eta_1, \wh{\va}_1 \in (0,1)$ can be adjusted finitely many times without changing notation.

\smallskip
\noindent\textbf{Step 1: A ``good" partition of $ \ol{\Ga}_{1,L} $.} Let $Q := (-L, L)^{n-2}$ and $\Phi : \ol{B}_L^{n-2} \to \ol{Q}$ be a bi-Lipschitz map satisfying $\Phi(0^{n-2}) = 0^{n-2}$, with $\|\na\Phi\|_{L^\ift} + \|\na(\Phi^{-1})\|_{L^\ift}$ bounded by a constant depending only on $n$. We first construct a family of cell decompositions of $\partial B_1^2\times Q$.

For $\theta \in (0,\pi)$, define $A(\theta) := \{a_j(\theta)\}_{j=1}^2 \subset \pa B_1^2$, where
\[
a_j(\theta) := (\cos(j\pi + \theta), \sin(j\pi + \theta)).
\]
The two points $a_1(\theta)$ and $a_2(\theta)$ divide $\partial B_1^2$
into two circular arcs. Thus, they define a cell decomposition of
$\partial B_1^2$, with $0$-cells given by these two points and $1$-cells
given by the two relatively open arcs between them. We write this
decomposition as
\be
\pa B_1^2 = \bigsqcup_{j=0}^1 \wt{\cR}_j^0(\theta), \quad
\wt{\cR}_j^0 = \bigsqcup_{i=1}^{k_j^0} \wt{c}_{ij}^0(\theta). \label{partition0}
\ee
Next, we define a family of grid decompositions of $Q$. Put
\[
N_L:=4\lfloor L\rfloor,\quad b_q:=-L+\frac{qL}{2\lfloor L\rfloor},\quad q\in\Z\cap[0,N_L].
\]
Thus $b_0=-L$ and $b_{N_L}=L$. We keep these two boundary grid points
fixed and allow only the interior grid points to move. More precisely, set
\[
T:=\left(-\frac1{10},\frac1{10}\right)^{(n-2)\times(N_L-1)}.
\]
For $\tau=(\tau_{\alpha q})\in T$, where
$\alpha\in\Z\cap[1,n-2]$ and $q\in\Z\cap[1,N_L-1]$, define
\[
B_\alpha(\tau):=\{b_0,b_{N_L}\}\cup\{b_q+\tau_{\alpha q}:\ q=1,\ldots,N_L-1\}\subset[-L,L],
\]
and
\[
\mathcal B(\tau):=B_1(\tau)\times...\times B_{n-2}(\tau)\subset\overline Q.
\]
The set $\mathcal B(\tau)$ is the set of vertices of a rectangular grid
in $Q$. The boundary vertices remain on $\partial Q$, while each interior grid hyperplane has a small independent displacement. Since the displacement is smaller than the grid size, this still gives a cell decomposition of $\ol{Q}$. We denote its $j$-skeleton by
\be
\ol{Q} = \bigsqcup_{j=0}^{n-2} \wt{\cR}_j^1(\tau), \quad
\wt{\cR}_j^1 = \bigsqcup_{i=1}^{k_j^1} \wt{c}_{ij}^1(\tau),\label{partition1}
\ee
where the cells $\widetilde c_{ij}^1(\tau)$ are relatively open
$j$-dimensional coordinate rectangles. Define $\wt{g}_{\wh{\va}} \in H^1(\pa B_1^2 \times Q, \R^m)$ by
\[
\wt{g}_{\wh{\va}}(y, z) := g_{\wh{\va}}(y, \Phi^{-1}(z)), \quad y \in \pa B_1^2,\,\,z \in Q.
\]
By \eqref{chooseL} and the bi-Lipschitz bound for $\Phi$,
\be
E_{\wh{\va}}(\wt{g}_{\wh{\va}}, \pa B_1^2 \times Q)
\leq C(n) E_{\wh{\va}}(g_{\wh{\va}}, \Ga_{1,\delta^{-1}}). \label{QchooseL}
\ee

The decompositions \eqref{partition0} and \eqref{partition1} of $\partial B_1^2$ and $Q$ induce, by taking products, a cell decomposition of $\partial B_1^2\times Q$. We write
\[
\pa B_1^2 \times \ol{Q}
= \bigsqcup_{j=0}^{n-1} \wt{\cR}_j(\theta,\tau), \quad
\wt{\cR}_j(\theta,\tau) := \bigsqcup_{i=1}^{k_j} \wt{c}_{ij}(\theta,\tau).
\]
For each $j \in \Z \cap [0, n-1]$ and $i \in \Z \cap [1, k_j]$, there exist $j_0 \in \{0,1\}$, $i_0 \in \Z \cap [1, k_{j_0}^0]$, $j_1 \in \Z \cap [0, n-2]$, and $i_1 \in \Z \cap [1, k_{j_1}^1]$
such that
\[
\wt{c}_{ij}(\theta,\tau) = \wt{c}_{i_0 j_0}^0(\theta) \times \wt{c}_{i_1 j_1}^1(\tau).
\]
We now choose the parameters in the decompositions. Since only finitely many types of skeletons occur, Fubini's theorem applies to the parameters
$(\theta,\tau)\in(0,\pi)\times T$ gives a pair $(\theta_*,\tau_*)\in(0,\pi)\times T$ such that, for every
$j\in\Z\cap[1,n-1]$,
\[
E_{\widehat\varepsilon}(\widetilde g_{\widehat\varepsilon},\widetilde\cR_j(\theta_*,\tau_*))\leq C(n)
E_{\widehat\varepsilon}(\widetilde g_{\widehat\varepsilon},\partial B_1^2\times Q),
\]
and
\[
\int_{\widetilde\cR_j(\theta_*,\tau_*)}f(\widetilde g_{\widehat\varepsilon})\ud\mathcal H^j\leq C(n)\int_{\partial B_1^2\times Q}f(\widetilde g_{\widehat\varepsilon})\ud\mathcal H^{n-1}.
\]
From now on, we will fix this pair and denote
\[
\widetilde\cR_j^*:=\widetilde\cR_j(\theta_*,\tau_*),\quad\widetilde c_{ij}^*
:=\widetilde c_{ij}(\theta_*,\tau_*).
\]
Pulling the decomposition back by $\Phi$, we obtain a cell decomposition of $\ol{\Gamma}_{1,L} $. Namely, for $j=\Z\cap[0,n-1]$ and $i\in\Z\cap[1,k_j]$, define
\be
c_{ij}:=(\operatorname{Id}_{\partial B_1^2}\times\Phi^{-1})(\widetilde c_{ij}^*),\quad\cR_j:=\bigsqcup_{i=1}^{k_j}c_{ij}.\label{cijkj}
\ee
Then
\[
\ol{\Gamma}_{1,L}=\bigsqcup_{j=0}^{n-1}\cR_j.
\]
Moreover, there exists $C_0=C_0(n)>0$ such that the following properties hold.
\begin{itemize}
\item Each $j$-cell $c_{ij}$ is bi-Lipschitz equivalent to $B_1^j$: there exists $\phi_{ij} :\ol{c}_{ij} \to\ol{B}_1^j$ satisfying
\be
\|\na\phi_{ij}\|_{L^\ift(c_{ij})} + \|\na(\phi_{ij}^{-1})\|_{L^\ift(B_1^j)} \leq C_0. \label{phiijestimate1}
\ee

\item For any $j \in \Z \cap [0, n-2]$ and $i \in \Z \cap [1, k_j]$,
\be
\#\{\ell \in \Z \cap [1, k_{j+1}] : c_{ij} \subset \pa c_{\ell,j+1}\} \leq C_0. \label{ellnumber1}
\ee

\item For any $j \in \Z \cap [0, n-1]$,
\begin{align}
E_{\wh{\va}}(g_{\wh{\va}}, \cR_j)
&\leq C_0 E_{\wh{\va}}(g_{\wh{\va}}, \pa B_1^2 \times B_L^{n-2}), \label{FubiniEgva} \\
\int_{\wt{\cR}_j} f(g_{\wh{\va}}) \ud\HH^j
&\leq C_0 \int_{\pa B_1^2 \times B_L^{n-2}} f(g_{\wh{\va}}) \ud\HH^{n-1}. \label{Fubinifgva}
\end{align}
\end{itemize}
This decomposition of $\pa B_1^2 \times B_L^{n-2}$ induces a partition of
$(B_1^2 \backslash B_{\f{1}{2}}^2) \times B_L^{n-2}$ by
\be
\begin{aligned}
&(B_1^2 \backslash B_{\f{1}{2}}^2) \times B_L^{n-2}
= \bigsqcup_{j=0}^{n-1} \widehat{\cR}_j, \quad
\widehat{\cR}_j = \bigsqcup_{i=1}^{k_j} \wh{c}_{ij}, \\
&\wh{c}_{ij} = \left\{(y,z) \in (B_1^2 \backslash B_{\f{1}{2}}^2) \times B_L^{n-2}
: \frac{y}{|y|} \in c_{ij}\right\}.
\end{aligned} \label{induceddecomposition1}
\ee
With this partition established, we carry out Steps 2--4 in close parallel with the proof of Lemma \ref{Luckhauslemma}; we present the arguments in full for completeness.

\smallskip
\noindent\textbf{Step 2: Construction of $v_{\wh{\va}}$ on $\cR_1$.} By reasoning exactly as in Step 2 of the proof of Lemma \ref{Luckhauslemma}, we obtain the analog of \eqref{dvalimit0}:
\[
 \lim_{\wh{\va}\to 0^+} \sup_{\cR_1}\dist(g_{\wh{\va}}, \cN)=0
\]
Hence, there exists $\wh{\va}_1 = \wh{\va}_1(f, M, \cN, n) \in (0,1)$ such that for any $\wh{\va} \in (0, \wh{\va}_1)$, the nearest-point projection $\Pi_{\cN}$ is well-defined on $g_{\wh{\va}}(\cR_1)$. Set $v_{\wh{\va}} := \Pi_{\cN} \circ g_{\wh{\va}}$ on $\cR_1$. Then for any $x \in \cR_1$,
\be
v_{\wh{\va}}(x) = \Pi_{\cN}(g_{\wh{\va}}(x)),
\quad |v_{\wh{\va}}(x) - g_{\wh{\va}}(x)| < \delta_f, \label{leqdeltaf1}
\ee
where $\delta_f > 0$ is given by \eqref{A31}.

\smallskip
\noindent\textbf{Step 3: Construction of $v_{\wh{\va}}$ on $\Ga_{1,L}$.}
We extend $v_{\wh{\va}}$ to $\cR_j$ for $j \in \Z \cap [2, n-1]$ by induction, showing at each stage that
\be
\int_{\cR_j} |\na_{\cR_j} v_{\wh{\va}}|^2 \ud\HH^j
\leq C E_{\wh{\va}}(g_{\wh{\va}}, \Ga_{1,\delta^{-1}}), \label{inductionvolva1}
\ee
where $C = C(f, M, \cN, n) > 0$. Taking $j = n-1$ in \eqref{inductionvolva1} gives \eqref{Vvaescylinder} and completes the construction of $v_{\wh{\va}}$.

Let $ j=2 $. We choose $\eta_1 = \eta_1(f, M, \cN, n) \in (0,1)$ small enough so that whenever the energy bound in \eqref{scalingproperties} holds, it follows from almost the same arguments in Step 3 in the proof of Lemma \ref{Luckhauslemma} that the restriction $v_{\wh{\va}}|_{\pa c_{i2}}$ is null-homotopic in $\cN$ for any $i \in \Z \cap [1, k_2]$. Then, Lemma \ref{ExtensionLemma1} extends $v_{\wh{\va}}|_{\pa c_{i2}}$ to $c_{i2}$ with
\[
\int_{c_{i2}} |\na_{c_{i2}} v_{\wh{\va}}|^2 \ud\HH^2
\leq C \int_{\pa c_{i2}} |\na_{\pa c_{i2}} v_{\wh{\va}}|^2 \ud\HH^1.
\]
Summing over $i$ and applying \eqref{ellnumber1} and \eqref{FubiniEgva}, we have
\[
\int_{\cR_2} |\na_{\cR_2} v_{\wh{\va}}|^2 \ud\HH^2
\leq C E_{\wh{\va}}(g_{\wh{\va}}, \Ga_{1,\delta^{-1}}),
\]
confirming \eqref{inductionvolva1} for $j = 2$.

Suppose $v_{\wh{\va}}$ is defined on $\cR_{j-1}$ and \eqref{inductionvolva1} holds for $j-1$. For each $c_{ij}$ with $i \in \Z \cap [1, k_j]$, extend $v_{\wh{\va}}$ by homogeneous extension from $\pa c_{ij}$:
\[
v_{\wh{\va}}(x):=v_{\wh{\va}}\circ(\phi_{ij}^{-1})\circ\(\f{\phi_{ij}(x)}{|\phi_{ij}(x)|}\),\quad x\in c_{ij}\backslash\{\phi_{ij}^{-1}(0^j)\},
\]
where $ \phi_{ij} $ is as in \eqref{phiijestimate1}. Direct computation gives
\[
\int_{\cR_j} |\na_{\cR_j} v_{\wh{\va}}|^2 \ud\HH^j
\leq C \int_{\cR_{j-1}} |\na_{\cR_{j-1}} v_{\wh{\va}}|^2 \ud\HH^{j-1}
\leq C E_{\wh{\va}}(g_{\wh{\va}}, \Ga_{1,\delta^{-1}}),
\]
where the last inequality uses the inductive hypothesis \eqref{inductionvolva1} for $j-1$. This completes the construction of $v_{\wh{\va}}$ on $\Ga_{1,L}$.

\smallskip
\noindent\textbf{Step 4: Construction of $\vp_{\wh{\va}}$.}
Recall the decomposition \eqref{induceddecomposition1} of $(B_1^2 \backslash B_{\f{1}{2}}^2) \times B_L^{n-2}$. We construct $\vp_{\wh{\va}}$ on $\wh{\cR}_j$ for $j \in \Z \cap [1, n-1]$ inductively, establishing at each stage
\be
E_{\wh{\va}}(\vp_{\wh{\va}}, \wh{\cR}_j) \leq C E_{\wh{\va}}(g_{\wh{\va}}, \Ga_{1,\delta^{-1}}), \label{inductionEstimate111}
\ee
where $C > 0$ depends only on $f, M, \cN, n$. Taking $j = n-1$ gives \eqref{Wvaestcylinder}, while $j = n-2$ gives \eqref{WpluesminusCcylinder}.

Let $j=1$. Recall that $k_1$ denotes the number of $1$-cells in $ \cR_1 $ as in \eqref{cijkj}. For any $ i\in[1,k_1] $ and $ x=(y,z)\in\wh{\cR}_1 $ such that $ \f{y}{|y|}\in\cR_1 $, set
\[
\vp_{\wh{\va}}(x):=(2|y|-1)g_{\wh{\va}}\(\f{y}{|y|},z\)+2(1-|y|)v_{\wh{\va}}\(\f{y}{|y|},z\),
\]
which also defines $\vp_{\wh{\va}}$ on $\wh{\cR}_0$. Using \eqref{leqdeltaf1} and \eqref{fBconvex}, we have
\[
f(\vp_{\wh{\va}})\leq C(1-2|y|)^2f\(g_{\wh{\va}}\(\f{y}{|y|},z\)\).
\]
It follows from \eqref{Fubinifgva} that
\be
\wh{\va}^{-2}\int_{\wh{\cR}_1}f(\vp_{\wh{\va}})\ud\HH^2\leq CE_{\wh{\va}}(g_{\wh{\va}},\Ga_{1,\delta^{-1}}).\label{fvpfresults1}
\ee
By \eqref{Fubinifgva} and \eqref{leqdeltaf1}, for $\wh{\va} \in (0, \wh{\va}_1)$ with $ \wh{\va}_1\in(0,1) $ being small,
\[
\int_{\wh{\cR}_1}|\na_{\wh{\cR}_1}\vp_{\wh{\va}}|^2\ud\HH^2\leq CE_{\wh{\va}}(g_{\wh{\va}},\Ga_{1,\delta^{-1}}).
\]
Together with \eqref{fvpfresults1}, this implies \eqref{inductionEstimate111} for $j = 1$.

Suppose $\vp_{\wh{\va}}$ is constructed on $\wh{\cR}_{j-1}$ and
\be
E_{\wh{\va}}(\vp_{\wh{\va}},\wh{\cR}_{j-1})\leq CE_{\wh{\va}}(g_{\wh{\va}},\Ga_{1,\delta^{-1}}).\label{inductionHypo111}
\ee
By \eqref{phiijestimate1}, for each $i \in \Z \cap [1, k_j]$ there exists a bi-Lipschitz map $\Phi_{ij} : \ol{\wh{c}}_{ij} \to\ol{B}_1^{j+1}$ with
\[
\|\na\Phi_{ij}\|_{L^{\ift}(\wh{c}_{ij})}+\|\na(\Phi_{ij}^{-1})\|_{L^{\ift}(B_1^{j+1})}\leq C(n).
\]
Extend $\vp_{\wh{\va}}$ to $\wh{c}_{ij}$ by homogeneous extension from $\pa\wh{c}_{ij}$:
\[
\vp_{\wh{\va}}(x):=\vp_{\wh{\va}}\circ(\Phi_{ij}^{-1})\circ\(\f{\Phi_{ij}(x)}{|\Phi_{ij}(x)|}\),\quad x\in\wh{c}_{ij}\backslash\{\Phi_{ij}^{-1}(0^{j+1})\}.
\]
Direct computation and \eqref{inductionHypo111} give
\[
E_{\wh{\va}}(\vp_{\wh{\va}}, \wh{\cR}_j)
\leq CE_{\wh{\va}}(g_{\wh{\va}}, \Ga_{1,\delta^{-1}}),
\]
establishing \eqref{inductionEstimate111} for $j$ and completing the induction.
\qed

\section{Clearing-out theory}

In this section, we establish the following clearing-out property: If the energy of a local minimizer $u_\va$ on a ball is small relative to $|\log\va|$, then the energy on the concentric half-ball is bounded independently of $\va$. More precisely, we have

\begin{prop}[Clearing-out]\label{clearingout}
Let $r > 0$ and $x \in \R^n$. Assume that $\{u_{\va}\}_{\va \in (0,1)}$ is a family of local minimizers of \eqref{GLfunctional} in $B_{2r}(x)$ such that
\[
\|u_{\va}\|_{L^\ift(B_{4r}(x))} \leq M
\]
for some $ M>0 $. Then, there exist $\eta_1, \ol{\va}_1 \in (0,1)$ and $C > 0$, depending only on $f, M, \cN$, and $n$, such that if $\va \in (0, \ol{\va}_1 r)$, $\eta \in (0, \eta_1)$, and
\[
E_{\va}(u_{\va}, B_r(x)) \leq \eta r^{n-2} \log\frac{r}{\va},
\]
then
\[
E_{\va}(u_{\va}, B_{\f{r}{2}}(x)) \leq C r^{n-2}.
\]
\end{prop}

The proof follows the strategy of \cite[Proposition 8]{Can17}, adapted to arbitrary dimensions. The proof is based on Lemma~\ref{Luckhauslemma}, which, by a comparison argument, yields a differential inequality for the energy of a minimizer~$u_\va$ on the ball~$B_\rho(x)$, with~$\f{r}{2} \leq \rho \leq r$. However, since Lemma \ref{Luckhauslemma} differs from its three-dimensional counterpart, the differential inequality we obtain is also slightly different from its three-dimensional analog. Nevertheless, this modified inequality is still enough to obtain uniform energy bounds on~$B_{\f{r}{2}}(x)$, as the following lemma shows.

\begin{lem}\label{AnODE}
Let $y : [\frac{1}{2}, 1] \to [0, +\ift)$ be a non-decreasing function. Suppose that there exists a measurable set $D \subset [\frac{1}{2}, 1]$ with the following properties.
\begin{enumerate}[label=$(\mathrm{\roman*})$]
\item $\HH^1(D) \geq \frac{1}{4}$.
\item\label{iilog} For every $t \in D$,
\[
y(t) \leq C_0\left((y'(t))^{\f{1}{2}}+ \left(\frac{y'(t)}{|\log\wh{\va}|}\right)^{\frac{n-2}{n-3}} |\log\wh{\va}| + 1\right),
\]
where $\wh{\va} \in (0,1)$ is a fixed parameter.
\end{enumerate}
Then, there exists $\eta_1 = \eta_1(C_0, n) > 0$ such that the following property holds. If there exists $t_0 \in [\frac{1}{2}, 1]$ satisfying
\[
\HH^1\left(D \cap \left[\frac{1}{2}, t_0\right]\right) \geq \frac{1}{8}
\quad \text{and} \quad
y(t_0) \leq \eta_1 |\log\wh{\va}|,
\]
then $y(\frac{1}{2}) \leq C$, where $C > 0$ depends only on $C_0$ and $n$.
\end{lem}

\begin{proof}
If there exists $t_* \in [\frac{1}{2}, 1]$ such that $y(t_*) \leq 2C_0$, then since $y(\cdot)$ is non-decreasing, $y(\f{1}{2}) \leq 2C_0$ and the proof is complete. We may therefore assume $y(t) > 2C_0$ for any $t \in [\frac{1}{2}, 1]$. Set $y_*(t) := y(t) - C_0 > C_0$ on $[\frac{1}{2}, 1]$. For any $t \in D$, the assumption \ref{iilog} implies the inequality
\[
y_*(t) \leq 2C_0 \max\left\{(y_*'(t))^{\f{1}{2}},
\left(\frac{y_*'(t)}{|\log\wh{\va}|}\right)^{\frac{n-2}{n-3}} |\log\wh{\va}|\right\},
\]
which rearranges to
\[
y_*'(t) \geq \min\left\{\left(\frac{y_*(t)}{2C_0}\right)^{2},
\left(\frac{y_*(t)}{2C_0 |\log\wh{\va}|}\right)^{\frac{n-3}{n-2}} |\log\wh{\va}|\right\}
\quad \text{for any } t \in D.
\]
In particular, for $t\in D$ we have
\[
\frac{y_*'(t)}{\left(\dfrac{y_*(t)}{2C_0}\right)^{2}}
+ \frac{y_*'(t)}{\left(\dfrac{y_*(t)}{2C_0|\log\wh{\va}|}\right)^{\frac{n-3}{n-2}}|\log\wh{\va}|}
\geq 1.
\]
Noting that the left-hand side is the derivative of
\[
F(t) := \frac{-4C_0^2}{y_*(t)}
+ (n-2)(2C_0)^{\frac{n-3}{n-2}} \left(\frac{y_*(t)}{|\log\wh{\va}|}\right)^{\frac{1}{n-2}},
\]
we obtain
\be
F'(t) \geq 1 \quad \text{for any } t \in D. \label{Fprimetgeq1}
\ee
Since $y_*$ is non-decreasing, so is $F$. By assumption,
\be
\HH^1\(D\cap\left[\f{1}{2},t_0\right]\)\geq\f{1}{8},\quad y(t_0)\leq\eta_1|\log\wh{\va}|.\label{geq18}
\ee
The fundamental theorem of calculus, \eqref{Fprimetgeq1}, and \eqref{geq18} give
\[
F(t_0)-F\(\f{1}{2}\)\geq\int_DF'(t)\ud t\geq \HH^1\left(D \cap \left[\frac{1}{2}, t_0\right]\right)\geq\f{1}{8}.
\]
Since by \eqref{geq18}, $y(t_0) \leq \eta_1 |\log\wh{\va}|$, we have
\[
F(t_0) \leq (n-2)(2C_0)^{\frac{n-3}{n-2}} \eta_1^{\frac{1}{n-2}}.
\]
Therefore,
\[
F\(\f{1}{2}\)\leq (n-2)(2C_0)^{\f{n-3}{n-2}}\eta_1^{\f{1}{n-2}}-\f{1}{8}.
\]
Choosing $\eta_1 = \eta_1(C_0, n) \in (0,1)$ small enough that the right-hand side is at most $-\frac{1}{16}$, we get $ F(\f{1}{2})\leq-\f{1}{16} $, which forces
\[
-\f{4C_0^2}{y_*\(\f{1}{2}\)}\leq-\f{1}{16},
\]
implying
\[
y\(\frac{1}{2}\) \leq y_*\(\frac{1}{2}\) + C_0 \leq 64C_0^2 + C_0 =: C.
\]
Then we complete the proof.
\end{proof}

\begin{proof}[Proof of Proposition \ref{clearingout}]
By translation, we may take $x = 0$. Assume
\[
E_{\va}(u_{\va},B_r)\leq\eta r^{n-2}\log\f{r}{\va}.
\]
Define the good radius set
\[
D^{\va}:=\left\{t\in\left[\f{r}{2},r\right]:E_{\va}(u_{\va},\pa B_t)\leq 4\eta r^{n-3}\log\f{r}{\va}\right\}.
\]
By Fubini's theorem, $\HH^1(D^{\va}) \geq \frac{r}{4}$. There exists
$\ol{\va}_1 = \ol{\va}_1(n) \in (0,1)$ such that for any $\va \in (0, \ol{\va}_1 r)$ and $t \in D^{\va}$,
\be
E_{\va}(u_{\va},\pa B_t)\leq 8\eta t^{n-3}\log\f{t}{\va}.\label{Dvadef}
\ee
Throughout what follows, $\ol{\va}_1$ may be further reduced finitely many times without change of notation.

Fix $t \in D^{\va}$ and set $\ol{\va} := \f{\va}{t}$ and $u_{\ol{\va}}(x) := u_{\va}(tx)$. Then
\[
E_{\ol{\va}}(u_{\ol{\va}}, \pa B_1) \leq 8\eta |\log\ol{\va}|.
\]
Applying Lemma \ref{Luckhauslemma} with $(\eta_0, \ol{\va}_0) = (\eta_0,\ol{\va}_0)(f,M,\cN,n) \in (0,1)$ and
\[
h(\ol{\va})=\(\f{E_{\ol{\va}}(u_{\ol{\va}},\pa B_1)+\eta_0}{\eta_0|\log\ol{\va}|}\)^{\f{1}{n-3}},
\]
we obtain $v_{\ol{\va}} \in H^1(\pa B_1, \cN)$ and
$\vp_{\ol{\va}} \in H^1(B_1 \backslash B_{1-h(\ol{\va})}, \R^m)$
such that, whenever $8\eta \in (0, \frac{\eta_0}{100^n})$ and $\ol{\va} \in (0, \ol{\va}_0)$,
\[
\vp_{\ol{\va}}(x) = u_{\ol{\va}}(x)
\quad \text{and} \quad
\vp_{\ol{\va}}((1-h(\ol{\va}))x) = v_{\ol{\va}}(x)
\quad \text{for }\HH^{n-1}\text{-a.e.}x \in \pa B_1,
\]
and
\begin{gather}
\frac{1}{2} \int_{\pa B_1} |\na_{\pa B_1} v_{\ol{\va}}|^2 \ud\HH^{n-1}
\leq CE_{\ol{\va}}(u_{\ol{\va}}, \pa B_1), \label{VboundL1} \\
E_{\ol{\va}}(\vp_{\ol{\va}}, B_1 \backslash B_{1-h(\ol{\va})})
\leq C \left(\frac{E_{\ol{\va}}(u_{\ol{\va}}, \pa B_1) + \eta_0}{\eta_0 |\log\ol{\va}|}\right)^{\frac{1}{n-3}}
E_{\ol{\va}}(u_{\ol{\va}}, \pa B_1). \label{WboundL1}
\end{gather}
Choose $\ol{\va}_1 \in (0, \frac{\ol{\va}_0}{2})$. Applying Lemma \ref{ExtensionLemma2}, there exists $w_{\ol{\va}} \in H^1(B_1, \cN)$ with $w_{\ol{\va}}|_{\pa B_1} = v_{\ol{\va}}$ and
\be
E_{\ol{\va}}(w_{\ol{\va}},B_1)=\f{1}{2}\int_{B_1}|\na w_{\ol{\va}}|^2\leq C\(\int_{\pa B_1}|\na_{\pa B_1}v_{\ol{\va}}|^2\ud\HH^{n-1}\)^{\f{1}{2}}\stackrel{\eqref{VboundL1}}{\leq}C(E_{\ol{\va}}(u_{\ol{\va}},\pa B_1))^{\f{1}{2}}.\label{wvaestimate}
\ee
Define the competitor
\[
\wt{u}_{\ol{\va}}:=\left\{\begin{aligned}
&\vp_{\ol{\va}}(x)&\text{ if }&x\in B_1\backslash B_{1-h(\ol{\va})},\\
&w_{\ol{\va}}\(\f{x}{1-h(\ol{\va})}\)&\text{ if }&x\in B_{1-h(\ol{\va})}.
\end{aligned}\right.
\]
Since $\wt{u}_{\ol{\va}}|_{\pa B_1} = u_{\ol{\va}}|_{\pa B_1}$, the minimality of $u_{\ol{\va}}$ gives
\[
E_{\ol{\va}}(u_{\ol{\va}},B_1)\leq E_{\ol{\va}}(\wt{u}_{\ol{\va}},B_1).
\]
Using \eqref{WboundL1} and \eqref{wvaestimate}, we obtain
\be
E_{\ol{\va}}(u_{\ol{\va}}, B_1)
\leq C\left((E_{\ol{\va}}(u_{\ol{\va}}, \pa B_1))^{\f{1}{2}}
+ \eta_0^{-\frac{1}{n-3}}
\left(\frac{E_{\ol{\va}}(u_{\ol{\va}},\pa B_1) + \eta_0}{|\log\ol{\va}|}\right)^{\frac{n-2}{n-3}}
|\log\ol{\va}|\right). \label{Evascaled}
\ee
Scaling back to $u_{\va}$ and using $t \in D^{\va} \subset [\frac{r}{2}, r]$ together with $\va \in (0, \ol{\va}_1 r)$ for $\ol{\va}_1$ sufficiently small, we get
\be
E_{\va}(u_{\va}, B_t)
\leq C\left(t^{\frac{n-1}{2}}(E_{\va}(u_{\va}, \pa B_t))^{\f{1}{2}}
+ \eta_0^{-\frac{1}{n-3}}
\left(\frac{E_{\va}(u_{\va}, \pa B_t) + \eta_0 t^{n-3}}{\log\frac{r}{\va}}\right)^{\frac{n-2}{n-3}}
\log\frac{r}{\va}\right), \label{Evaleqestimate}
\ee
where $C = C(f, M, \cN, n) > 0$.

Set $ \wh{\va}:=\f{\va}{r} $, $ u_{\wh{\va}}(x):=u_{\va}(rx) $, and
\[
y(s) := E_{\wh{\va}}(u_{\wh{\va}}, B_s) + \frac{\eta_0 s^{n-2}}{n-2}, \quad s \in [0,1].
\]
From \eqref{Evaleqestimate}, the rescaled set $D^{\wh{\va}} := r^{-1} D^{\va} \subset [\frac{1}{2}, 1]$ satisfies $\HH^1(D^{\wh{\va}}) \geq \frac{1}{4}$, and for any $ s\in D^{\wh{\va}} $,
\[
y(s) \leq C_0\left((y'(s))^{\f{1}{2}}
+ \left(\frac{y'(s)}{|\log\wh{\va}|}\right)^{\frac{n-2}{n-3}} |\log\wh{\va}| + 1\right),
\]
where $C_0 = C_0(f, M, \cN, n) > 0$. Since $\HH^1(D^{\wh{\va}}) \geq \frac{1}{4}$,
there exists $s_0 \in D^{\wh{\va}}$ with
\[
\HH^1\left(D^{\wh{\va}} \cap \left[\f{1}{2}, s_0\right]\right) \geq \frac{1}{8}.
\]
By \eqref{Dvadef},
\[
0 \leq \frac{y(s_0)}{|\log\wh{\va}|} \leq 8\eta + \frac{\eta_0}{|\log\wh{\va}|}.
\]
Choosing $(\eta_1, \ol{\va}_1) = (\eta_1, \ol{\va}_1)(f, M, \cN, n)\in(0,1)$ sufficiently small and applying Lemma \ref{AnODE}, we conclude $y(\frac{1}{2}) \leq C$ whenever $\wh{\va} \in (0, \ol{\va}_1)$ and $\eta \in (0, \eta_1)$. Scaling back to $u_{\va}$, we obtain
\[
E_{\va}(u_{\va}, B_{\f{r}{2}}) \leq C r^{n-2},
\]
where $C = C(f, M, \cN, n) > 0$. This completes the proof.
\end{proof}

\begin{rem}
Heuristically, at least, the energy divided by~$|\log\va|$ provides an upper bound for the $(n-2)$-dimensional area of singularities that arise in the limit as $\va\to 0^+$. (This heuristic can be made rigorous under suitable assumptions on~$\cN$; see~\cite{ABO05, CO21}.) Keeping this observation into account, it appears that the scaling of~\eqref{Evaleqestimate} --- specifically, the power of~$\f{n - 2}{n - 3}$ on the right-hand side --- is consistent with the isoperimetric inequality; see~e.g.~\cite[4.2.10]{Fed69}. This is not surprising, since the proof of Proposition~\ref{clearingout} is based on ideas very similar to those in \cite{Fed69}; in particular, both proofs rely on grids of comparable size.
\end{rem}

\section{The singular set}\label{SectionSingularset1}

Throughout this section, $\{u_{\va}\}_{\va\in(0,1)} \subset H^1(\om, \R^m)$ denotes a family of local minimizers of \eqref{GLfunctional}. Recalling the measure $\mu_\va$ defined by \eqref{definemeasurewithuva} and the assumption \eqref{assumptionbound}, we have
\[
\sup_{\va\in(0,\frac{1}{2})} \mu_{\va}(\ol{\om}) \leq M\left(1 + \frac{1}{\log 2}\right).
\]
Hence, there exist a non-negative Radon measure $\mu_* \in (C^0(\ol{\om}))'$ and a sequence $\va_i \to 0^+$ such that
\be
\mu_{\va_i} \wc^* \mu_* \quad \text{in } (C^0(\ol{\om}))' \label{munconve}
\ee
as $i \to +\ift$. Define the singular set as $S_* := \supp(\mu_*)$, which is a closed subset of $\ol{\om}$.

\subsection{Preliminary properties of the singular set} In this subsection, we establish the basic structural properties of $S_*$. The first lemma shows that a ball on which $\mu_*$ has subcritical mass must lie entirely outside the singular set.

\begin{lem}\label{smallmu0re}
Let $\eta_1 > 0$ be given by Proposition \ref{clearingout}. Let $x \in \om$ and $r \in (0, \dist(x, \pa\om))$. If $\mu_*(\ol{B}_r(x)) < \eta_1 r^{n-2}$, then $ \mu_*(B_{\f{r}{2}}(x))=0 $, that is, $ B_{\f{r}{2}}(x)\subset\om\backslash S_* $.
\end{lem}
\begin{proof}
By \cite[Theorem 1.40]{EG15} and the boundedness of $\om$, for any closed set $F \subset \ol{\om}$ and open set $G \subset \ol{\om}$,
\be
\mu_*(F) \geq \limsup_{i\to+\ift} \mu_{\va_i}(F),
\quad
\mu_*(G) \leq \liminf_{i\to+\ift} \mu_{\va_i}(G). \label{weakconFG}
\ee
Suppose $\mu_*(\ol{B}_r(x)) < \eta_1 r^{n-2}$. From \eqref{munconve} and \eqref{weakconFG},
\begin{align*}
\limsup_{i\to+\ift}
\frac{E_{\va_i}(u_{\va_i}, B_r(x))}{r^{n-2} \log\frac{r}{\va_i}}
&= \limsup_{i\to+\ift}
\frac{E_{\va_i}(u_{\va_i}, \ol{B}_r(x))}{r^{n-2} \log\frac{1}{\va_i}}
\cdot \frac{\log\frac{1}{\va_i}}{\log r + \log\frac{1}{\va_i}} 
\leq \frac{\mu_*(\ol{B}_r(x))}{r^{n-2}} < \eta_1.
\end{align*}
Proposition \ref{clearingout} then gives
\[
E_{\va_i}(u_{\va_i}, B_{\f{r}{2}}(x)) \leq C r^{n-2}
\]
for any sufficiently large $i$ with $\va_i \in (0, \ol{\va}_1 r)$. Applying \eqref{weakconFG} to the open ball $B_{\f{r}{2}}(x)$,
\[
\mu_*(B_{\f{r}{2}}(x))
\leq \liminf_{i\to+\ift} \mu_{\va_i}(B_{\f{r}{2}}(x))
\leq \liminf_{i\to+\ift} \frac{C r^{n-2}}{\log\frac{1}{\va_i}} = 0.
\]
This completes the proof.
\end{proof}

The next lemma establishes a monotonicity formula for the $(n-2)$-density ratio of $\mu_*$, which will be used to identify the singular set.

\begin{lem}\label{monomu0}
For $x \in \om$ and $r \in (0, \dist(x, \pa\om))$, define
\[
\Theta_r^{n-2}(\mu_*, x)
:= \frac{\mu_*(\ol{B}_r(x))}{\w_{n-2} r^{n-2}},
\]
where $\w_{n-2} := \HH^{n-2}(B_1^{n-2})$ is the volume of the unit $(n-2)$-ball. Then, the function $r \mapsto \Theta_r^{n-2}(\mu_*, x)$ is non-decreasing.
\end{lem}
\begin{proof}
Let $0 < r_1 < r_2 < \dist(x, \pa\om)$ and choose $\delta \in (0,1)$ with $(1+\delta)r_1 < r_2$. Using \eqref{munconve}, \eqref{weakconFG}, and Proposition \ref{Mo}, we have
\begin{align*}
\Theta_{r_1}^{n-2}(\mu_*, x)
&= (1+\delta)^{n-2} \Theta_{(1+\delta)r_1}^{n-2}(\mu_*, x) \\
&\leq \liminf_{i\to+\ift}\frac{E_{\va_i}(u_{\va_i}, B_{(1+\delta)r_1}(x))}{\w_{n-2}((1+\delta)r_1)^{n-2} \log\frac{1}{\va_i}}\cdot (1+\delta)^{n-2} \\
&\leq \limsup_{i\to+\ift}\frac{E_{\va_i}(u_{\va_i}, \ol{B}_{r_2}(x))}
{\w_{n-2} r_2^{n-2} \log\frac{1}{\va_i}}\cdot (1+\delta)^{n-2}
\leq (1+\delta)^{n-2} \Theta_{r_2}^{n-2}(\mu_*, x).
\end{align*}
Letting $\delta \to 0^+$ gives $\Theta_{r_1}^{n-2}(\mu_*, x) \leq
\Theta_{r_2}^{n-2}(\mu_*, x)$.
\end{proof}

Since $r \mapsto \Theta_r^{n-2}(\mu_*, x)$ is non-decreasing and bounded from below, the limit as $r \to 0^+$ exists. We define the $(n-2)$-density of $\mu_*$ at $x$ by
\be
\Theta^{n-2}(\mu_*, x)
:= \lim_{r\to 0^+} \Theta_r^{n-2}(\mu_*, x), \quad x \in \om. \label{Thetadensity}
\ee

The following lemma characterizes $S_* \cap \om$ in terms of this density.

\begin{lem}\label{Thetamu0eta0}
With $\eta_1 > 0$ as in Proposition \ref{clearingout}, we have
\[
S_* \cap \om
= \{x \in \om : \Theta^{n-2}(\mu_*, x) > 0\}
= \left\{x \in \om : \Theta^{n-2}(\mu_*, x) \geq
  \frac{\eta_1}{2\w_{n-2}}\right\}.
\]
\end{lem}

\begin{proof}
Since $S_* = \supp(\mu_*)$, a point $x \in \om$ belongs to $\om \backslash S_*$ if and only if $\mu_*(B_r(x)) = 0$ for some $r > 0$. Therefore,
\be
\left\{x \in \om : \Theta^{n-2}(\mu_*, x) \geq \frac{\eta_1}{2\w_{n-2}}\right\}\subset\{x \in \om : \Theta^{n-2}(\mu_*, x) > 0\}
\subset S_* \cap \om. \label{Slinesubset}
\ee
For the reverse inclusion, let $x \in S_* \cap \om$. Then $\mu_*(B_r(x)) > 0$ for any $r > 0$. By Lemma \ref{smallmu0re}, $\mu_*(\ol{B}_{2r}(x)) \geq\eta_1 (2r)^{n-2}$ for any $r > 0$, so
\[
\Theta^{n-2}(\mu_*, x)
= \lim_{r\to 0^+} \frac{\mu_*(\ol{B}_r(x))}{\w_{n-2} r^{n-2}}
\geq \frac{\eta_1 (2r)^{n-2}}{\w_{n-2}(2r)^{n-2}}
= \frac{\eta_1}{\w_{n-2}}.
\]
Hence,
$$
S_* \cap \om \subset \left\{x \in \om : \Theta^{n-2}(\mu_*, x) \geq
\frac{\eta_1}{2\w_{n-2}}\right\}.
$$
Together with \eqref{Slinesubset}, this completes the proof.
\end{proof}

We can now establish the rectifiability of $S_*$ and an integral representation
of $\mu_*$.

\begin{prop}\label{proprectifiable}
The set $S_* \cap \om$ is $(n-2)$-rectifiable and
\be
\HH^{n-2}(S_* \cap \om) \leq C, \label{HausdorffOmega}
\ee
where $C > 0$ depends only on $f, M, \cN$, and $n$. Moreover, for every Borel set
$A \subset \ol{\om}$, there holds
\be
\mu_*(A \cap \om)= \int_{(A \cap \om) \cap S_*} \Theta^{n-2}(\mu_*, y) \ud\HH^{n-2}(y).
\label{mustarrepresent}
\ee
\end{prop}

\begin{proof}
Fix $\delta > 0$. For each $x \in S_* \cap \om$, Lemma \ref{Thetamu0eta0} gives $r_x \in (0, \min\{\frac{1}{10}\dist(x, \pa\om), \delta\})$ with $\mu_*(\ol{B}_{r_x}(x)) \geq \frac{\eta_1}{2} r_x^{n-2}$. The collection $\{\ol{B}_{r_x}(x)\}_{x \in S_* \cap \om}$ covers $S_* \cap \om$. By Vitali's covering lemma, there exists a countable disjoint subcollection $\{\ol{B}_{r_i}(x_i)\}_{i \in \Z_+}$ such that $S_* \cap \om \subset \cup_i \ol{B}_{10r_i}(x_i)$. Therefore, we have
\[
\HH_\delta^{n-2}(S_* \cap \om)
\leq C \sum_{i=1}^{+\ift} r_i^{n-2}
\leq \frac{C}{\eta_1} \sum_{i=1}^{+\ift} \mu_*(\ol{B}_{r_i}(x_i))
\leq \frac{C}{\eta_1} \mu_*\left(\bigcup_{i \in \Z_+} \ol{B}_{r_i}(x_i)\right)
\leq \frac{C(n)}{\eta_1} \mu_*(\om)
\leq C,
\]
where the last bound uses \eqref{assumptionbound} and \eqref{weakconFG}. Letting $\delta \to 0^+$ yields \eqref{HausdorffOmega}.

By Lemma \ref{Thetamu0eta0}, for any $x \in \supp(\mu_*)$ the density
$\Theta^{n-2}(\mu_*, x)$ is positive and bounded below by
$\frac{\eta_1}{2\w_{n-2}}$. By \cite[Theorem 5.3]{Pre87}, $S_* = \supp(\mu_*)$ is $(n-2)$-rectifiable and $\mu_*$ is absolutely continuous with respect to $\HH^{n-2} \llcorner (S_* \cap \om)$. Hence, there exists an $\HH^{n-2}$-measurable function $\Theta : \om \to \R$ such that
\[
(\mu_* \llcorner \om)(A)
= \int_{(S_* \cap \om) \cap A} \Theta(y) \ud\HH^{n-2}(y)
\]
for every Borel set $A \subset \ol{\om}$. The Besicovitch differentiation theorem gives
\be
\lim_{r\to 0^+}
\frac{\mu_*(\ol{B}_r(x))}{\HH^{n-2}(\ol{B}_r(x) \cap S_*)} = \Theta(x)
\label{muHausdorff}
\ee
for $\HH^{n-2}$-a.e. $x \in S_* \cap \om$. Since $S_*$ is $(n-2)$-rectifiable and \eqref{HausdorffOmega} holds, \cite[Theorem 3.2.19]{Fed69} gives
\[
\lim_{r\to 0^+}
\frac{\HH^{n-2}(\ol{B}_r(x) \cap S_*)}{\w_{n-2} r^{n-2}} = 1
\]
for $ \HH^{n-2} $-a.e. $ x\in S_*\cap\om $. Combining this with \eqref{Thetadensity} and \eqref{muHausdorff}, we obtain $\Theta(x) = \Theta^{n-2}(\mu_*,x)$ for $\HH^{n-2}$-a.e. $x \in S_* \cap \om$, which gives \eqref{mustarrepresent}.
\end{proof}

The following lemma shows that $u_{\va_i}$ has a uniformly bounded energy away from the singular set, which is the basis for the strong convergence in Proposition~\ref{propombackSstar} below.

\begin{lem}\label{CUleq}
For any open set $ U\subset\subset\om\backslash S_* $, there exists $ C>0 $, depending only on $ f,M,\cN,n $, and $ U $, such that for sufficiently large $ i\in\Z_+ $, $ E_{\va_i}(u_{\va_i},U)\leq C $.
\end{lem}
\begin{proof}
We first consider the case~$U$ is a ball, $U = B_r(x) \subset\subset \om \backslash S_*$. We have $\mu_*(\ol{B}_r(x)) = 0$. By \eqref{weakconFG}, we have
\[
\limsup_{i\to+\ift}
\frac{E_{\va_i}(u_{\va_i}, B_r(x))}{r^{n-2} \log\frac{r}{\va_i}}
\leq r^{2-n} \mu_*(\ol{B}_r(x)) = 0.
\]
Proposition \ref{clearingout} gives
\[
E_{\va_i}(u_{\va_i}, B_{\f{r}{2}}(x)) \leq
Cr^{n-2}
\]
for all indices $I$ large enough. The result for general $U$ follows by covering $\ol{U}$ with finitely many balls $\{B_{\f{r_x}{2}}(x)\}_{x \in \ol{U}}$ with $B_{r_x}(x) \subset\subset \om \backslash S_*$.
\end{proof}

We now establish the strong convergence of $u_{\va_i}$ away from $S_*$ to a minimizing harmonic map.

\begin{prop}\label{propombackSstar}
There exists $u_* \in H_{\loc}^1(\om \backslash S_*, \cN)$, a local minimizer of the Dirichlet energy \eqref{DirichletEnergy}, such that the following properties hold.
\begin{enumerate}[label=$(\theenumi)$]
\item Up to a subsequence, $u_{\va_i} \to u_*$ strongly in $H_{\loc}^1(\om \backslash S_*, \R^m)$.
\item The singular set $\sing(u_*)$ of $u_*$ is $(n-3)$-rectifiable, and $u_{\va_i} \to u_*$ in $C_{\loc}^0(\om \backslash (S_* \cup \sing(u_*)))$.
\end{enumerate}
\end{prop}

\begin{proof}
This follows directly from Proposition \ref{boundedenergyuse} and
Lemma \ref{CUleq}.
\end{proof}

\subsection{The limiting measure \texorpdfstring{$ \mu_* $}{} is a stationary varifold}\label{stationaryvarifoldsec}

We now identify $\mu_* \llcorner \om$ with a stationary integer-multiplicity rectifiable varifold. By Proposition \ref{proprectifiable}, $\mu_* \llcorner \om$ is $(n-2)$-rectifiable. By \cite[Theorem 11.6]{Sim83}, for $\mu_*$-a.e. $x \in \om$ there exists a unique $(n-2)$-dimensional subspace $V_x \subset \R^n$ such that
\be
\lim_{\lda\to 0^+} \lda^{2-n}
\int_{\R^n} \vp\left(\frac{y-x}{\lda}\right)\ud\mu_*(y)
= \Theta^{n-2}(\mu_*, x) \int_{V_x} \vp(y) \ud\HH^{n-2}(y) \label{tangentplane}
\ee
for every $\vp \in C_0^0(\R^n)$. We call $V_x$ the approximate tangent space of $\mu_*$ at $x$ and write $T_x\mu_*$ for it. Let $\MM_n(\R)$ denote the space of $n \times n$ real matrices, and for a subspace $V \subset \R^n$ let $P_V \in\MM_n(\R)$ be the orthogonal projection onto $V$. Set
\[
\bG(n, n-2)
:= \{P_V \in \MM_n(\R) : V \subset \R^n \text{ is an } (n-2)\text{-dimensional subspace}\}.
\]
For $\mu_*$-a.e. $x \in \om$, let $A(x)$ denote the orthogonal projection from $\R^n$ onto $T_x\mu_*$. The varifold associated with $\mu_* \llcorner \om$ is the push-forward
\[
V_* := (\op{Id}, A)_\#(\mu_* \llcorner \om) \in (C_0^0(\om \times \bG(n,n-2)))',
\]
or equivalently, for any Borel set $E \subset \om \times \bG(n,n-2)$,
\[
V_*(E) := \mu_*(\{x \in \om : (x, A(x)) \in E\}).
\]
We show that the varifold $V_*$ associated with $\mu_*$ is stationary.

\begin{prop}\label{stationarity}
The varifold $V_*$ associated with $\mu_*$ is stationary, i.e.,
\be
\int_{\om} A_{jk}(x) \pa_k \vp_j(x) \ud \mu_*(x) = 0 \label{Ajkpakvpj}
\ee
for every $\vp \in C_0^1(\om, \R^n)$.
\end{prop}

\begin{proof}
The argument follows the standard strategy of \cite{AS97}; see also
\cite[Theorem IX.1]{BBO01} for an analogous treatment of the complex
Ginzburg--Landau model. Define the matrix-valued map $A^\va = (A_{jk}^\va) : \om \to \MM_n(\R)$ by
\[
A_{jk}^\va:= \frac{1}{|\log\va|}(e_{\va}(u_{\va})\delta_{jk} - \pa_j u_{\va}:\pa_k u_{\va}),
\quad j, k \in \Z \cap [1,n].
\]
This tensor satisfies
\be
(A^\va)^{\op{T}} = A^\va,\quad
\tr A^\va= \frac{1}{|\log\va|}(ne_{\va}(u_{\va})-|\na u_\va|^2)
\geq (n-2)\frac{\ud\mu_{\va}}{\ud x},
\quad
|A^\va| \leq C\frac{\ud\mu_{\va}}{\ud x}, \label{Avaproperty}
\ee
where $C = C(n) > 0$. For any unit vector $v \in \Ss^{n-1}$, we have
\[
A^\va[v,v]
:= A_{jk}^\va v_j v_k
= \frac{1}{|\log\va|}(e_\va(u_\va) - |v \cdot \na u_\va|^2)
\leq \frac{\ud\mu_\va}{\ud x}.
\]
Testing \eqref{stressidentity} with $\vp \in C_0^1(\om, \R^n)$ gives
\be
\int_\om A_{jk}^\va(x) \pa_k \vp_j(x) \ud x = 0. \label{Avavpj}
\ee

Set $\va = \va_i$. By \eqref{assumptionbound} and \eqref{Avaproperty}, passing to a subsequence if necessary,
\be
A^{\va_i} \ud x \wc^* A^0
\quad \text{in } (C_0^0(\om, \MM_n(\R)))'. \label{Availimit}
\ee
Since $|A^0| \leq C \mu_* \llcorner \om$, the measure $A^0$ is absolutely
continuous with respect to $\mu_* \llcorner \om$. Hence there exists
$A^* \in L^1(\om, \MM_n(\R); \mu_*)$ such that
$A^0 = A^*(x)\mu_* \llcorner \om$ as measures in $(C_0^0(\om, \MM_n(\R)))'$.

Passing to the limit $\va_i \to 0^+$ in \eqref{Avavpj} via \eqref{Availimit},
and using \eqref{Avaproperty}, we obtain: for $\mu_*$-a.e. $x \in \om$,
\be
(A^*(x))^{\op{T}} = A^*(x), \quad
\tr A^*(x) \geq n-2, \quad
\max_{i \in \Z \cap [1,n]} \lambda_i(A^*(x)) \leq 1, \label{Astareigenvalues}
\ee
where $\{\lambda_i(A^*(x))\}_{i=1}^n$ are the eigenvalues of $A^*(x)$, and
\be
\int_\om A^*_{jk}(x) \pa_k \vp_j(x) \ud \mu_*(x) = 0 \label{Astarmustar}
\ee
for every $\vp \in C_0^1(\om, \R^n)$.

We now identify $A^*$ with the projection $A$ onto the approximate tangent space of~$\mu_*$. Let $x_0 \in \om$ be a Lebesgue point of $A^*$ with respect to $\mu_*$ at which the approximate tangent space $T_{x_0}\mu_*$ exists. By Proposition \ref{proprectifiable} and the analysis at the beginning of this subsection, this holds for $\mu_*$-a.e. $x_0 \in \om$. Fix $\lambda \in (0, \dist(x_0, \pa\om))$. By \eqref{Astarmustar}, we have
\[
\lambda^{2-n}
\int_{\R^n} A_{jk}^*(x)\pa_k \vp_j\left(\frac{x - x_0}{\lambda}\right)
\ud\mu_*(x) = 0
\]
for any $\vp \in C_0^1(B_1, \R^n)$. Letting $\lambda \to 0^+$ and applying \eqref{tangentplane}, we obtain
\be
\Theta^{n-2}(\mu_*, x_0) A_{jk}^*(x_0)
\int_{T_{x_0}\mu_*} \pa_k \vp_j \ud \HH^{n-2} = 0
\label{Thetamux0}
\ee
for every $\vp \in C_0^1(B_1, \R^n)$. By Lemma \ref{Thetamu0eta0},
$\Theta^{n-2}(\mu_*, x) > 0$ for $\mu_*$-a.e. $x \in \om$, so \eqref{Thetamux0}
forces
\[
A_{jk}^*(x_0) \int_{T_{x_0}\mu_*} \pa_k \vp_j \ud \HH^{n-2} = 0.
\]
By \cite[Lemma 3.9]{AS97}, this implies that at least two eigenvalues of $A^*(x)$ vanish for $\mu_*$-a.e. $x \in \om$. Combining this information with \eqref{Astareigenvalues}, we deduce that the eigenvalues of $A^*(x)$ are one, with multiplicity~$n-2$, and zero, with multiplicity~$2$. Therefore, $ A^*(x) $ is the orthogonal projection onto an $(n-2)$-dimensional subspace of $\R^n$. Since $\Theta^{n-2}(\mu_*, x) > 0$ for $\mu_*$-a.e. $x \in \om$, \cite[Theorem 3.3]{AS97} and \eqref{Astarmustar} imply that the varifold $\wt{V} := (\op{Id}, A^*)_\#(\mu_* \llcorner \om)$ is $(n-2)$-rectifiable. It follows that the $(n-2)$-dimensional subspace onto which $A^*(x)$ projects must coincide with $T_x\mu_*$ for $\mu_*$-a.e. $x \in \om$, i.e., $A^*(x) = A(x)$ for $\mu_*$-a.e. $x\in\om$. Substituting into \eqref{Astarmustar} yields \eqref{Ajkpakvpj}, and the stationarity of $V_*$ follows.
\end{proof}

\section{Energy quantization and the proof of Theorem \ref{maintheorem}}\label{SectionSingularset1prime}

The goal of this section is to establish the energy quantization property stated in Proposition \ref{densitydescrete} and to complete the proof of Theorem \ref{maintheorem}. We begin with the quantization of the density function on the singular set $S_*$. We recall that~$[\Ss^1, \cN]$ is the set of free homotopy classes of continuous maps~$\Ss^1\to\cN$, as defined in Definition \ref{definitionfreehomotopyclass}, and that $\abs{\,\cdot\,}_*$ is the norm-like function defined in Definition \ref{DefinitionEsg}.

\begin{prop}\label{densitydescrete}
Under the same assumptions as in Theorem \ref{maintheorem}, for $\HH^{n-2}$-a.e. $x \in S_*$, we have $\Theta^{n-2}(\mu_*, x) \in \{|\sg|_* : \sg \in [\Ss^1, \cN]\}$.
\end{prop}

The following lemma is the key technical ingredient in the proof of Proposition \ref{densitydescrete}. It asserts that, under suitable energy bounds on the boundary data, the energy of a minimizer on a cylinder is quantized and depends on the homotopy class in $ [\Ss^1,\cN] $ of the boundary datum, up to a controlled error. Recall that for $ r,L>0 $, $\Lda_{r,L}$ and $\Ga_{r,L}$ are defined in \eqref{LdaGarLdef}.

\begin{lem}\label{chooseseveralcases}
Let $\delta \in (0, \frac{1}{10}]$. Assume that for each $\ol{\va} \in (0,1)$, the map $g_{\delta,\ol{\va}} \in H^1(\pa\Lda_{\delta,1}, \R^m)$ satisfies
\begin{align}
\|g_{\delta,\ol{\va}}\|_{L^{\ift}(\pa\Lda_{\delta,1},\R^m)} &\leq M,
\label{fGboundboundfinal} \\
E_{\ol{\va}}(g_{\delta,\ol{\va}}, B_\delta^2 \times \pa B_1^{n-2})
&\leq M \log\frac{\delta}{\ol{\va}}, \label{GvaboundaryH1final}
\end{align}
for some constant $M > 0$. There exist constants $\eta_1, \wh{\va}_1 \in (0,1)$,
depending only on $f$, $M$, $\cN$, and $n$, such that for any $\eta \in (0, \eta_1)$
and $\ol{\va} \in (0, \wh{\va}_1 \delta)$, if
\be
E_{\ol{\va}}(g_{\delta,\ol{\va}}, \Ga_{\delta,1})
\leq \delta^{n-3} \eta \log\frac{\delta}{\ol{\va}}, \label{GvaboundaryH1final1}
\ee
and $u_{\delta,\ol{\va}}$ is a global minimizer of $E_{\ol{\va}}(\cdot,\Lda_{\delta,1})$ with boundary condition $u_{\delta,\ol{\va}} = g_{\delta,\ol{\va}}$ on $\pa\Lda_{\delta,1}$, then there exists $\sg \in [\Ss^1, \cN]$ such that
\be
\left| E_{\ol{\va}}(u_{\delta,\ol{\va}}, \Lda_{\delta,1})
- |\sg|_* \HH^{n-2}(B_1^{n-2}) \log\frac{\delta}{\ol{\va}} \right|
\leq \al(M, \delta, \eta) \log\frac{\delta}{\ol{\va}} + C,
\label{EolvaapproximateEstar}
\ee
where $C > 0$ depends only on $f$, $M$, $\cN$, $n$, and
\[
\al(M, \delta, \eta) \leq C(\delta + \delta^{n-4}\eta).
\]
\end{lem}

\begin{proof}
Applying Lemma \ref{Luckhauslemma1}, we fix $(\eta_1, \wh{\va}_1) =
(\eta_1, \wh{\va}_1)(f, M, \cN, n) \in (0,1)$ such that for any
$\eta \in (0, \eta_1)$ and $\ol{\va} \in (0, \wh{\va}_1\delta)$, there exist
$r \in (1-\delta, 1-\frac{\delta}{2})$, together with maps
$v_{\delta,\ol{\va}} \in H^1(\pa B_\delta^2 \times B_r^{n-2}, \cN)$,
$\vp_{\delta,\ol{\va}} \in H^1((B_\delta^2 \backslash B_{\frac{\delta}{2}}^2)\times B_r^{n-2}, \R^m)$, and $\vp_{\delta,\ol{\va}}^* \in H^1((B_\delta^2 \backslash B_{\frac{\delta}{2}}^2)\times \pa B_r^{n-2}, \R^m)$,
satisfying estimates \eqref{Vvaes1}, \eqref{Wvaest1}, and \eqref{WpluesminusC11}, respectively. By Lemma \ref{H1homotopyclass}, the homotopy class $[v_{\delta,\ol{\va}}]_\cN$ is well-defined; we denote it by $\sg \in [\Ss^1, \cN]$. We establish the estimate \eqref{EolvaapproximateEstar} in two steps, proving the upper and lower bounds of $E_{\ol{\va}}(u_{\delta,\ol{\va}}, \Lda_{\delta,1})$
in turn.

\smallskip
\noindent\textbf{Step 1. Upper bound on $E_{\ol{\va}}(u_{\delta,\ol{\va}},\Lda_{\delta,1})$.}
We introduce the decomposition of $ \Lda_{\delta,1} $ by
\[
\Lda_{\delta,1}=\Lda_{\delta}\cup\Lda_{\delta}^0\cup E_{\delta},
\]
where
\begin{align*}
\Lda_\delta &:= B_\delta^2 \times (B_1^{n-2} \backslash B_r^{n-2}), \\
\Lda_\delta^0 &:= B_{\frac{\delta}{2}}^2 \times B_r^{n-2}, \\
E_\delta &:= (B_\delta^2 \backslash B_{\frac{\delta}{2}}^2) \times B_r^{n-2}.
\end{align*}
Since $v_{\delta,\ol{\va}}\in H^1(\pa B_\delta^2 \times B_r^{n-2}, \cN)$ and $\ol{\va} \in (0, \wh{\va}_1\delta) \subset (0, \frac{\delta}{2})$, we apply Lemma \ref{cylinderextension1} to $g(\cdot) = v_{\delta,\ol{\va}}(2\cdot)$ in
$\Lda_{\frac{\delta}{2}, r}$. This yields a map $\wh{u}_{\delta,\ol{\va}}^{(0)}$
with $\wh{u}_{\delta,\ol{\va}}^{(0)}|_{\pa B_{\frac{\delta}{2}}^2 \times B_r^{n-2}}
= v_{\delta,\ol{\va}}(2\cdot)$, satisfying
\be
\begin{aligned}
E_{\ol{\va}}(\wh{u}_{\delta,\ol{\va}}^{(0)}, \Lda_\delta^0)
&\leq Cr\left(\frac{r}{\delta} + \frac{\delta}{r}\right)
\|\na_{\Ga_{\delta,1}} v_{\delta,\ol{\va}}\|_{L^2(\Ga_{\delta,r})}^2
+ |\sg|_* \HH^{n-2}(B_r^{n-2}) \log\frac{\delta}{2\ol{\va}} + Cr^{n-2} \\
&\leq Cr\left(\frac{r}{\delta} + \frac{\delta}{r}\right)
E_{\ol{\va}}(v_{\delta,\ol{\va}}, \Ga_{\delta,r})
+ |\sg|_* \HH^{n-2}(B_1^{n-2}) \log\frac{\delta}{\ol{\va}} + C \\
&\leq (|\sg|_* \HH^{n-2}(B_1^{n-2})
+ (\delta + \delta^{-1})\delta^{n-3}\eta) \log\frac{\delta}{\ol{\va}} + C
\end{aligned}\label{inteU0hat}
\ee
and
\be
\begin{aligned}
E_{\ol{\va}}(\wh{u}_{\delta,\ol{\va}}^{(0)}, B_{\frac{\delta}{2}}^2
\times \pa B_r^{n-2})
&\leq C\left(\frac{r}{\delta} + \frac{\delta}{r}\right)
\|\na_{\Ga_{\delta,r}} v_{\delta,\ol{\va}}\|_{L^2(\Ga_{\delta,r})}^2
+ |\sg|_* \HH^{n-3}(\pa B_r^{n-2}) \log\frac{\delta}{2\ol{\va}} + Cr^{n-3} \\
&\leq (|\sg|_* \HH^{n-3}(\pa B_1^{n-2})
+ C(\delta^{-1} + \delta)\delta^{n-3}\eta)\log\frac{\delta}{\ol{\va}} + C.
\end{aligned}\label{zplusmiunsU0}
\ee
We now construct a competitor map on $\Lda_{\delta,1}$. Define
$\wt{g}_{\delta,\ol{\va}} \in H^1(\pa\Lda_\delta, \R^m)$ by
\begin{align*}
\wt{g}_{\delta,\ol{\va}}(x) := \begin{cases}
g_{\delta,\ol{\va}}(x)
& \text{if } x \in \pa\Lda_\delta \cap \pa\Lda_{\delta,1}, \\
\vp_{\delta,\ol{\va}}^*(x)
& \text{if } x \in (B_\delta^2 \backslash B_{\frac{\delta}{2}}^2)
\times \pa B_r^{n-2}, \\
\wh{u}_{\delta,\ol{\va}}^{(0)}(x)
& \text{if } x \in B_{\frac{\delta}{2}}^2 \times \pa B_r^{n-2}.
\end{cases}
\end{align*}
Combining \eqref{Vvaes1}, \eqref{Wvaest1}, \eqref{WpluesminusC11},
\eqref{GvaboundaryH1final}, and \eqref{zplusmiunsU0}, we obtain
\be
\begin{aligned}
E_{\ol{\va}}(\wt{g}_{\delta,\ol{\va}}, \pa\Lda_\delta)
&\leq E_{\ol{\va}}(\wh{u}_{\delta,\ol{\va}}^{(0)},
B_{\frac{\delta}{2}}^2 \times \pa B_r^{n-2})
+ E_{\ol{\va}}(\vp_{\delta,\ol{\va}}^*,
(B_\delta^2 \backslash B_{\frac{\delta}{2}}^2) \times \pa B_r^{n-2}) \\
&\quad + E_{\ol{\va}}(g_{\delta,\ol{\va}}, \pa\Lda_\delta \cap \Ga_{\delta,1})
+ E_{\ol{\va}}(g_{\delta,\ol{\va}}, B_\delta^2 \times \pa B_1^{n-2}) \\
&\leq (C + C(\delta^{-1}+\delta)\delta^{n-3}\eta)
\log\frac{\delta}{\ol{\va}}
+ C\delta^{n-3}\eta \log\frac{\delta}{\ol{\va}}
+ M\log\frac{\delta}{\ol{\va}} + C \\
&\leq (C + C(1+\delta^{-1}+\delta)\delta^{n-3}\eta)
\log\frac{\delta}{\ol{\va}} + C.
\end{aligned}\label{Gvadeltaplusminus}
\ee
To extend $\wt{g}_{\delta,\ol{\va}}$ to the interior of $\Lda_\delta$, we exploit the
product structure. Note that
\[
\Lda_\delta = B_\delta^2 \times [1-r,1] \times \Ss^{n-3}\quad\text{and}\quad\pa\Lda_\delta = \pa(B_\delta^2 \times [1-r,1]) \times \Ss^{n-3}.
\]
There exists a bi-Lipschitz map $\Phi: \ol{B_\delta^2 \times [1-r,1]} \to \ol{B}_\delta^3$ with
\[
\|\nabla\Phi\|_{L^\infty(B_\delta^2 \times [1-r,1])}
+ \|\nabla(\Phi^{-1})\|_{L^\infty(B_\delta^3)} \leq C(n).
\]
For $x = (y, z) \in ((B_\delta^2 \times [1-r,1])\backslash\{\Phi^{-1}(0^3)\}) \times \Ss^{n-3}$, we set
\[
\wh{w}_{\delta,\ol{\va}}(y,z)
:= \wt{g}_{\delta,\ol{\va}}\left(\Phi^{-1}\left(\frac{\delta\Phi(y)}{|\Phi(y)|}\right), z\right).
\]
By \eqref{Gvadeltaplusminus}, the energy of this cone-type extension satisfies
\be
E_{\ol{\va}}(\wh{w}_{\delta,\ol{\va}}, \Lda_\delta)
\leq C\delta E_{\ol{\va}}(\wt{g}_{\delta,\ol{\va}}, \pa\Lda_\delta)
\leq C(\delta + (\delta + \delta^2 + 1)\delta^{n-3}\eta)
\log\frac{\delta}{\ol{\va}} + C. \label{U0plusminusinmte}
\ee
We now assemble the global competitor by setting
\begin{align*}
\wh{u}_{\delta,\ol{\va}}(x) := \begin{cases}
\wh{u}_{\delta,\ol{\va}}^{(0)}(x) & \text{if } x \in \Lda_\delta^0, \\
\vp_{\delta,\ol{\va}}(x) & \text{if } x \in E_\delta, \\
\wh{w}_{\delta,\ol{\va}}(x) & \text{if } x \in \Lda_\delta.
\end{cases}
\end{align*}
Combining \eqref{Wvaest1}, \eqref{inteU0hat}, and \eqref{U0plusminusinmte}, we get
\[
E_{\ol{\va}}(\wh{u}_{\delta,\ol{\va}}, \Lda_{\delta,1})
\leq (|\sg|_* \HH^{n-2}(B_1^{n-2})
+ C(\delta + \delta^{2-n}\eta))\log\frac{\delta}{\ol{\va}} + C.
\]
The upper bound in \eqref{EolvaapproximateEstar} then follows from minimality:
\[
E_{\ol{\va}}(u_{\delta,\ol{\va}}, \Lda_{\delta,1})
\leq E_{\ol{\va}}(\wh{u}_{\delta,\ol{\va}}, \Lda_{\delta,1}).
\]

\smallskip
\noindent\textbf{Step 2. Lower bound on $E_{\ol{\va}}(u_{\delta,\ol{\va}},
\Lda_{\delta,1})$.}
To apply the slicing argument of Lemma \ref{LowerBound}, we construct an auxiliary map that incorporates both $u_{\delta,\ol{\va}}$ and the transition map $\vp_{\delta,\ol{\va}}$. In cylindrical coordinates $(\rho, \theta, y) \in [0,\delta] \times \Ss^1 \times B_1^{n-2}$,
define
\begin{align*}
\wh{u}_{\delta,\ol{\va}}^{(1)}(\rho,\theta,y) := \left\{\begin{aligned}
&u_{\delta,\ol{\va}}(2\rho, \theta, y)
& \text{ if }& \rho \in \left[0, \tfrac{\delta}{2}\right),\,\,
\theta \in \Ss^1,\,\, y \in B_r^{n-2}, \\
&\vp_{\delta,\ol{\va}}\left(\tfrac{3\delta}{2} - \rho,\,\,\theta, y\right)
& \text{ if }& \rho \in \left[\tfrac{\delta}{2}, \delta\right),\,\,
\theta \in \Ss^1,\,\, y \in B_r^{n-2}.
\end{aligned}\right.
\end{align*}
Then $\wh{u}_{\delta,\ol{\va}}^{(1)} \in H^1(B_\delta^2 \times B_r^{n-2}, \R^m)$ and
\be
\wh{u}_{\delta,\ol{\va}}^{(1)}(x) = v_{\delta,\ol{\va}}(x)
\quad\text{for } \HH^{n-1}\text{-a.e. } x \in \pa B_\delta^2 \times B_r^{n-2}.
\label{boundaryconUV11}
\ee
Moreover, a direct computation together with \eqref{Wvaest1} gives
\be
\begin{aligned}
E_{\ol{\va}}(\wh{u}_{\delta,\ol{\va}}^{(1)}, \Lda_{\delta,1})
&\leq E_{\ol{\va}}(u_{\delta,\ol{\va}}, \Lda_{\delta,1})
+ E_{\ol{\va}}(\vp_{\delta,\ol{\va}},
(B_\delta^2 \backslash B_{\frac{\delta}{2}}^2) \times B_r^{n-2}) \\
&\leq E_{\ol{\va}}(u_{\delta,\ol{\va}}, \Lda_{\delta,1})
+ C\delta^{n-3}\eta \log\frac{\delta}{\ol{\va}}.
\end{aligned}\label{U1Udeltaminus}
\ee
Applying Fubini's theorem and Lemma \ref{LowerBound} to each slice
$B_\delta^2 \times \{y\}$, we obtain that
\[
E_{\ol{\va}}(\wh{u}_{\delta,\ol{\va}}^{(1)}, B_\delta^2 \times \{y\})
+ C\delta E_{\ol{\va}}(v_{\delta,\ol{\va}}, \pa B_\delta^2 \times \{y\})
\geq |\sg|_* \log\frac{\delta}{\ol{\va}} - C
\]
for $\HH^{n-2}$-a.e. $y \in B_r^{n-2}$. Integrating over $y \in B_r^{n-2}$ yields
\[
E_{\ol{\va}}(\wh{u}_{\delta,\ol{\va}}^{(1)}, \Lda_{\delta,1})
+ C\delta \int_{\pa B_\delta^2 \times B_r^{n-2}} |\na v_{\delta,\ol{\va}}|^2 \ud\HH^2
\geq (|\sg|_* \HH^{n-2}(B_1^{n-2}) - C\delta)\log\frac{\delta}{\ol{\va}} - C.
\]
Combining this with \eqref{Vvaes1}, \eqref{GvaboundaryH1final1}, and
\eqref{U1Udeltaminus}, we conclude that
\[
E_{\ol{\va}}(u_{\delta,\ol{\va}}, \Lda_{\delta,1})
\geq (|\sg|_* \HH^{n-2}(B_1^{n-2})
- C(\delta+1)\delta^{n-3}\eta)\log\frac{\delta}{\ol{\va}} - C,
\]
which gives the lower bound and completes the proof.
\end{proof}

\begin{proof}[Proof of Proposition \ref{densitydescrete}]
Recall that
\[
\mu_{\va_i} := \frac{e_{\va_i}(u_{\va_i})}{|\log\va_i|} \ud x
\wc^* \mu_* := \Theta^{n-2}(\mu_*, x) \HH^{n-2} \llcorner (S_* \cap \om)
\quad \text{in } (C_0^0(\om))'.
\]
Fix a point $x_0 \in S_*$ at which there exists a unique approximate tangent space $V_{x_0} \subset \R^n$ satisfying \eqref{tangentplane}, and at which $\Theta^{n-2}(\mu_*, x_0) > 0$. By Lemma \ref{Thetamu0eta0} and Propositions \ref{proprectifiable} and \ref{stationarity}, the set of such points has full $\HH^{n-2}$-measure in $S_*$.

We introduce the rescaled measures. For $s > 0$, define $\cD_{x_0,s}(\mu_*)$ on $s^{-1}(\om - x_0)$ by
\[
\cD_{x_0,s}(\mu_*)(E) := s^{2-n} \mu_*(x_0 + sE)
\]
for any Borel set $E \subset s^{-1}(\om - x_0)$. By \eqref{tangentplane}, there exists a sequence $s_j \to 0^+$ such that
\be
\cD_{x_0,s_j}(\mu_*) \wc^* \cD_{x_0,0}(\mu_*)
:= \Theta^{n-2}(\mu_*, x_0) \HH^{n-2} \llcorner V_{x_0}
\quad \text{in } (C_0^0(\R^n))' \label{nu0con}
\ee
as $s_j \to 0^+$. Without loss of generality, up to a rotation, we assume $x_0 = 0^n$ and
\[
V_{x_0} = \{x \in \R^n : x_1 = x_2 = 0\}.
\]
Let $\delta, \eta \in (0, \frac{1}{10})$ be parameters to be chosen later. From \eqref{nu0con} we have
\[
\cD_{x_0,0}(\mu_*)((\ol{B}_\delta^2 \backslash B_{\frac{\delta}{2}}^2)\times [-1,1]) = 0.
\]
By passing to a further subsequence, we find $t_i \to 0^+$ such that
\be
\frac{\va_i}{t_i} \to 0^+ \quad \text{and} \quad
\cD_{x_0,t_i}(\mu_{\va_i}) \wc^* \cD_{x_0,0}(\mu_*)
\quad \text{in } (C_0^0(\R^n))'. \label{vaitilimit}
\ee
By \eqref{weakconFG} and \eqref{vaitilimit},
\[
\lim_{i \to +\infty} \cD_{x_0,t_i}(\mu_{\va_i})(
(\ol{B}_\delta^2 \backslash B_{\frac{\delta}{2}}^2) \times [-1,1]) = 0,
\]
which gives
\be
\lim_{i \to +\infty}
\frac{t_i^{2-n} E_{\va_i}(u_{\va_i},
(\ol{B}_{\delta t_i}^2 \backslash B_{\frac{\delta t_i}{2}}^2)
\times [-t_i, t_i])}{|\log\va_i|} = 0. \label{annulusvanish}
\ee
Applying Fubini's theorem to \eqref{annulusvanish}, we select
$\wt{t}_i \in (\frac{t_i}{2}, t_i)$ such that, for any sufficiently large $i$ (depending on $\delta$ and $\eta$),
\[
\frac{E_{\va_i}(u_{\va_i}, \Ga_{\delta\wt{t}_i, t_i})}{|\log\va_i|}
\leq \delta^{n-3} t_i^{n-3} \eta.
\]
Replacing $t_i$ by $\wt{t}_i$ (which we relabel as $t_i$), for any
sufficiently large $i$ we have
\be
E_{\va_i}(u_{\va_i}, \Ga_{\delta t_i, t_i})
\leq \delta^{n-3} t_i^{n-3} \eta \log\frac{\delta t_i}{\va_i}.
\label{nn0geq}
\ee
Using \eqref{assumptionbound} and Proposition \ref{Mo}, by taking $i$ larger
if necessary,
\[
E_{\va_i}(u_{\va_i}, \Lda_{\delta t_i, t_i})
\leq C t_i^{n-2} E_{\va_i}(u_{\va_i}, \om)
\leq C t_i^{n-2} \log\frac{\delta t_i}{\va_i}.
\]
A further application of Fubini's theorem yields $r_i \in [\frac{3t_i}{4}, t_i]$
and a subsequence (not relabeled) such that
\be
E_{\va_i}(u_{\va_i},B_{\delta t_i}^2 \times \pa B_{r_i}^{n-2})
\leq C t_i^{n-3} \log\frac{\delta t_i}{\va_i}, \label{upperlowb}
\ee
where $C = C(M, \om, n) > 0$. Without loss of generality, we set $r_i = t_i$; the general case follows by the same argument after a translation.

We now rescale to a fixed domain. Define $\ol{\va}_i := \frac{\va_i}{t_i}$
and $u_{\delta, \ol{\va}_i}(y) := u_{\va_i}(t_i y)$ for $y \in \Lda_{\delta,1}$.
From \eqref{assumptionbound}, \eqref{nn0geq}, and \eqref{upperlowb}, the rescaled map satisfies
\begin{gather*}
\|u_{\delta, \ol{\va}_i}\|_{L^\infty(\Lda_{\delta,1})} \leq M', \\
E_{\ol{\va}_i}(u_{\delta, \ol{\va}_i}, \Ga_{\delta,1})
\leq \delta^{n-3} \eta \log\frac{\delta}{\ol{\va}_i}, \\
E_{\ol{\va}_i}(u_{\delta, \ol{\va}_i}, B_\delta^2 \times \pa B_1^{n-2})
\leq M' \log\frac{\delta}{\ol{\va}_i},
\end{gather*}
where $M'$ depends only on $f$, $M$, $\cN$, $\om$, and $n$. We may therefore
apply Lemma \ref{chooseseveralcases} with boundary data
$g_{\delta, \ol{\va}_i} := u_{\delta, \ol{\va}_i}|_{\pa\Lda_{\delta,1}}$.
Fix $\eta_1, \wh{\va}_1 \in (0,1)$ as provided by that lemma, depending only
on $f$, $M'$, $\cN$, and $n$. By \eqref{vaitilimit}, $\ol{\va}_i \to 0^+$,
so $\ol{\va}_i \in (0, \wh{\va}_1\delta)$ for any large $i$. Choosing
$\eta \in (0, \eta_1)$, the conclusion \eqref{EolvaapproximateEstar} applies:
there exists $\sg \in [\Ss^1, \cN]$ such that
\be
\left| E_{\va_i}(u_{\va_i}, \Lda_{\delta t_i, t_i})
- |\sg|_* \HH^{n-2}(B_1^{n-2}) \log\frac{\delta t_i}{\va_i} \right|
\leq \al \log\frac{\delta t_i}{\va_i} + C, \label{recallimportant}
\ee
where $C > 0$ depends only on $f$, $M$, $\cN$, $\om$, $n$, and
\[
\al \leq C(\delta + \delta^{n-4}\eta).
\]
Here we used the change of variables $y =\f{x}{t_i}$ to pass between
$E_{\va_i}(u_{\va_i}, \Lda_{\delta t_i, t_i})$ and
$E_{\ol{\va}_i}(u_{\delta,\ol{\va}_i}, \Lda_{\delta,1})$.

Dividing \eqref{recallimportant} by $|\log\va_i|$ and passing $i \to +\infty$,
using $n \geq 4$ and \eqref{vaitilimit}, we obtain
\[
\bigl|\cD_{x_0,0}(\mu_*)(\Lda_{\delta,1})
- |\sg|_* \HH^{n-2}(B_1^{n-2})\bigr| \leq C(\delta + \eta).
\]
Since $\cD_{x_0,0}(\mu_*)(\Lda_{\delta,1}) = \Theta^{n-2}(\mu_*,x_0)
\HH^{n-2}(B_1^{n-2})$ by \eqref{nu0con}, dividing by
$\HH^{n-2}(B_1^{n-2}) > 0$ gives
\[
\bigl|\Theta^{n-2}(\mu_*, x_0) - |\sg|_*\bigr| \leq C(\delta + \eta).
\]
Choosing $\delta = \eta$ and letting $\eta \to 0^+$, we conclude that
$\Theta^{n-2}(\mu_*, x_0) = |\sg|_*$, and hence
\[
\Theta^{n-2}(\mu_*, x_0) \in \{|\sg|_* : \sg \in [\Ss^1, \cN]\}.
\]
This completes the proof.
\end{proof}

\begin{proof}[Proof of Theorem \ref{maintheorem}]
The four properties of the theorem follow from the results established earlier in this paper. The first property is a consequence of Proposition \ref{proprectifiable}. The second follows from Propositions \ref{stationarity} and \ref{densitydescrete}. The third is given by Lemma \ref{CUleq}, and the fourth by
Proposition \ref{propombackSstar}.
\end{proof}

\section*{Acknowledgment}

Haotong Fu and Wei Wang are partially supported by the National Key R$\&$D Program of China under Grant 2023YFA1008801 and NSF of China under Grant 12288101. Giacomo Canevari is partially supported by INdAM--GNAMPA, project ref.~E53C25002010001.

\bibliographystyle{plain}

\end{document}